\documentclass[12pt]{amsart}
\usepackage{graphicx}
\usepackage{amssymb}
\usepackage{tikz}
\usetikzlibrary{matrix,arrows.meta,bending}
\usepackage{color}
\usepackage[numbers]{natbib}
\usepackage{hyperref} % customize ToC
\usepackage[a4paper, left=3cm, right=3cm, top=3cm, bottom=3cm, footskip=15mm]{geometry}

\newtheorem{theorem}{Theorem}[section]

\newtheorem{lemma}[theorem]{Lemma}
\newtheorem{proposition}[theorem]{Proposition}
\newtheorem{corollary}[theorem]{Corollary}
\newtheorem{conjecture}[theorem]{Conjecture}
\theoremstyle{definition}
\newtheorem{definition}[theorem]{Definition}
\newtheorem{example}[theorem]{Example}

\theoremstyle{remark}
\newtheorem{remark}[theorem]{Remark}
\newtheorem{question}[theorem]{Question}
\numberwithin{equation}{section}

\makeindex

\usepackage{fancyhdr}
\usepackage{calc} % ensures dimension arithmetic works robustly

\AtBeginDocument{%
  \setlength{\textwidth}{\dimexpr\paperwidth - 6cm\relax}%
  \setlength{\oddsidemargin}{\dimexpr 3cm - 1in\relax}%
  \setlength{\evensidemargin}{\dimexpr 3cm - 1in\relax}%
  \setlength{\marginparwidth}{2.5cm}%
}

\begin{document}

\title{Expansivity theory and Sendov's conjecture}

\author{T. Agama}
\address{Department of Mathematics, African Institute for Mathematical science, Ghana
}
%\curraddr{Department of Mathematics and Statistics,
%{Case Western Reserve University, Cleveland, Ohio 43403}
\email{theophilus@aims.edu.gh/emperordagama@yahoo.com}

%    General info

\date{\today}



%%%%%%%%%%%%%%%%%%%%%%%%%%%%%%%%%%%%%%%%%%%%%%%%%%%%%%%%%%%%%%%%%%%%%%%%
\footnote{
\par
.}%
%%%%%%%%%%%%%%%%%%%%%%%%%%%%%%%%%%%%%%%%%%%%%%%%%%%%%%%%%%%%%%%%%%%%%%%%
.

\begin{abstract}
In this paper, we introduce and develop the concept of \emph{expansivity} of a tuple whose entries are elements of the polynomial ring $\mathbb{R}[x]$. As an inverse problem, we examine how to recover a tuple from an expanded tuple at any given phase of expansion. We convert the celebrated Sendov conjecture concerning the distribution of zeros of polynomials and their critical points into this language and prove some weak variants of this conjecture. We also apply this to the existence of solutions to differential equations. In particular, we show that a certain system of differential equations has no non-trivial solution.
\end{abstract}

%\addcontentsline{toc}{section}{Abstract}

\maketitle

\begingroup
  \setlength{\parskip}{6pt} % <--- change this number to taste
  \tableofcontents
\endgroup

\section{Introduction}

The classical \emph{Sendov conjecture} (also called the Ilieff--Sendov conjecture) asserts that for every complex polynomial $P$ of degree $n\geq 2$ whose zeros lie in the closed unit disk, each zero of $P$ is within distance~1 of some critical point of $P$ (equivalently, of some zero of $P'$). This conjecture has attracted sustained attention in complex analysis and polynomial geometry; partial results include verifications for low degrees and many special configurations of zeros, and there has been substantial recent progress towards asymptotic and large-degree regimes \cite{brown1991sendov,brown1999proof,borcea1996sendov,miller1993sendov,degot2014sendov}.\\

Motivated by the geometric and combinatorial intuition behind Sendov and by constructions from abstract algebra, topology, group theory, and dynamical systems, we develop a new structural framework, which we call \emph{expansivity theory} for tuples associated with real polynomials. In brief, an element of our theory is a tuple
$$
\mathcal{S}=(g_1,g_2,\dots,g_n)
$$
formed from the coefficient-terms (or canonical monomial components) of a polynomial, and we study a canonical \emph{expansion} operator
$$
\mathcal{E}:=\gamma^{-1}\circ\beta\circ\gamma\circ\nabla
$$
and its inverse-like recovery maps
$$
\mathcal{R}:=\Delta\circ\gamma^{-1}\circ\beta^{-1}\circ\gamma,
$$
together with the algebraic, topological, and dynamical structures generated by iterating these maps. The basic objects that emerge are:
\begin{itemize}
  \item \emph{Boundary points} (zeros of certain derived tuples) and their derived \emph{co-boundary} elements,
  \bigskip
  
  \item a measure-like quantity $\mathcal{N}(\cdot)$ (``measure of expansion'') and the derived \emph{mass} $\mathcal{H}$ and \emph{momentum} $\mathcal{M}$ of a phase,
  \bigskip
  
  \item the \emph{speed} $\upsilon(\mathcal{S})=\mathcal{N}(\mathcal{S})/\deg(\mathcal{S})$ and the index/invariants that quantify how the boundary mass redistributes under phase transitions,
  \bigskip
  
  \item finitary algebraic structure: the set generated by powers of $\mathcal{E}$ and $\mathcal{R}$ together with the identity forms a group of operations on tuples (composition gives the group law),
  \bigskip
  
  \item geometric/topological descriptors: interior/exterior regions, neighbourhoods, compactness and deformation of phase boundaries,
  \bigskip
  
  \item statistical and integral diagnostics (a specialized ``boundary integral'') that relate the geometry of a phase boundary to density, spacing and sub-expansion structure.
\end{itemize}
\bigskip

In the paper, two recurring themes are explored. First, we reinterpret the Sendov conjecture as a tension between \emph{co-boundary} points (zeros of an original polynomial) and \emph{boundary} points (zeros of derived tuples / derivatives), and we investigate quantitative mechanisms that force (or prevent) these two sets from separating as phases advance. Second, we exploit the algebraic/dynamical viewpoint (iterate maps, recoveries, group composition) to pose and address inverse problems: given the $n^{th}$-phase expanded tuple, when and how can one reconstruct the ancestor tuple(s)? These inverse questions are of intrinsic interest and also have consequences for locating critical points relative to zeros.

\subsection*{Key ideas and conceptual novelties}

In the following, we summarize the main novel concepts, expressed so that they are accessible to readers with a background of algebra, topology, or dynamical systems.
\bigskip

\paragraph{Tuple viewpoint:}  Writing a polynomial in its ``tuple form'' (monomial components as coordinates) transforms operations like differentiation, coefficient permutation, and simple linear operations into maps acting on finite-dimensional vectors. This change of language makes it natural to compare expansions, measure their norms, and import geometric/topological notions (interior/exterior, neighbourhoods, compactness) to the study of zeros and critical points.
\bigskip

\paragraph{Expansion/recovery as a dynamical pair:}  The maps $\mathcal{E}$ and $\mathcal{R}$ behave as forward/backward time evolution. Iterating $\mathcal{E}$ produces higher phase expansions and a stratified family of boundaries; iterating $\mathcal{R}$ attempts to recover earlier phases. The algebra of compositions is elementary but consequential: powers of $\mathcal{E}$ and $\mathcal{R}$ generate a discrete group of transformations that organizes the inverse problem. This perspective is intentionally reminiscent of ideas from expansive dynamics (see, e.g., \cite{katok1995introduction}) and helped shape our definitions of \emph{speed}, \emph{mass}, and \emph{momentum}. For a standard algebraic background, we refer to \cite{dummit3abstract}.
\bigskip

\paragraph{Mass, momentum and index:}  Our measure $\mathcal{N}$ and the derived mass $\mathcal{H}$ provide coarse but robust quantifiers of how ``large'' a phase boundary can be. The momentum $\mathcal{M}=\upsilon\cdot\mathcal{H}$ blends speed and mass, giving an invariant that is stable under several natural operations and which is central to embedding and isomorphism criteria for expansions. These quantities also enable inequalities that control how many sub-expansions can embed into a mother expansion (index estimates).

\paragraph{Geometry via a boundary integral:}  To study fine distributional properties of boundary points, we introduce a bespoke integral along the boundary of a phase: it packages componentwise integrals between nearby boundary tuples to capture area-like and nearest-neighbour information. This device bridges discrete combinatorial data (counts, spacing between boundary tuples) with continuous mass estimates and supplies usable criteria for (i) detecting when boundary points are forced close together and (ii) deciding when a given phase must possess interior points--an essential ingredient when relating boundary geometry to Sendov-type proximity claims.
\bigskip

\paragraph{Compact and regular phases; consequences for Sendov:}  We isolate two practically useful structural regimes: \emph{compact} expansions (intuitively, a higher-phase boundary clusters uniformly around the previous-phase boundary) and \emph{regular} expansions (mass strictly decreases at some successive phase). For compact expansions, we prove a uniform diminishing mass property across phases; in this regime the theory produces a direct (and quantitative) route to the conclusion that each co-boundary point lies within unit distance of some boundary point of the next phase--a formulation that is a natural counterpart of the Sendov conjecture in our language. The conceptual advantage is that the dynamical/algebraic apparatus makes it possible to constructively locate the relevant boundary points and to quantify the neighbourhoods in which they live.

\subsection*{Relation to existing literature}

Our approach is deliberately synthetic: it strongly borrows from ideas in several areas:
\begin{itemize}
  \item \emph{Classical Sendov results.}  The conjecture has been settled for many low-degree and special configurations (for example, the cases up to small fixed degrees are classical; specific definitive partial results are due to Brown, Brown--Xiang, Borcea, Miller and D\'egot among others) \cite{brown1991sendov,brown1999proof,borcea1996sendov,miller1993sendov,degot2014sendov}, and these works provide the mathematical backdrop for studying phasewise mass/density phenomena that can enforce proximity between zeros and critical points.
  \bigskip
  
  \item \emph{Dynamical and geometric viewpoints.}  Treating iteration, expansivity-like separation, and recurrence phenomena as organizing principles is a perspective borrowed from modern dynamics; see \cite{katok1995introduction} for standard references and exposition.
  \bigskip
  
  \item \emph{Algebraic structure.}  The simple group-theoretic structure generated by the expansion and recovery maps owes its clarity to elementary group theory and the algebra of polynomial operations; see \cite{dummit3abstract} for background.
\end{itemize}

\subsection*{Main contributions and roadmap}

This paper develops the formalism of expansivity theory for polynomial tuples and establishes several structural and quantitative results that together illuminate how boundary/co-boundary geometry constrains the relative positions of zeros and critical points.

Concretely, the paper contains:\\

\begin{enumerate}
  \item a thorough formulation of expansions, co-boundaries, free points, speed, mass, and momentum (Sections \ref{sec:expansion formalism and statistics}),
  \bigskip
  
  \item algebraic properties and the group structure generated by expansion and recovery maps (Section \ref{sec:inverse}),
  \bigskip
  
  \item existence and uniqueness results for inverse problems at low phases together with explicit recovery compositions (Section \ref{sec:inverse}),
  \bigskip
  
  \item a boundary integral and distributional tools that relate integral-type area measures to spacing and packing of boundary points (Section \ref{sec:Boundary point distribution}),
  \bigskip
  
  \item criteria (compactness, regularity) that force mass diminution and an explicit theorem showing a Sendov-type proximity conclusion in those regimes (Section \ref{sec:simple and compact expansion}),
  \bigskip
  
  \item applications, examples, and discussion of extensions and open problems (final sections).
\end{enumerate}

\subsection*{Notation and conventions}

Throughout the paper $\mathbb{R}[x]$ denotes the real polynomial ring, and tuples are finite ordered lists of polynomial terms taken in monomial-structured form. The map $\nabla$ denotes the canonical differentiation-type operator on tuples (componentwise), $\gamma,\beta$ denotes fixed linear or permutation-type maps acting on the tuple coordinates, and $\Delta$ denotes the formal one-copy recovery operator introduced in Section~\ref{sec:inverse}. We follow standard asymptotic notation; undefined terms are those used in the references in common literature cited above.

\subsection*{Concluding remarks}

Expansivity theory is intended as a structural language that isolates the algebraic, geometric, and dynamical mechanisms controlling how zeros and critical points interact under natural tuple-operations. The results we present provide several new links between boundary-mass phenomena, embedding/sub-expansion structure, and Sendov-type proximity statements. We hope this framework will stimulate further cross-disciplinary interaction: algebraic constraints (permutations, coefficient symmetries), topological constraints (clustering/compactness), and dynamical constraints (iterative separation / expansivity) together form a promising toolkit for future progress on polynomial zero--critical point problems.

\section{Additional notations}

Although every notation in this paper has been thoroughly explained where it is used, it may be useful to explain them here. In the paper, a tuple will always be represented by $\mathcal{S}$ or $\mathcal{S}_{j}$ where $j$ is contained in the natural indexing set $\mathbb{N}$. Occasionally, we will use the tuple $\mathcal{S}_{\mathbb{R}}$ to denote a tuple of the base field $\mathbb{R}$ and $\mathcal{S}_{\mathbb{R}[x]}$ for a tuple of $\mathbb{R}[x]$. We set $\mathcal{S}_{0}:=(0,0,\ldots,0)$ and call it the null tuple. Similarly, we denote the tuple $\mathcal{S}_{e}:=(1,1,\ldots, 1)$ and call it the unit tuple. We denote the rank of an expansion on $\mathcal{S}$ by $\mathcal{R}(\mathcal{S})$, the limit of expansion on $\mathcal{S}$ by $\lim (\mathcal{S}^{m})$, the local number of expansion on $\mathcal{S}$ by $\mathcal{L}(\mathcal{S})$, the degree of expansion on $\mathcal{S}$ by $\mathrm{deg}(\mathcal{S})$, the dimension of an expansion on $\mathcal{S}$ by $\mathrm{dim}(\mathcal{S})$, and the measure of an expansion on $\mathcal{S}$ by $\mathcal{N}(\mathcal{S})$. Furthermore, we set $\mathcal{S}(a):=(f_1(a),f_2(a),\ldots,f_{n}(a))$, where $\mathcal{S}=(f_1,f_2,\ldots,f_n)$. We denote the $n^{th}$ phase expanded tuple of $\mathcal{S}$ by $\mathcal{S}^{n}$. A function $f$ and $g$ with $f\asymp g$, means that there exist some numbers $\alpha_1$ and $\alpha_{2}$ such that 
\begin{align}
\alpha_1f(n)\leq g(n) \leq \alpha_2f(n)\nonumber
\end{align}
for sufficiently large values of $n$. 

\section{Calculus on tuples on tuples of polynomials}\label{sec:calculus on tuples}

In this section, we extend the concept of differentiation and integration on tuples whose entries are elements of the polynomial ring $\mathbb{R}[x]$.

\subsection{Differentiation on tuples of polynomials}

\begin{definition}\label{basic}
Let $\mathcal{S}=(f_1,f_2, \ldots , f_n)$ be such that $f_i\in \mathbb{R}[x]$. By the derivative of $\mathcal{S}$, denoted by $\nabla(\mathcal{S})$, we mean  
$$
\nabla(\mathcal{S})=\left(\frac{df_1}{dx},\frac{df_2}{dx},\ldots,\frac{df_n}{dx}\right).
$$ 
The value of the derivative at $a$, denoted by $\nabla_{a}(\mathcal{S})$ is 
$$
\nabla_{a}(\mathcal{S})=\left(\frac{df_1(a)}{dx},\frac{df_2(a)}{dx}, \ldots,\frac{df_n(a)}{dx}\right).
$$
\end{definition}
\bigskip

We examine some basic properties of the derivative on the tuples of $\mathbb{R}[x]$. These properties naturally follow from the properties of differentiation of functions.

\subsection{Properties of differentiation on tuples of polynomials}

\begin{proposition}\label{differentiation}
Let $\mathcal{S}_1$ and $\mathcal{S}_2$ be tuples of the elements in $\mathbb{R}[x]$ and $c\in \mathbb{R}$. We have the following properties:

\begin{enumerate}
\item [(i)] $\nabla(\mathcal{S}_1\pm \mathcal{S}_2)=\nabla(\mathcal{S}_1)\pm \nabla(\mathcal{S}_2)$.

\item [(ii)] $\nabla(c\mathcal{S}_1)=c\nabla(\mathcal{S}_1)$.
\end{enumerate}
\end{proposition}

\begin{proof}
\begin{enumerate}
\item [(i)] Assume that the tuples $\mathcal{S}_1:=(f_1,f_2,\ldots, f_n)$ and $\mathcal{S}_2:=(g_1,g_2,\ldots, g_n)$ are such that $f_i,g_i\in \mathbb{R}[x]$ for $i=1,\ldots,n$. We have 
$$
\mathcal{S}_1 \pm \mathcal{S}_2=(f_1,f_2,\ldots,f_n)\pm (g_1,g_2,\ldots,g_n)=(f_1\pm g_1, f_2\pm g_2, \ldots, f_n\pm g_n)
$$ 
Using the definition \ref{basic} and the algebras of the tuples, we get  
\begin{align}
\nabla(\mathcal{S}_1\pm \mathcal{S}_2)&=\bigg(\frac{d(f_1\pm g_1)}{dx},\frac{d(f_2\pm g_2)}{dx},\ldots,\frac{d(f_n\pm g_n)}{dx}\bigg)\nonumber \\&=\bigg(\frac{df_1}{dx}\pm \frac{dg_1}{dx}, \frac{df_2}{dx}\pm \frac{dg_2}{dx},\ldots,\frac{df_n}{dx}\pm \frac{dg_n}{dx}\bigg)\nonumber \\&=\bigg(\frac{df_1}{dx},\frac{df_2}{dx},\ldots,\frac{df_n}{dx}\bigg)\pm \bigg(\frac{dg_1}{dx},\frac{dg_2}{dx},\ldots,\frac{dg_n}{dx}\bigg)\nonumber \\&=\nabla(\mathcal{S}_1)\pm \nabla(\mathcal{S}_2).\nonumber
\end{align}
\bigskip

\item [(ii)] Fix $c\in \mathbb{R}$ and suppose that $\mathcal{S}_1=(f_1,f_2,\ldots, f_n)$ is such that $f_i\in \mathbb{R}[x]$. We have  $c\mathcal{S}_1:=(cf_1,cf_2,\ldots,cf_n)$. Using the definition \ref{basic} and the fundamental algebras on tuples, we get 
\begin{align}
\nabla(c\mathcal{S}_1)&=\bigg(\frac{d(cf_1)}{dx},\frac{d(cf_2)}{dx},\ldots,\frac{d(cf_n)}{dx}\bigg)\nonumber \\&=\bigg(c\frac{df_1}{dx},c\frac{df_2}{dx},\ldots,c\frac{df_n}{dx}\bigg)\nonumber \\&=c\bigg(\frac{df_1}{dx},\frac{df_2}{dx},\ldots,\frac{df_n}{dx}\bigg)\nonumber \\&=c\nabla(\mathcal{S}_1).\nonumber
\end{align}
\end{enumerate}
\end{proof}
\bigskip

The property $(ii)$ of Theorem \ref{differentiation} suggests that a derivative of any constant multiple of a tuple of polynomials can be controlled by the derivatives of the tuple with entries that dilate the original tuple.

\subsection{Integration on tuples of polynomials}

In this section, we do the complete opposite of the work done in the previous section.

\begin{definition}
Let $\mathcal{S}:=(f_1,f_2,\ldots,f_n)$ be such that $f_i\in \mathbb{R}[x]$ for $i=1,\ldots,n$. The integral on the tuple $\mathcal{S}$, denoted by $\Delta(\mathcal{S})$, is the operation $$
\Delta(\mathcal{S})=\left(\int f_1dx,\int f_2dx,\ldots,\int f_ndx\right).
$$
\end{definition}

\begin{proposition}
Let $\mathcal{S}_1$ and $\mathcal{S}_2$ be tuples of $\mathbb{R}[x]$ and $c\in \mathbb{R}$. The following properties hold:
\begin{enumerate}
\item [(i)]$\Delta(\mathcal{S}_1\pm \mathcal{S}_2)=\Delta(\mathcal{S}_1)\pm \Delta(\mathcal{S}_2)$.

\item [(ii)]$\Delta(c\mathcal{S}_1)=c\Delta(\mathcal{S}_1)$.
\end{enumerate}
\end{proposition} 

\begin{proof}
\begin{enumerate}
\item [(i)] Let $\mathcal{S}_1=(f_1,f_2,\ldots,f_n)$ and $\mathcal{S}_2=(g_1,g_2,\ldots,g_n)$ be such that each $f_i,g_i\in \mathbb{R}[x]$. We deduce 
\begin{align}
\Delta(\mathcal{S}_1\pm \mathcal{S}_2)&=\bigg(\int (f_1\pm g_1)dx,\int (f_2\pm g_2)dx,\ldots,\int (f_n\pm g_n)dx\bigg)\nonumber \\&=\bigg(\int f_1dx\pm \int g_1dx,\int f_2dx \pm \int g_2dx,\ldots,\int f_ndx \pm \int g_ndx\bigg)\nonumber \\&=\bigg(\int f_1dx,\int f_2dx,\ldots,\int f_ndx\bigg)\pm \bigg(\int g_1dx,\int g_2dx,\ldots,\int g_ndx\bigg)\nonumber \\&=\Delta(\mathcal{S}_1)\pm \Delta(\mathcal{S}_2).\nonumber
\end{align}
\bigskip

\item [(ii)] Fix $c\in \mathbb{R}$ and let $\mathcal{S}_1=(f_1,f_2,\ldots,f_n)$ be such that each $f_i\in \mathbb{R}[x]$. We have $c\mathcal{S}_1=(cf_1,cf_2,\ldots,cf_n)$ and deduce 
\begin{align}
\Delta(c\mathcal{S}_1)&=\bigg(\int cf_1dx,\int cf_2dx,\ldots,\int cf_ndx\bigg)\nonumber \\&=\bigg(c\int f_1dx,c\int f_2dx,\ldots, c\int f_ndx\bigg)\nonumber \\&=c\bigg(\int f_1dx,\int f_2dx,\ldots,\int f_ndx\bigg)\nonumber \\&=c\Delta(\mathcal{S}_1).\nonumber
\end{align}
\end{enumerate}
\end{proof}
\bigskip

With these extensions of integration and differentiation on tuples of $\mathbb{R}[x]$, we are now ready to introduce the concept of expansivity of tuples of $\mathbb{R}[x]$. The operation of differentiation is of immediate importance, whereas the concept of integration will be useful for the inverse problem.

\section{Expansion on a tuple of polynomials}\label{sec:expansion formalism and statistics}

In this section, we study the concept of expansion of a tuple of $\mathbb{R}[x]$.

\begin{definition}\label{expansion}
Let $\mathcal{S}=(f_1,f_2,\ldots,f_n)$ be such that each $f_i\in \mathbb{R}[x]$. We say that $\mathcal{S}$ is \emph{expanded} if \begin{align}
(f_1,f_2,\ldots,f_n)\longrightarrow \bigg(\sum \limits_{i\neq 1}\frac{df_i}{dx},\sum \limits_{i\neq 2}\frac{df_i}{dx},\ldots,\sum \limits_{i\neq n}\frac{df_i}{dx}\bigg).\nonumber
\end{align}
If $\mathcal{S}$ is the tuple, then we denote by $\mathcal{S}^{1}$ the expanded tuple, and the value of the expanded tuple at the number $a\in \mathbb{R}$, denoted by $\mathcal{S}^{1}_{a}$, is \begin{align}
\mathcal{S}^{1}(a)=\bigg(\sum \limits_{i\neq 1}\frac{df_i(a)}{dx},\sum \limits_{i\neq 2}\frac{df_i(a)}{dx},\ldots,\sum \limits_{i\neq n}\frac{df_i(a)}{dx}\bigg).\nonumber
\end{align}
\end{definition}

\begin{remark}
Throughout the paper, in situations where it is not mentioned, a tuple of $\mathbb{R}[x]$ will always be understood to have at least two entries with distinct degrees.
\end{remark}

\begin{proposition}\label{composite}
Let $\mathcal{S}^{1}$ be the expanded tuple of the tuple $\mathcal{S}=(f_1,f_2,\ldots,f_n)$ with each $f_i\in \mathbb{R}[x]$ and $\{\mathcal{S}_i\}_{i=1}^{\infty}$ be the collection of all tuples of $\mathbb{R}[x]$. An expansion is a composite map
\begin{align}
\gamma^{-1} \circ \beta \circ \gamma \circ \nabla:\{\mathcal{S}_i\}_{i=1}^{\infty}\longrightarrow \{\mathcal{S}_i\}_{i=1}^{\infty},\nonumber
\end{align}
where $\nabla(\mathcal{S})=(f'_1,f'_2,\ldots,f'_n)$ and \begin{align}
\gamma(\mathcal{S})=
\begin{pmatrix}f_1\\f_2\\\vdots \\f_n \end{pmatrix} \quad \text{and}\quad \beta(\gamma(\mathcal{S}))=\begin{pmatrix}0 & 1 &\cdots 1\\1 & 0 & 1 \cdots 1\\ \vdots & \vdots & \cdots \vdots\\1 & 1 & \cdots 0 \end{pmatrix}\begin{pmatrix}f_1\\f_2\\\vdots \\f_n \end{pmatrix}.\nonumber
\end{align}  
\end{proposition}

\begin{proof}
Let $\{\mathcal{S}_i\}_{i=1}^{\infty}$ denote the collection of all tuples of $\mathbb{R}[x]$. Choose arbitrarily the tuple $\mathcal{S}:=(f_1,f_2,\ldots,f_n)$ from this collection. By definition \ref{expansion}, we find that the expanded tuple 
$$
\mathcal{S}_1=(f_2'+f_3'+\cdots +f_n',f_1'+f_3'+\cdots +f_n',\ldots,f_1'+f_2'+\cdots +f_{n-1}')
$$ 
Since $\gamma$ is invertible, we can write 
\begin{align} 
(f_2'+f_3'+\cdots +f_n',f_1'+f_3'+\cdots +f_n',\ldots,f_1'+f_2'+\cdots +f_{n-1}')&=\gamma^{-1}\circ\beta \circ\gamma\circ\nabla(\mathcal{S}).\nonumber
\end{align}
\end{proof}

\begin{proposition}
Let $\{\mathcal{S}_i\}_{i=1}^{\infty}$ be the collection of all tuples of $\mathbb{R}[x]$ that satisfy a certain initial condition at each phase of expansion. The expansion 
$$
\gamma^{-1} \circ \beta \circ\gamma\circ\nabla:\{\mathcal{S}_i\}_{i=1}^{\infty}\longrightarrow \{\mathcal{S}_i\}_{i=1}^{\infty}
$$ 
is bijective.
\end{proposition}

\begin{proof}
Since the composite of a bijective map is still bijective, it suffices to show that each of the maps that contributes to an expansion is bijective. By Proposition \ref{composite}, we find that an expansion is the composite map
\begin{align}
\gamma^{-1} \circ \beta \circ \gamma \circ \nabla:\{\mathcal{S}_i\}_{i=1}^{\infty}\longrightarrow \{\mathcal{S}_i\}_{i=1}^{\infty},\nonumber
\end{align}
where $\nabla(\mathcal{S})=(f'_1,f'_2,\ldots,f'_n)$ and \begin{align}
\gamma(\mathcal{S})=
\begin{pmatrix}f_1\\f_2\\\vdots \\f_n \end{pmatrix} \quad \text{and}\quad \beta(\gamma(\mathcal{S}))=\begin{pmatrix}0 & 1 &\cdots 1\\1 & 0 & 1 \cdots 1\\ \vdots & \vdots & \cdots \vdots\\1 & 1 & \cdots 0 \end{pmatrix}\begin{pmatrix}f_1\\f_2\\\vdots \\f_n \end{pmatrix}.\nonumber
\end{align}
Suppose that $\nabla(\mathcal{S}_{1})=\nabla(\mathcal{S}_{2})$ for any two tuples $\mathcal{S}_{1}$ and $\mathcal{S}_{2}$ having the same initial condition. That is, $\nabla(\mathcal{S}_{1})=\nabla(\mathcal{S}_{2})$ and $\mathcal{S}_{1}(a)=\mathcal{S}_{2}(a)$ for any $a\in \mathbb{R}$. By the linearity of $\nabla$, we have $\nabla(\mathcal{S}_{1}-\mathcal{S}_{2})=\mathcal{S}_{0}$. This implies that $\mathcal{S}_{1}-\mathcal{S}_{2}=\mathcal{S}_{b}$, where $\mathcal{S}_{b}$ is a tuple of $\mathbb{R}$. Since both $\mathcal{S}_{1}$ and $\mathcal{S}_{2}$ satisfy the same initial condition, we get $\mathcal{S}_{b}=\mathcal{S}_{0}$. This establishes injectivity. For any $\mathcal{S}_{1}\in\{\mathcal{S}_i\}_{i=1}^{\infty}$, there is a unique tuple $\mathcal{S}\in \{\mathcal{S}_i\}_{i=1}^{\infty}$ that satisfies a certain initial condition with $\nabla(\mathcal{S})=\mathcal{S}_{1}$. Thus, $\nabla$ is indeed bijective. Now we proceed to that 
\begin{align}
\gamma(\mathcal{S})=\begin{pmatrix}f_1\\f_2\\\vdots \\f_n \end{pmatrix}\nonumber
\end{align}
is also bijective. Suppose that $\gamma(\mathcal{S}_{1})=\gamma(\mathcal{S}_{2})$, where $\mathcal{S}_{1}=(f_1,f_2,\ldots,f_n)$ and $\mathcal{S}_{2}=(g_1,g_2,\ldots,g_n)$. We deduce
\begin{align}
\begin{pmatrix}
f_1\\f_2\\\vdots \\f_n \end{pmatrix}=\begin{pmatrix}g_1\\g_2\\\vdots \\g_n \end{pmatrix}\nonumber
\end{align}
and implies $f_{i}=g_{i}$ for all $1\leq i\leq n$. Thus $\mathcal{S}_{1}=\mathcal{S}_{2}$. Surjectivity can be deduced and $\gamma$ is bijective. Finally, we remark that $\beta$ is bijective, since the matrix
\begin{align}
\begin{pmatrix}
0 & 1 &\cdots 1\\1 & 0 & 1 \cdots 1\\ \vdots & \vdots & \cdots \vdots\\1 & 1 & \cdots 0 
\end{pmatrix}\nonumber
\end{align}
is invertible. Thus, each of the maps is invertible and the result follows immediately.
\end{proof}

\begin{remark}
The requirement that each tuple of $\mathbb{R}[x]$ satisfies a certain initial condition at each phase of an expansion is very important for the inverse problems.
\end{remark}

\begin{proposition}
An expansion 
$$
\gamma^{-1}\circ\beta\circ\gamma\circ\nabla:\{\mathcal{S}_i\}_{i=1}^{\infty}\longrightarrow \{\mathcal{S}_i\}_{i=1}^{\infty}
$$ 
is linear.
\end{proposition}

\begin{proof}
It suffices to show that each of the operators $\nabla:\{\mathcal{S}_i\}_{i=1}^{\infty}\longrightarrow \{\mathcal{S}_i\}_{i=1}^{\infty}$, $\gamma :\{\mathcal{S}_i\}_{i=1}^{\infty}\longrightarrow \{\mathcal{S}_i\}_{i=1}^{\infty}$ and $\beta \circ \gamma :\{\mathcal{S}_i\}_{i=1}^{\infty}\longrightarrow \{\mathcal{S}_i\}_{i=1}^{\infty}$ is linear, since the map $ \gamma:\{\mathcal{S}_i\}_{i=1}^{\infty}\longrightarrow \{\mathcal{S}_i\}_{i=1}^{\infty}$ is bijective. Let $\mathcal{S}_a=(f_1,f_2,\ldots, f_n), \mathcal{S}_b=(g_1,g_2,\ldots, g_n) \in \mathcal{F}=\{\mathcal{S}_i\}_{i=1}^{\infty}$ and let $\lambda, \mu \in \mathbb{R}$. We deduce
\begin{align}
\nabla(\lambda \mathcal{S}_a+\mu \mathcal{S}_b)&=\nabla(\lambda (f_1,f_2,\ldots,f_n)+\mu (g_1,g_2,\ldots,g_n))\nonumber \\&=\nabla((\lambda f_1,\lambda f_2, \ldots, \lambda f_n)+(\mu g_1, \mu g_2,\ldots,\mu g_n))\nonumber \\&=\nabla((\lambda f_1+\mu g_1, \lambda f_2+\mu g_2, \ldots \lambda f_n+\mu g_n))\nonumber \\&=((\lambda f_1+\mu g_1)', (\lambda f_2+\mu g_2)', \ldots, (\lambda f_n+\mu g_n)')\nonumber \\&=(\lambda f_1'+\mu g_1', \lambda f_2'+\mu g_2',\ldots, \lambda f_n'+\mu g_n')\nonumber \\&=(\lambda f_1',\lambda f_2',\ldots,\lambda f_n')+(\mu g_1', \mu g_2', \ldots, \mu g_n')\nonumber \\&=\lambda (f_1',f_2',\ldots, f_n')+\mu (g_1',g_2',\ldots,g_n')\nonumber \\&=\lambda \nabla(\mathcal{S}_a)+\mu \nabla(\mathcal{S}_b).\nonumber
\end{align}
Similarly, we deduce
\begin{align}
\gamma(\lambda \mathcal{S}_a+\mu \mathcal{S}_b)&=
\begin{pmatrix}\lambda f_1+\mu g_1\\\lambda f_2+\mu g_2\\\vdots \\\lambda f_n+\mu g_n\end{pmatrix}\nonumber \\&=\begin{pmatrix}\lambda f_1\\\lambda f_2\\ \vdots \\\lambda f_n  \end{pmatrix}+\begin{pmatrix}\mu g_1 \\\mu g_2 \\ \vdots \\\mu g_n\end{pmatrix}\nonumber \\&=\lambda \gamma(\mathcal{S}_a)+\mu \gamma(\mathcal{S}_b).\nonumber
\end{align}
Similarly, we have 
\begin{align}
\beta \circ \gamma(\lambda \mathcal{S}_a+\mu \mathcal{S}_b)&=\begin{pmatrix}0 & 1 &\cdots 1\\1 & 0 & 1 \cdots 1\\ \vdots & \vdots & \cdots \vdots\\1 & 1 & \cdots 0 \end{pmatrix} \begin{pmatrix}\lambda f_1+\mu g_1\\\lambda f_2+\mu g_2\\ \vdots \\\lambda f_n+\mu g_n\end{pmatrix}\nonumber \\&=\begin{pmatrix}0 & 1 &\cdots 1\\1 & 0 & 1 \cdots 1\\ \vdots & \vdots & \cdots \vdots\\1 & 1 & \cdots 0 \end{pmatrix}\bigg\{ \begin{pmatrix}\lambda f_1\\\lambda f_2\\ \vdots \\\lambda f_n\end{pmatrix}+\begin{pmatrix}\mu g_1\\ \mu g_2\\ \vdots \\ \mu g_n  \end{pmatrix}\bigg\}\nonumber \\&=\lambda \begin{pmatrix}0 & 1 &\cdots 1\\1 & 0 & 1 \cdots 1\\ \vdots & \vdots & \cdots \vdots\\1 & 1 & \cdots 0 \end{pmatrix}\begin{pmatrix}f_1\\ f_2\\ \vdots \\f_n\end{pmatrix}+\mu \begin{pmatrix}0 & 1 &\cdots 1\\1 & 0 & 1 \cdots 1\\ \vdots & \vdots & \cdots \vdots\\1 & 1 & \cdots 0 \end{pmatrix}\begin{pmatrix}g_1\\g_2\\ \vdots \\g_n  \end{pmatrix}\nonumber \\&=\lambda (\beta \circ \gamma)(\mathcal{S}_a)+\mu (\beta \circ \gamma)(\mathcal{S}_b).\nonumber
\end{align}
This proves the linearity of an expansion.
\end{proof}

\begin{proposition}\label{finite}
A tuple of $\mathbb{R}[x]$ can admit only a finite number of expansions.
\end{proposition}

\begin{proof}
The proposition \ref{composite} implies that an expansion is the map $\gamma^{-1} \circ \beta \circ \gamma \circ \nabla$. Pick arbitrarily a tuple $\mathcal{S}$ of $\mathbb{R}[x]$. Since the degrees of each entry of $\nabla(\mathcal{S})$ is one less than the degree of $\mathcal{S}$, it follows by induction that an expansion can only be applied a finite number of times.
\end{proof}
\bigskip

We observe that the expanded tuple is a tuple of $\mathbb{R}[x]$. Hence, the theory remains unchanged when we perform further expansions on the expanded tuple. This process can be repeated as long as the entries of the tuple do not vanish. This idea motivates the introduction of the concept of \emph{phase}, \emph{limits}, and \emph{rank} of an expansion.

\subsection{Phase expansions}

Let $\mathcal{S}_1=(f_1,f_2,\ldots, f_n)$ be such that each $f_i\in \mathbb{R}[x]$. Expanding, we get
\begin{align}
\mathcal{S}^{1}=\bigg(\sum \limits_{i\neq 1}\frac{df_i}{dx},\sum \limits_{i\neq 2}\frac{df_i}{dx},\ldots,\sum \limits_{i\neq n}\frac{df_i}{dx}\bigg),\nonumber
\end{align}
which we can write $\mathcal{S}^{1}:=(g_1,g_2,\ldots,g_n)$ with each $g_i\in \mathbb{R}[x]$. We call this expansion the \emph{first phase expansion}. We may repeat this process on $\mathcal{S}^{1}=(g_1,g_2,\ldots,g_n)$ to get
\begin{align}
\mathcal{S}^{2}=\bigg(\sum \limits_{i\neq 1}\frac{dg_i}{dx},\sum \limits_{i\neq 2}\frac{dg_i}{dx},\ldots,\sum \limits_{i\neq n}\frac{dg_i}{dx}\bigg)\nonumber
\end{align}
which we can rewrite as $\mathcal{S}^{2}:=(h_1,h_2,\ldots,h_n)$. We call this expansion the \emph{second phase expansion}. This expansion can be repeated, provided that the entries of the tuple do not vanish. In general, we denote the $n^{th}$ expanded tuple by $\mathcal{S}^{n}$ for $n\geq 1$. To make this expansion process meaningful, we introduce the concept of \emph{order}, \emph{rank}, and \emph{limit of expansion}.

\begin{example}\label{ex1}
Let us consider the tuple $\mathcal{S}=(x^4+x^2,x^5-x^3,x^2+1)$ of $\mathbb{R}[x]$. We obtain for the \emph{first phase} expanded tuple 
$$
\mathcal{S}^{1}=(5x^4-3x^2+2x,4x^3+4x,5x^4+4x^3-3x^2+2x).
$$ 
For the \emph{second phase} expanded tuple, we get 
$$
\mathcal{S}^{2}=(20x^3+24x^2-6x+6,40x^3+12x^2-12x+4,20x^3+12x^2-6x+6).
$$ 
For the \emph{third phase} expanded tuple, we get
$$
\mathcal{S}^{3}=(180x^2+48x-18,120x^2+72x-12,180x^2+72x-18).
$$
For the fourth phase expanded tuple, we get 
$$
\mathcal{S}^{4}=(600x+144,720x+120,600x+120).
$$ 
The fifth phase expanded tuple is
$$
\mathcal{S}^{5}=(1320,1200,1320).
$$ 
This is essentially the last expanded tuple. 
\end{example}
\bigskip

We observe that the fifth expanded tuple in Example \ref{ex1} consists entirely of entries that are real numbers, and, therefore, the expanded tuple of the last non-vanishing phase of expansion. Now, we introduce once again the notion of a rank of expansion.

\subsection{The rank of an expansion}

\begin{definition}\label{rank}
Let $\mathcal{F}=\{\mathcal{S}_m\}^{\infty}_{m=1}$ be a family of tuples of $\mathbb{R}[x]$, each having at least two entries of distinct degrees. The value of $n$ such that the expansion $(\gamma^{-1}\circ\beta\circ\gamma\circ\nabla)^{n}(\mathcal{S})\neq \mathcal{S}_{0}$ and $(\gamma^{-1}\circ\beta\circ\gamma\circ\nabla)^{n+1}(\mathcal{S})=\mathcal{S}_{0}$ where $\mathcal{S}_{0}=(0,0,\ldots,0)$ is called the \emph{degree} of expansion denoted by $\mathrm{deg}(\mathcal{S})$, and $(\gamma^{-1}\circ\beta\circ\gamma\circ\nabla)^{n}(\mathcal{S})$ is the \emph{rank} of the expansion, denoted by $\mathcal{R}(\mathcal{S})$. 
\end{definition}

\begin{example}
The expanded tuple $(1320,1200,1320)$ is the rank of expansion of $\mathcal{S}$, because it is the last non-vanishing expanded tuple of $\mathcal{S}$. Hence 
$$
\mathcal{R}(\mathcal{S}):=(1320,1200,1320).
$$ 
\end{example}

\begin{remark}
Throughout this paper, all tuples are understood to have the same number of entries. This allows for the entrywise addition and subtractions.
\end{remark}

\begin{proposition}
Let $\mathcal{F}=\{\mathcal{S}_{k}\}^{\infty}_{k=1}$ be a family of tuples of $\mathbb{R}[x]$. Suppose that $\mathcal{S}_{i}$ and $\mathcal{S}_{j}\in \mathcal{F}$ are of the same degree $m$ of expansion. The following properties hold:
\begin{enumerate}
\item [(i)] $\mathcal{R}(\mathcal{S}_{i}+\mathcal{S}_{j})=\mathcal{R}(\mathcal{S}_{i})+\mathcal{R}(\mathcal{S}_{j})$.
\bigskip

\item [(ii)] $\mathcal{R}(c\mathcal{S}_{i})=c\mathcal{R}(\mathcal{S}_{i})$.
\end{enumerate}
\end{proposition}

\begin{proof}
\begin{enumerate}
\item [(i)] Choose arbitrarily $\mathcal{S}_{i}$, $\mathcal{S}_{j}\in \mathcal{F}$ with the same degree $m$ of expansion. We obtain $\mathcal{R}(\mathcal{S}_{i}+\mathcal{S}_{j})=(\gamma^{-1} \circ \beta \circ \gamma \circ \nabla)^{m}(\mathcal{S}_{i}+\mathcal{S}_{j})=(\gamma^{-1} \circ \beta \circ \gamma \circ \nabla)^{m-1})\circ (\gamma^{-1} \circ \beta \circ \gamma \circ \nabla)(\mathcal{S}_{i}+\mathcal{S}_{j})$. By the linearity of the expansion, we get $(\gamma^{-1} \circ \beta \circ \gamma \circ \nabla)^{m-1}\circ (\gamma^{-1} \circ \beta \circ \gamma \circ \nabla)(\mathcal{S}_{i}+\mathcal{S}_{j})=(\gamma^{-1} \circ \beta \circ \gamma \circ \nabla)^{m-1})((\gamma^{-1} \circ \beta \circ \gamma \circ \nabla)(\mathcal{S}_{i})+(\gamma^{-1} \circ \beta \circ \gamma \circ \nabla)(\mathcal{S}_{j}))$. By induction, have $(\gamma^{-1} \circ \beta \circ \gamma \circ \nabla)^{m}(\mathcal{S}_{i}+\mathcal{S}_{j})=(\gamma^{-1} \circ \beta \circ \gamma \circ \nabla)^{m}(\mathcal{S}_{i})+(\gamma^{-1} \circ \beta \circ \gamma \circ \nabla)^{m}(\mathcal{S}_{j})=\mathcal{R}(\mathcal{S}_{i})+\mathcal{R}(\mathcal{S}_{j})$.
\bigskip

\item [(ii)] The result follows by applying the linearity property of each of the $m$ copies of the map.
\end{enumerate}
\end{proof}
\bigskip

\begin{conjecture}\label{hardy-littlewood}
Let $\mathcal{R}(\mathcal{S})$ be the rank of an expansion on $\mathcal{S}$, where $\mathcal{S}$ consist of polynomials with integer coefficients each of the same parity. There exist some tuple $(b,b,\ldots, b)$ with $b\in \mathbb{Z}$ such that all the entries of $\mathcal{R}(\mathcal{S})+(b,b,\ldots, b)$ and $\mathcal{R}(\mathcal{S})-(b,b,\ldots, b)$ are all prime.
\end{conjecture}

\begin{theorem}\label{rank1}
Let $\{\mathcal{S}_{k}\}^{\infty}_{k=1}$ be the family of tuples of $\mathbb{R}[x]$ such that each tuple has at least two entries of distinct degrees. Let $\mathcal{S}_{i}$, $\mathcal{S}_{j}\in \{\mathcal{S}_{k}\}^{\infty}_{k=1}$ with $\mathrm{deg}(\mathcal{S}_{i})=\mathrm{deg}(\mathcal{S}_{j})=n$. We have $\mathcal{R}(\mathcal{S}_{i})=\mathcal{R}(\mathcal{S}_{j})$ if and only if $\mathcal{S}_{i}-\mathcal{S}_{j}=(a_1,a_2,\ldots,a_n)$ for each $a_i\in \mathbb{R}$. 
\end{theorem}

\begin{proof}
Choose arbitrarily $\mathcal{S}_{i}$, $\mathcal{S}_{j}\in \{\mathcal{S}_{k}\}^{\infty}_{k=1}$, the family of tuples of $\mathbb{R}[x]$ such that $\mathrm{deg}(\mathcal{S}_{i})=\mathrm{deg}(\mathcal{S}_{j})$. Assume  $\mathcal{S}_{i}-\mathcal{S}_{j}=(a_1,a_2,\ldots,a_n)$ for each $a_i\in \mathbb{R}$. Applying $n$ copies of the expansion on both sides of the relation, we get 
$$
(\gamma^{-1} \circ \beta \circ \gamma \circ \nabla)^{n}(\mathcal{S}_{i}-\mathcal{S}_{j})=\mathcal{S}_{0}
$$
where $\mathcal{S}_{0}$ is the null tuple. Since an expansion is linear, we find that 
$$
(\gamma^{-1} \circ \beta \circ \gamma \circ \nabla)^{n}(\mathcal{S}_{i})-(\gamma^{-1} \circ \beta \circ \gamma \circ \nabla)^{n}(\mathcal{S}_{j})=\mathcal{S}_{0}.
$$ 
This implies 
$$
\mathcal{R}(\mathcal{S}_{i})=\mathcal{R}(\mathcal{S}_{j})+\mathcal{S}_{0}=\mathcal{R}(\mathcal{S}_{j}).
$$ 
Conversely, suppose that $\mathcal{R}(\mathcal{S}_{i})=\mathcal{R}(\mathcal{S}_{j})$. By the properties of the rank, we get 
$$
\mathcal{R}(\mathcal{S}_{i}-\mathcal{S}_{j})=\mathcal{S}_{0}.
$$ 
This implies that the entries of $\mathcal{S}_{i}$ and $\mathcal{S}_{j}$ must differ by elements in $\mathbb{R}$.
\end{proof}
\bigskip

The condition $\mathcal{R}(\mathcal{S})=\mathcal{R}(\mathcal{S}_{j})$ on the tuples in $\{\mathcal{S}_{k}\}^{\infty}_{k=1}$ induces an equivalence relation and consequently partitions $\{\mathcal{S}_{k}\}^{\infty}_{k=1}$ into infinitely disjoint classes. Denote by $\mathcal{S}_{i}\sim \mathcal{S}_{j}$ if and only if $\mathcal{R}(\mathcal{S}_{i})=\mathcal{R}(\mathcal{S}_{j})$. It follows that $\mathcal{S}\sim \mathcal{S}$, since $\mathcal{R}(\mathcal{S})=\mathcal{R}(\mathcal{S})$, a reflexive relation is fulfilled. It is also clear that the relation is symmetric. Suppose that $\mathcal{S}_{a}\sim \mathcal{S}_{b}$ and $\mathcal{S}_{b}\sim \mathcal{S}_{c}$. We get $\mathcal{R}(\mathcal{S}_{a})=\mathcal{R}(\mathcal{S}_{b})$ and $\mathcal{R}(\mathcal{S}_{b})=\mathcal{R}(\mathcal{S}_{c})$. It implies $\mathcal{R}(\mathcal{S}_{a})=\mathcal{R}(\mathcal{S}_{c})$, fulfilling transitivity.
\bigskip

\begin{theorem}\label{rank2}
Let $\mathcal{S}_{1}$ and $\mathcal{S}_{2}$ be tuples of $\mathbb{R}[x]$ with $\mathrm{deg}(\mathcal{S}_1)>\mathrm{deg}(\mathcal{S}_{2})$ that satisfy certain initial conditions at each phase of expansion. If $\mathcal{R}(\mathcal{S}_{1})=\mathcal{R}(\mathcal{S}_{2})$, then there exists some $j$ satisfying $1\leq j <\mathrm{deg}(\mathcal{S}_{1})$ such that 
$$
(\gamma^{-1} \circ \beta \circ \gamma \circ \nabla)^{j}(\mathcal{S}_1):=\mathcal{S}_{1}^{j}=\mathcal{S}_{2}.
$$
\end{theorem}

\begin{proof}
Suppose that $\mathcal{S}_{1}$ and $\mathcal{S}_{2}$ are tuples of $\mathbb{R}[x]$. Let $\mathrm{deg}(\mathcal{S}_{1})=k_1$ and $\mathrm{deg}(\mathcal{S}_{2})=k_2$. By definition \ref{rank}, we can write 
$\mathcal{R}(\mathcal{S}_{1})=(\gamma^{-1} \circ \beta \circ \gamma \circ \nabla)^{k_1}(\mathcal{S}_1)$ and $\mathcal{R}(\mathcal{S}_{2})=(\gamma^{-1} \circ \beta \circ \gamma \circ \nabla)^{k_2}(\mathcal{S}_2)$. Under the assumption that $\mathcal{R}(\mathcal{S}_{1})=\mathcal{R}(\mathcal{S}_{2})$, we have $(\gamma^{-1} \circ \beta \circ \gamma \circ \nabla)^{k_2}(\mathcal{S}_2)=(\gamma^{-1} \circ \beta \circ \gamma \circ \nabla)^{k_1}(\mathcal{S}_1)$ if and only if $(\gamma^{-1} \circ \beta \circ \gamma \circ \nabla)^{k_1-k_2}(\mathcal{S}_1)=\mathcal{S}_{2}$. Since $1\leq k_1-k_2<k_1=\mathrm{deg}(\mathcal{S}_{1})$, the claim follows immediately.
\end{proof}

\subsection{The limit of an expansion}

\begin{definition}\label{limit}
Let $\{\mathcal{S}^{m}\}^{\infty}_{m=1}$ be a family of expanded tuples of $\mathcal{S}$ that have at least two entries with distinct degrees. The \emph{limit} of expansion of $\mathcal{S}$ is the first expanded tuple $\mathcal{S}^{j}=(g_1,g_2,\ldots,g_n)$ such that $\mathrm{deg}(g_1)=\mathrm{deg}(g_2)=\cdots =\mathrm{deg}(g_n)$ for $n\geq 3$ and $1\leq j\leq m$. We denote the limit by 
\begin{align}
\lim (\mathcal{S}^{n})=\mathcal{S}^{j}.\nonumber
\end{align}
\end{definition} 

\begin{example}\label{ex2}
Let us consider the tuple $\mathcal{S}=(x^4+x^2,x^5-x^3,x^2+1)$. For the \emph{first phase} expanded tuple, we have 
$$
\mathcal{S}^{1}=(5x^4-3x^2+2x,4x^3+4x,5x^4+4x^3-3x^2+2x).
$$ 
For the \emph{second phase} expanded tuple, we have 
$$
\mathcal{S}^{2}=(20x^3+24x^2-6x+6,40x^3+12x^2-12x+4,20x^3+12x^2-6x+6).
$$ 
We get for the \emph{third phase} expanded tuple 
$$
\mathcal{S}^{3}=(180x^2+48x-18,120x^2+72x-12,180x^2+72x-18).
$$ 
For the \emph{fourth phase} expanded tuple, we have $\mathcal{S}^{4}=(600x+144,720x+120,600x+120)$. For the \emph{fifth phase} expanded tuple, we have 
$$
\mathcal{S}^{5}=(1320,1200,1320).
$$ 
By definition \ref{limit}, we find that $\mathcal{S}^{2}$ is the limit of the expansion. 
\end{example}
\bigskip

Now we prove a result about the existence of the limit of expansion of any tuple of $\mathbb{R}[x]$ having at least two entries of distinct degrees. The method of proof is basically an argument by infinite descent.

\begin{theorem}\label{exist}
Let $\{\mathcal{S}^{m}\}^{\infty}_{m=1}$ be a family of expansions of the tuple $\mathcal{S}$ of $\mathbb{R}[x]$ such that at least two entries have distinct degrees. The limit of expansions $\lim (\mathcal{S}^{n})$ of $\mathcal{S}$ exists.
\end{theorem}

\begin{proof}
Let $\{\mathcal{S}^{m}\}^{\infty}_{m=1}$ be a family of expansions of the tuple $\mathcal{S}$ of $\mathbb{R}[x]$ having at least two entries of distinct degree. Suppose that the limit of expansion does not exist and let $\mathcal{S}^{1}=(f_1,f_2,\ldots,f_n)$ be the first phase expansion of $\mathcal{S}$. We deduce that  $\mathrm{deg}(f_i)\neq \mathrm{deg}(f_j)$ for some $1\leq i,j\leq n$ with $i\neq j$. It follows, in particular, that $\mathcal{S}^{1}\neq \mathcal{R}(\mathcal{S})$ and $\mathcal{S}^{1}\neq \mathcal{S}_{0}$. Thus, the second phase of the expansion exists. Let $\mathcal{S}^{2}=(g_1,g_2,\ldots,g_n)$ be the second phase expanded tuple. By hypothesis, we have $\mathrm{deg}(g_i)\neq \mathrm{deg}(g_{j})$ for some $1\leq i,j\leq n$ with $i\neq j$. We deduce $\mathcal{S}^{2}\neq \mathcal{R}(\mathcal{S})$ and $\mathcal{S}^{2}\neq \mathcal{S}_{0}$. Thus, the third phase expansion also exists. By induction, it follows that the tuple $\mathcal{S}$ of $\mathbb{R}[x]$ admits an infinite number of expansions. This violates Proposition \ref{finite}.
\end{proof}
\bigskip

\begin{theorem}\label{lemma1}
Let $\{\mathcal{S}^{n}\}_{n=1}^{\infty}$ be a family of expanded tuples of the tuple $\mathcal{S}$ of $\mathbb{R}[x]$ such that at least two entries have distinct degrees and satisfy certain initial conditions at each phase of expansion. There exists some number $k$ called the \emph{dimension} of expansion, denoted by $\mathrm{dim}(\mathcal{S})$, such that 
$$
\lim (\mathcal{S}^{n})=(\Delta\circ\gamma^{-1}\circ\beta^{-1}\circ\gamma)^{k}(\mathcal{R}(\mathcal{S}))
$$ 
for some $k<\mathrm{deg}(\mathcal{S})$. 
\end{theorem}

\begin{proof}
Let $\mathcal{S}$ be any tuple of $\mathbb{R}[x]$ that can be expanded, with at least two entries having distinct degrees. The limit exists by Theorem \ref{exist} and since an expansion can only be applied a finite number of times and the map $\Delta\circ\gamma^{-1}\circ\beta^{-1}\circ\gamma$ is a recovery that exists, there exists such a number $k$, so that  
$$
\lim (\mathcal{S}^{n})=(\Delta\circ\gamma^{-1}\circ\beta^{-1}\circ\gamma)^{k}(\mathcal{R}(\mathcal{S})).
$$ 
We only need to show that $k$ lies in the stated range. Suppose that 
$$
\lim (\mathcal{S}^{n})=(\Delta\circ\gamma^{-1}\circ\beta^{-1}\circ\gamma)^{k}(\mathcal{R}(\mathcal{S}))
$$ 
for any $k\geq \mathrm{deg}(\mathcal{S})$. Since  the map is a bijection, it follows that $(\gamma^{-1}\circ\beta\circ\gamma\circ\nabla)^{k}(\lim (\mathcal{S}^{n}))=\mathcal{R}(\mathcal{S})$. We get 
$$
(\gamma^{-1}\circ\beta\circ\gamma\circ\nabla)^{k}(\lim (\mathcal{S}^{n})=\mathcal{S}_{0}
$$ 
which implies $\mathcal{R}(\mathcal{S})=\mathcal{S}_{0}$, and so the rank of an expansion is null. This violates the definition \ref{rank}.
\end{proof}
\bigskip

The above result suggests that we only need the rank and the dimension to determine the limit of an expansion. Here, we use this result to prove an important property concerning the limit of an expansion.

\begin{theorem}
Let $\mathcal{S}_{1}$ and $\mathcal{S}_{2}$ be tuples of the polynomial ring $\mathbb{R}[x]$ with their corresponding family of expanded tuples $\{\mathcal{S}_{1}^{m}\}_{m=1}^{\infty}$ and $\{\mathcal{S}_{2}^{n}\}_{n=1}^{\infty}$ satisfying certain initial conditions. Assume that the limit of each expansion exists. We have $\lim (\mathcal{S}_{1}^{m})=\lim (\mathcal{S}_{2}^{n})$ if and only if $\mathcal{S}_{1}-\mathcal{S}_{2}=(b_1,b_2,\ldots, b_n)$ where $b_{i}\in \mathbb{R}$ for $1\leq i\leq n$. 
\end{theorem}

\begin{proof}
By Theorem \ref{lemma1}, we can write $\lim (\mathcal{S}_{1}^{m})=(\Delta \circ \gamma^{-1}\circ \beta^{-1} \circ \gamma)^{k_1}(\mathcal{R}(\mathcal{S}_1))$ and $\lim (\mathcal{S}_{2}^{n})=(\Delta \circ \gamma^{-1}\circ \beta^{-1} \circ \gamma)^{k_2}(\mathcal{R}(\mathcal{S}_2))$ for some $k_1,k_2\in \mathbb{N}$. Suppose that $\lim (\mathcal{S}_{1}^{m})=\lim (\mathcal{S}_{2}^{n})$, then we must have $(\Delta \circ \gamma^{-1}\circ \beta^{-1} \circ \gamma)^{k_1}(\mathcal{R}(\mathcal{S}_1))=(\Delta \circ \gamma^{-1}\circ \beta^{-1} \circ \gamma)^{k_2}(\mathcal{R}(\mathcal{S}_2))$ if and only if $(\Delta \circ \gamma^{-1}\circ \beta^{-1} \circ \gamma)^{k_1-k_2}(\mathcal{R}(\mathcal{S}_1))=\mathcal{R}(\mathcal{S}_2)$. We claim $k_1=k_2$. Suppose $k_1>k_2$, then we get $(\gamma^{-1} \circ \beta \circ \gamma \circ \nabla)^{k_1-k_2}(\mathcal{R}(\mathcal{S}_2))=\mathcal{S}_{0}=\mathcal{R}(\mathcal{S}_{1})$, which is a contradiction. Again, if $k_2>k_1$, then we have 
$$
(\gamma^{-1} \circ \beta \circ \gamma \circ \nabla)^{k_2-k_1}({\mathcal{R}}(\mathcal{S}_1))=\mathcal{R}(\mathcal{S}_2)=\mathcal{S}_{0}
$$
which cannot hold. Therefore, we must have $k_1=k_2$. It follows that $\mathcal{R}(\mathcal{S}_1)=\mathcal{R}(\mathcal{S}_2)$. Now, thanks to Theorem \ref{rank1}, it implies $\mathcal{S}_{1}-\mathcal{S}_{2}=(b_1,b_2,\ldots, b_n)$ where $b_i\in \mathbb{R}$ for $1\leq i\leq n$. The converse, on the other hand, is straight-forward.
\end{proof}
\bigskip

\subsection{The local number}

\begin{definition}
Let $\mathcal{S}$ be a tuple of $\mathbb{R}[x]$ and $\{\mathcal{S}^{m}\}_{m=1}^{\infty}$ be the family of expanded tuples of $\mathcal{S}$. By the \emph{local number} of expansions, denoted by $\mathcal{L}(\mathcal{S})$, we mean the value of $n$ such that 
$$
(\gamma^{-1}\circ\beta\circ\gamma\circ\nabla)^{n}(\mathcal{S})=\lim (\mathcal{S}^{m}).
$$
\end{definition}
\bigskip

Using Theorem \ref{lemma1}, we find that any tuple of $\mathbb{R}[x]$ satisfying certain initial conditions at each phase of expansion satisfies 
\begin{align}
(\gamma^{-1}\circ\beta\circ\gamma\circ\nabla)^{n}(\mathcal{S})=(\Delta\circ\gamma^{-1}\circ\beta^{-1}\circ\gamma)^{k}(\mathcal{R}(\mathcal{S}))\nonumber
\end{align}
if and only if 
\begin{align}
(\gamma^{-1}\circ\beta\circ\gamma\circ\nabla)^{n+k}(\mathcal{S})=\mathcal{R}(\mathcal{S})\nonumber
\end{align}
for some $k\geq 0$. By the definition of the rank of an expansion, it follows that 
\begin{align}
n+k=\mathrm{deg}(\mathcal{S}).\nonumber
\end{align}
We call this equation the \emph{principal equation}, where $\mathcal{L}(\mathcal{S})=n$, $\mathrm{dim}(\mathcal{S})=k$ and $\mathrm{deg}(\mathcal{S})$ are the local number, dimension, and degree of expansion, respectively, of the expansion of $\mathcal{S}$. It is interesting to recognize that the value of the local number $\mathcal{L}(\mathcal{S})$ is bounded and cannot be greater than the dimension of an expansion. This assertion is confirmed as follows. 
\bigskip

\begin{theorem}\label{local number}
Let $\mathcal{S}$ be a tuple of $\mathbb{R}[x]$ that satisfies a certain initial conditions at each phase with $\mathrm{deg}(\mathcal{S})\geq 4$. If $\mathrm{dim}(\mathcal{S})>2$, then the local number $\mathcal{L}(\mathcal{S})$ must satisfy the inequality 
\begin{align}
0 \leq \mathcal{L}(\mathcal{S})\leq 2.\nonumber
\end{align}
\end{theorem}

\begin{proof}
Suppose on the contrary that $\mathcal{L}(\mathcal{S})>2$. We deduce from the principal equation that $\mathrm{dim}(\mathcal{S})<\mathrm{deg}(\mathcal{S})-2$, so that $(\gamma^{-1} \circ \beta \circ \gamma \circ \nabla)^{\mathrm{dim}(\mathcal{S})+2}(\mathcal{S})\neq \mathcal{R}(\mathcal{S})$ and  $(\gamma^{-1} \circ \beta \circ \gamma \circ \nabla)^{\mathrm{dim}(\mathcal{S})+2}(\mathcal{S})\neq \mathcal{S}_{0}$. It follows that $(\gamma^{-1}\circ\beta\circ\gamma\circ \nabla)^{\mathrm{dim}(\mathcal{S})+2}(\mathcal{S})=\mathcal{S}_{1}$. Theorem \ref{rank2} gives $\mathcal{R}(\mathcal{S})=\mathcal{R}(\mathcal{S}_1)$. We have 
\begin{align}
(\gamma^{-1}\circ \beta \circ \gamma \circ \nabla)^{\mathrm{deg}(\mathcal{S})}(\mathcal{S})=(\gamma^{-1} \circ \beta \circ \gamma \circ \nabla)^{\mathrm{deg}(\mathcal{S}_{1})}(\mathcal{S}_{1})\nonumber
\end{align}
if and only if 
\begin{align}
(\gamma^{-1} \circ \beta \circ \gamma \circ \nabla)^{\mathrm{deg}(\mathcal{S})-\mathrm{deg}(\mathcal{S}_{1})}(\mathcal{S})=\mathcal{S}_{1}.\nonumber
\end{align}
This implies $\mathrm{deg}(\mathcal{S})-\mathrm{deg}(\mathcal{S}_{1})=\mathrm{dim}(\mathcal{S})+2$. Again, using the principal equation, we find that
\begin{align}
\mathcal{L}(\mathcal{S})=\mathrm{deg}(\mathcal{S}_{1})+2.\nonumber
\end{align}
We deduce $\mathrm{deg}(\mathcal{S}_{1})+2=\mathcal{L}(\mathcal{S})=\mathrm{deg}(\mathcal{S})-\mathrm{dim}(\mathcal{S})<\mathrm{deg}(\mathcal{S})-2$, so that $\mathrm{deg}(\mathcal{S}_{1})+4<\mathrm{deg}(\mathcal{S})$. Since $\mathrm{deg}(\mathcal{S})\geq 4$, we must have $\mathrm{deg}(\mathcal{S}_{1})+4\leq 4$. We get $\mathrm{deg}(\mathcal{S}_{1})\leq 0$. This leaves the only possibility that $\mathrm{deg}(\mathcal{S}_{1})=0$, violating the fact that $(\gamma^{-1}\circ\beta\circ\gamma\circ \nabla)^{\mathrm{dim}(\mathcal{S})+2}(\mathcal{S})\neq \mathcal{R}(\mathcal{S})$ and  $(\gamma^{-1}\circ\beta\circ\gamma\circ \nabla)^{\mathrm{dim}(\mathcal{S})+2}(\mathcal{S})\neq \mathcal{S}_{0}$. 
\end{proof}
\bigskip

\begin{theorem}\label{main}
Let $\mathcal{S}$ be a tuple of $\mathbb{R}[x]$ satisfying certain initial conditions at each phase of expansion such that $\mathcal{S}$ is not a tuple of $\mathbb{R}$. The system 
\begin{align}
\lim (\mathcal{S}^n)=\mathcal{S}_0,\nonumber
\end{align}
 where $\mathcal{S}_0$ is the null tuple has no non-trivial solution.
\end{theorem}

\begin{proof}
Let $\mathcal{S}\in \mathbb{R}[x]$ and suppose that there exist some $a\in \mathbb{R}$ for $a\neq 0$ such that $\lim(\mathcal{S}^n)(a)=\mathcal{S}_0$. By Theorem \ref{lemma1}, we can write \begin{align}
(\Delta \circ \gamma^{-1}\circ \beta^{-1}\circ \gamma)^{\mathrm{dim}(\mathcal{S})}(\mathcal{R}(\mathcal{S}))(a)=\mathcal{S}_0.\nonumber
\end{align}
We get
\begin{align}
\mathcal{R}(\mathcal{S})(a)&=(\gamma^{-1}\circ\beta\circ\gamma \circ\nabla)^{\mathrm{dim}(\mathcal{S})}(\mathcal{S}_0)\nonumber \\&=\mathcal{S}_{0}.\nonumber
\end{align}
We deduce $\mathcal{R}(\mathcal{S})(a)=\mathcal{R}(\mathcal{S})=\mathcal{S}_0$. This can only happen if $\mathcal{S}$ is a tuple of $\mathbb{R}$, which violates the requirement that $\mathcal{S}$ is not a tuple of $\mathbb{R}$.
\end{proof}
\bigskip

This result can be used to investigate the existence of a solution to certain systems of differential equations. The following will illustrate this claim in clear detail.

\subsection{Application to solutions of systems of differential equations}

\begin{corollary}
Let $f_1,f_2,\ldots f_n\in \mathbb{R}[x]$ be such that $\mathrm{deg}(f_i)\neq \mathrm{deg}(f_j)$ for some $1\leq i,j\leq n$ and satisfy $f_i(a)=b_{i1}, f_i'(a)=b_{i2},\ldots f_i^{n}(a)=b_{in}$ for each $i=1,2,\ldots n$ for $a\in \mathbb{R}$. If \begin{align}
\mathrm{deg}\left(\sum \limits_{i\neq 1}\frac{df_i}{dx}\right)=\mathrm{deg}\left(\sum \limits_{i\neq 2}\frac{df_i}{dx}\right)=\ldots \mathrm{deg}\left(\sum \limits_{i\neq n}\frac{df_i}{dx}\right)\nonumber
\end{align}
then the system 
\begin{eqnarray}
f_2'+f_3'+\cdots +f_n'=0\nonumber \\f_1'+f_3'+\cdots +f_n'=0\nonumber \\ \vdots \vdots  \vdots \vdots \vdots \vdots \vdots \vdots \vdots \vdots \vdots \vdots \vdots \vdots \vdots \vdots \vdots \vdots \vdots \vdots \vdots \vdots \vdots \vdots \vdots \vdots \vdots \vdots \vdots \vdots \vdots \vdots \vdots \vdots \nonumber \\ f'_1+f_2'+\cdots +f'_{n-1}=0\nonumber
\end{eqnarray}
has no non-trivial solution.
\end{corollary}

\begin{proof}
Suppose that $f_1,f_2,\ldots, f_n\in \mathbb{R}[x]$ is such that $\mathrm{deg}(f_i)\neq \mathrm{deg}(f_j)$ for some $1\leq i,j \leq n$. Consider the tuple $\mathcal{S}:=(f_1,f_2,\ldots,f_n)$. It follows that $\mathcal{S}$ admits an expansion, since 
\begin{align}
\mathrm{deg}\left(\sum \limits_{i\neq 1}\frac{df_i}{dx}\right)=\mathrm{deg}\left(\sum \limits_{i\neq 2}\frac{df_i}{dx}\right)=\ldots \mathrm{deg}\left(\sum \limits_{i\neq n}\frac{df_i}{dx}\right).\nonumber
\end{align}
By Theorem \ref{local number}, we get 
\begin{align}
\lim(\mathcal{S}^n)=\left(f_2'+f_3'+\cdots+f_n',f_1'+f_3'+\cdots +f_n',\ldots,f'_1+f_2'+\cdots+f'_{n-1}\right).\nonumber 
\end{align}
Since $f_i$ for each $i=1,2,\ldots n$ with its higher order derivatives satisfy certain initial conditions, it follows that each phase of expansion of the tuple $\mathcal{S}$ satisfies a certain initial condition. The hypotheses of Theorem \ref{main} are satisfied and the system does not have a solution.
\end{proof}

\subsection{The measure of an expansion}

\begin{definition}\label{measure1}
Let $\mathcal{S}$ be a tuple of $\mathbb{R}[x]$. By the \emph{measure} of expansion on $\mathcal{S}$, denoted by $\mathcal{N}(\mathcal{S})$, we mean $\mathcal{N}(\mathcal{S})=||\mathcal{R}(\mathcal{S})||$, where $||\cdot ||$ is the usual norm in $\mathbb{R}^{n}$. 
\end{definition}
\bigskip

\begin{proposition}\label{measure2}
Let $\mathcal{S}_{1}$, $\mathcal{S}_{2}$ be tuples of $\mathbb{R}[x]$ each having the same degree of expansion. The following properties of the measure of expansions hold:
\begin{enumerate}
\item [(i)] $\mathcal{N}(\mathcal{S})\geq 0$. (Positivity)
\bigskip

\item [(ii)] $\mathcal{N}(\mu \mathcal{S})=\mu \mathcal{N}(\mathcal{S})$, for $\mu \in \mathbb{R}$. (Homogeneity)
\bigskip

\item [(iii)] $\mathcal{N}(\mathcal{S}_{1}+\mathcal{S}_{2})\leq \mathcal{N}(\mathcal{S}_{1})+\mathcal{N}(\mathcal{S}_{2})$. (Triangle inequality)
\end{enumerate}
\end{proposition} 

\begin{proof}
\begin{enumerate}
\item [(i)] Clearly, $\mathcal{R}(\mathcal{S}_{0})=\mathcal{S}_{0}$ and it follows that $\mathcal{N}(\mathcal{S}_{0})=0$. Conversely, suppose that $\mathcal{N}(\mathcal{S})=0$, then $||\mathcal{R}(\mathcal{S})||=0$. This implies $\mathcal{R}(\mathcal{S})=\mathcal{S}_{0}$. By definition \ref{rank}, we get $\mathcal{S}=\mathcal{S}_{0}$. Thus, the positivity property follows immediately.
\bigskip

\item [(ii)] Let $\mu \in \mathbb{R}$. We have $\mathcal{N}(\mu \mathcal{S})=||\mathcal{R}(\mu \mathcal{S})||$. By the properties of rank, we get $\mathcal{N}(\mu \mathcal{S})=||\mu \mathcal{R}(\mathcal{S})||=||\mu|| ||\mathcal{R}(\mathcal{S})||=\mu ||\mathcal{R}(\mathcal{S})||=\mu \mathcal{N}(\mathcal{S})$. Thus, the homogeneity property is also satisfied.
\bigskip

\item [(iii)] Let $\mathcal{S}_{1}$ and $\mathcal{S}_{2}$ be any $n$ tuples of $\mathbb{R}[x]$ each having the same degree of expansion. We have $\mathcal{N}(\mathcal{S}_{1}+\mathcal{S}_{2})=||\mathcal{R}(\mathcal{S}_{1}+\mathcal{S}_{2})||$. Using the properties of rank, we get $\mathcal{N}(\mathcal{S}_{1}+\mathcal{S}_{2})=||\mathcal{R}(\mathcal{S}_{1}+\mathcal{S}_{2})||=||\mathcal{R}(\mathcal{S}_{1})+\mathcal{R}(\mathcal{S}_{2})||\leq ||\mathcal{R}(\mathcal{S}_{1})||+||\mathcal{R}(\mathcal{S}_{2})||=\mathcal{N}(\mathcal{S}_{1})+\mathcal{N}(\mathcal{S}_{2})$, and the triangle inequality is satisfied.
\end{enumerate}
\end{proof}
\bigskip

The proposition \ref{measure2} indicates that the measure of an expansion is a norm. It also assigns concrete values to expansions on the tuples of $\mathbb{R}[x]$. This measure is expected to be large in magnitude if and only if the expansion process is very long. That is, if the degree of expansion is very large, then we would expect the norm of expansion to be relatively large. This will serve as a criterion for determining the degree of expansion, which we shall discuss later. We will show that the norm for the expansion of any tuple of $\mathbb{R}[x]$ is unique up to rearrangement of entries and translation by a tuple of $\mathbb{R}$. 

\begin{lemma}\label{permutation1}
Let $\mathcal{\tau}$ be any permutation on the set $\{1,2,\ldots, n\}$. We have 
$$
\tau \circ (\gamma^{-1} \circ \beta \circ \gamma \circ \nabla)=(\gamma^{-1} \circ \beta \circ \gamma \circ \nabla) \circ \tau.
$$
\end{lemma}

\begin{proof}
Let $\mathcal{S}=(f_1,f_2,\ldots,f_n)$ be a tuple of $\mathbb{R}[x]$. We get $\tau(\mathcal{S})=\tau((f_1,f_2,\ldots,f_n))=(f_{\tau(1)},f_{\tau(2)},\ldots,f_{\tau(n)})$ and deduce by Proposition \ref{composite} that $(\gamma^{-1} \circ \beta \circ \gamma \circ \nabla)(\tau(\mathcal{S}))=(f_{\tau(2)}'+f_{\tau(3)}'+\cdots +f_{\tau(n)}',f_{\tau(1)}'+f_{\tau(3)}'+\cdots +f_{\tau(n)}',\ldots,f_{\tau(1)}'+f_{\tau(2)}'+\cdots +f_{\tau(n-1)}')$. On the other hand, by proposition \ref{composite}, we deduce $\tau \circ (\gamma^{-1} \circ \beta \circ \gamma \circ \nabla)(\mathcal{S})=\tau(((f_2'+f_3'+\cdots +f_n',f_1'+f_3'+\cdots +f_n',\ldots,f_1'+f_2'+\cdots +f_{n-1}')))=(f_{\tau(2)}'+f_{\tau(3)}'+\cdots +f_{\tau(n)}',f_{\tau(1)}'+f_{\tau(3)}'+\cdots +f_{\tau(n)}',\ldots,f_{\tau(1)}'+f_{\tau(2)}'+\cdots+f_{\tau(n-1)}')$. Comparing both sides, the claim follows immediately.
\end{proof}

\begin{theorem}\label{permutation2}
Let $\mathcal{S}_{1}$ and $\mathcal{S}_{2}$ be any two $n$ tuples of $\mathbb{R}[x]$ each having the same degree of expansion. Then  $\mathcal{N}(\mathcal{S}_{1})=\mathcal{N}(\mathcal{S}_{2})$ if and only if there exists a tuple $\mathcal{S}_{a}$ with $\mathrm{deg}(\mathcal{S}_a)<\mathrm{deg}(\mathcal{S}_1)$ and a permutation $\tau:\{1,2,\ldots,n\}\longrightarrow \{1,2,\ldots,n\}$ such that $\mathcal{S}_{2}=\mathrm{Sgn}(\tau)\tau(\mathcal{S}_{1})+\mathcal{S}_{a}$, where 
$$
\tau(\mathcal{S}_{1})=\tau((f_1,f_2,\ldots,f_n)):=(f_{\tau(1)},f_{\tau(2)},\ldots,f_{\tau(n)}).
$$
\end{theorem}

\begin{proof}
Let $\mathcal{S}_{1}$ and $\mathcal{S}_{2}$ be any two $n$ tuples of $\mathbb{R}[x]$, each having the same degree of expansion, and suppose there exist a tuple $\mathcal{S}_{a}$ such that $\mathrm{deg}(\mathcal{S}_a)<\mathrm{deg}(\mathcal{S}_2)$ and a permutation $\tau$ such that $\mathcal{S}_{2}=\mathrm{Sgn}(\tau)\tau(\mathcal{S}_{1})+\mathcal{S}_{a}$. It follows from Theorem \ref{rank1} that $\mathcal{R}(\mathcal{S}_{2})=\mathcal{R}(\mathrm{Sgn}(\tau)\tau(\mathcal{S}_{1})+\mathcal{S}_{a})=\mathcal{R}(\mathrm{Sgn}(\tau)\tau(\mathcal{S}_{1}))=(\gamma^{-1} \circ \beta \circ \gamma \circ \nabla)^{\mathrm{deg}(\mathcal{S}_{2})}(\mathrm{Sgn}(\tau)\tau(\mathcal{S}_{1}))=\mathrm{Sgn}(\tau)(\gamma^{-1} \circ \beta \circ \gamma \circ \nabla)^{\mathrm{deg}(\mathcal{S}_{2})}\circ \tau(\mathcal{S}_{1})$. Applying the lemma \ref{permutation1} for the $\mathrm{deg}(\mathcal{S}_{1})=\mathrm{deg}(\mathcal{S}_{2})$ number of times, we find that 
\begin{align}
\mathcal{R}(\mathcal{S}_{2})&=\mathrm{Sgn}(\tau)(\gamma^{-1} \circ \beta \circ \gamma \circ \nabla)^{\mathrm{deg}(\mathcal{S}_{2})}\circ \tau(\mathcal{S}_{1})\nonumber \\&=\mathrm{Sgn}(\tau)\tau \circ (\gamma^{-1} \circ \beta \circ \gamma \circ \nabla)^{\mathrm{deg}(\mathcal{S}_{2})}(\mathcal{S}_{1})\nonumber \\&=\mathrm{Sgn}(\tau)\tau \circ (\gamma^{-1} \circ \beta \circ \gamma \circ \nabla)^{\mathrm{deg}(\mathcal{S}_{1})}(\mathcal{S}_{1})\nonumber \\&=\mathrm{Sgn}(\tau)\tau(\mathcal{R}(\mathcal{S}_{1})).\nonumber 
\end{align}
It follows from the relation derived above that the rank of expansion of $\mathcal{S}_{1}$ is a permutation of the rank of expansion of $\mathcal{S}_{2}$ up to signs. Thus, we must have $\mathcal{N}(\mathcal{S}_{1})=||\mathcal{R}(\mathcal{S}_{1})||=||\mathcal{R}(\mathcal{S}_{2})||=\mathcal{N}(\mathcal{S}_{2})$. Conversely, suppose that $\mathcal{N}(\mathcal{S}_{1})=\mathcal{N}(\mathcal{S}_{2})$. By definition \ref{measure1}, we get $||\mathcal{R}(\mathcal{S}_{1})||=||\mathcal{R}(\mathcal{S}_{2})||$. This implies that $\mathcal{R}(\mathcal{S}_{1})$ is a permutation of $\mathcal{R}(\mathcal{S}_{2})$ up to signs. That is, there exists some permutation $\tau$ on $\mathcal{R}(\mathcal{S}_{1})$ such that $\mathrm{Sgn}(\tau)\tau(\mathcal{R}(\mathcal{S}_{1}))=\mathcal{R}(\mathcal{S}_{2})$. By Lemma \ref{permutation1}, we can write $\mathcal{R}(\mathcal{S}_{2})=\mathcal{R}(\mathrm{Sgn}(\tau)\tau(\mathcal{S}_{1}))$. Since $\mathrm{deg}(\mathcal{S}_{1})=\mathrm{deg}(\mathcal{S}_{2})$, Theorem \ref{rank1} implies that $\mathrm{Sgn}(\tau)\tau(\mathcal{S}_{1})-\mathcal{S}_{2}=\mathcal{S}_{b}$, where $\mathcal{S}_{b}$ is a tuple of $\mathbb{R}$ and $\mathrm{deg}(\mathcal{S}_b)<\mathrm{deg}(\mathcal{S}_2)$.
\end{proof}

\begin{conjecture}\label{measure inequality}
Let $\mathcal{S}$ be any tuple of $\mathbb{R}[x]$ with $\mathrm{deg}(\mathcal{S})>1$ and satisfying a certain initial conditions at each phase. The double inequality 
\begin{align}
||\mathcal{S}(deg(\mathcal{S}))||\asymp \mathcal{N}(\mathcal{S})\nonumber
\end{align}
holds.
\end{conjecture}
\bigskip

Conjecture \ref{measure inequality} relates the degree of expansion of any tuple of $\mathbb{R}[x]$ to their measure of expansion.

\subsection{The boundary of an expansion}\label{zero}

In this section, we introduce the concept of the \emph{boundary} of an expansion of a tuple of the polynomial ring $\mathbb{R}[x]$.

\begin{definition}
Let $\{\mathcal{S}_{j}\}_{j=1}^{\infty}$ be a collection  of all tuples of $\mathbb{R}[x]$. By the \emph{boundary point} of the $n^{th}$ phase expansion, denoted by 
$$
\mathcal{Z}[(\gamma^{-1} \circ \beta \circ \gamma \circ \nabla)^n(\mathcal{S}_j)]
$$ 
we mean the set of points
$$
\mathcal{Z}[(\gamma^{-1} \circ \beta \circ \gamma \circ \nabla)^n(\mathcal{S}_j)]:=\bigg\{(a_1,a_2,\ldots, a_m):\mathrm{Id}_{i}[(\gamma^{-1} \circ \beta \circ \gamma \circ \nabla)_{a_i}^{n}(\mathcal{S}_j)]=0\bigg\}.
$$
\end{definition}
\bigskip

 Here, we show that the boundary points decrease with expansion. That is, we expect fewer boundary points at high phases of expansion.

\begin{proposition}\label{bdecreasing}
Let $\{\mathcal{S}_{j}\}_{j=1}^{\infty}$ be a collection  of all tuples of $\mathbb{R}[x]$ and let $\mathcal{Z}[(\gamma^{-1} \circ \beta\circ\gamma \circ \nabla)^n(\mathcal{S}_j)]$ be the boundary points of the $n^{th}$ phase expansion. We have 
\begin{align}
\#\mathcal{Z}[(\gamma^{-1} \circ \beta \circ \gamma \circ \nabla)^{n_1}(\mathcal{S}_j)]>\#\mathcal{Z}[(\gamma^{-1} \circ \beta \circ \gamma \circ \nabla)^{n_2}(\mathcal{S}_j)]\nonumber
\end{align} 
for $1\leq n_1<n_2<\deg(\mathcal{S}_j)$.
\end{proposition}

\begin{proof}
Recall that any polynomial of degree $n$ has at most $n$ roots. Since the degrees of entry of the tuples in the ring $\mathbb{R}[x]$ decrease by $1$ for successive phases of expansion, it follows that the boundary points must decrease with higher phases of expansion.
\end{proof}
\bigskip

We show that any two tuples of polynomials in $\mathbb{R}[x]$ having the same boundaries at some phase of expansion cannot be different.

\begin{theorem}\label{boundary coincides}
Let $\{\mathcal{S}_j\}_{j=1}^{\infty}$ be a collection of tuples of the polynomial ring $\mathbb{R}[x]$. For any $\mathcal{S}_{a}, \mathcal{S}_{b}\in \{\mathcal{S}_j\}_{j=1}^{\infty}$ with $\mathrm{\deg}(\mathcal{S}_a)=\mathrm{deg}(\mathcal{S}_b)$, we have $\mathcal{S}_{a}=\mathcal{S}_{b}+\mathcal{S}_{\mathbb{R}}$ if and only if 
\begin{align}
\mathcal{Z}[(\gamma^{-1}\circ\beta\circ\gamma\circ \nabla)^n(\mathcal{S}_a)]=\mathcal{Z}[(\gamma^{-1}\circ\beta\circ \gamma\circ\nabla)^n(\mathcal{S}_b)]\nonumber
\end{align}
for some $1\leq n<\mathrm{deg}(\mathcal{S}_a)-1=\mathrm{deg}(\mathcal{S}_{b})-1$.
\end{theorem}
\bigskip

\begin{proof}
Suppose that $\mathcal{S}_{a}=\mathcal{S}_{b}+\mathcal{S}_{\mathbb{R}}$. By Theorem \ref{rank1}, we have $\mathcal{R}(\mathcal{S}_{a})=\mathcal{R}(\mathcal{S}_{b})$. There exist some $k\geq 1$ such that $(\Delta \circ \gamma^{-1}\circ \beta^{-1} \circ \gamma)^{\deg(\mathcal{S}_a)-k}(\mathcal{S}_a)=(\gamma^{-1} \circ \beta \circ \gamma \circ \nabla)^n(\mathcal{S}_a)$ and $(\Delta \circ \gamma^{-1}\circ \beta^{-1} \circ \gamma)^{\deg(\mathcal{S}_b)-k}(\mathcal{S}_b)=(\gamma^{-1} \circ \beta \circ \gamma \circ \nabla)^n(\mathcal{S}_b)$. It follows that 
\begin{align}
(\gamma^{-1} \circ \beta \circ \gamma \circ \nabla)^n(\mathcal{S}_a)=(\gamma^{-1} \circ \beta \circ \gamma \circ \nabla)^n(\mathcal{S}_b).\nonumber
\end{align}
Thus, $\mathcal{Z}[(\gamma^{-1}\circ\beta\circ\gamma\circ \nabla)^n(\mathcal{S}_a)]=\mathcal{Z}[(\gamma^{-1} \circ \beta \circ \gamma \circ \nabla)^n(\mathcal{S}_b)]$. Conversely, let $ \mathcal{Z}[(\gamma^{-1}\circ\beta\circ\gamma\circ \nabla)^n(\mathcal{S}_a)]=\mathcal{Z}[(\gamma^{-1}\circ\beta\circ \gamma\circ\nabla)^n(\mathcal{S}_b)]$ and suppose that $\mathcal{S}_a=\mathcal{S}_b+\mathcal{S}_{\mathbb{R}[x]}$. It  follows that $\mathrm{deg}(\mathcal{S}_{\mathbb{R}[x]})\leq \mathrm{deg}(\mathcal{S}_a)=\mathrm{deg}(\mathcal{S}_b)$. By Proposition \ref{bdecreasing}, we deduce
\begin{align}
\# \mathcal{Z}[(\gamma^{-1}\circ\beta\circ\gamma\circ \nabla)^n(\mathcal{S}_{\mathbb{R}[x]})]&\leq \#\mathcal{Z}[(\gamma^{-1}\circ\beta\circ\gamma\circ \nabla)^n(\mathcal{S}_a)]\nonumber \\&=\#\mathcal{Z}[(\gamma^{-1} \circ\beta\circ\gamma\circ\nabla)^n(\mathcal{S}_b)].\nonumber
\end{align} 
On the other hand, we observe
\begin{align}
\mathcal{Z}[(\gamma^{-1}\circ\beta\circ\gamma\circ \nabla)^n(\mathcal{S}_a)]&\subset \mathcal{Z}[(\gamma^{-1} \circ \beta\circ\gamma\circ\nabla)^n(\mathcal{S}_a-\mathcal{S}_b)]\nonumber \\&\subset \mathcal{Z}[(\gamma^{-1}\circ \beta\circ\gamma\circ\nabla)^n(\mathcal{S}_{\mathbb{R}[x]})].\nonumber
\end{align}
Thus, it follows 
\begin{align}
\#\mathcal{Z}[(\gamma^{-1}\circ\beta\circ\gamma\circ \nabla)^n(\mathcal{S}_a)]&=\#\mathcal{Z}[(\gamma^{-1}\circ\beta \circ\gamma\circ\nabla)^n(\mathcal{S}_b)]\nonumber \\&<\# \mathcal{Z}[(\gamma^{-1}\circ\beta\circ\gamma\circ \nabla)^n(\mathcal{S}_{\mathbb{R}[x]})]\nonumber
\end{align}
which is a contradiction.
\end{proof}

\subsection{The co-boundary of expansion}\label{boundary}

In this section, we introduce the concept of the \emph{co-boundary} of an expansion.

\begin{definition}
Let $\{\mathcal{S}_j\}_{j=1}^{\infty}$ be a collection of tuples of the polynomial ring $\mathbb{R}[x]$, and let $\mathcal{S}_0$ be a boundary point of the $n$th phase expansion. By the \emph{free point} generated by $\mathcal{S}_0$, we mean the tuples \begin{align}
(\gamma^{-1}\circ\beta\circ\gamma\circ\nabla)_{a_i}^{n}(\mathcal{S}_j)\nonumber
\end{align}
where $\mathcal{S}_0=(a_1,a_2,\ldots,a_m)$. By \emph{co-boundary points} generated by the boundary points, we mean points of the form $a_i\mathcal{S}_e$ for $1\leq i \leq m$. The co-boundary points form the \emph{co-boundary}.
\bigskip

We will show that the norm of free points of the $n$th phase expansion for $n\geq 1$ cannot be small. It is also reasonable to believe that the boundary points are distant from each other for a higher phase expansion, where the boundary points become sparse.

\begin{conjecture}\label{distribution}
Let $\{\mathcal{S}_{j}\}_{j=1}^{\infty}$ be a collection  of all tuples of $\mathbb{R}[x]$ and let $\mathcal{Z}[(\gamma^{-1}\circ \beta\circ\gamma\circ\nabla)^n(\mathcal{S}_j)]$ be the boundary of the $n^{th}$ phase expansion. Let $\mathcal{S}_k$, $\mathcal{S}_l\in \mathcal{Z}[(\gamma^{-1}\circ\beta\circ\gamma \circ\nabla)^n(\mathcal{S}_j)]$ be any two boundary points. We have 
$$
\inf ||\mathcal{S}_k-\mathcal{S}_l||\geq \epsilon
$$ 
for all $n\geq n_0$ for some $n_0>0$.
\end{conjecture}

\begin{remark}
We are now ready to prove the theorem. 
\end{remark}
\bigskip

\begin{theorem}
Let $\{\mathcal{S}_j\}_{j=1}^{\infty}$ be a collection of tuples of the polynomial ring $\mathbb{R}[x]$, and let $\mathcal{S}_t\in \mathcal{Z}[(\gamma^{-1}\circ\beta\circ\gamma\circ\nabla)^{n}(\mathcal{S}_j)]$ be a boundary point of the $n^{th}$ phase expansion where $n<\deg{\mathcal{S}_j}$. We have 
\begin{align}
||(\gamma^{-1}\circ\beta\circ\gamma\circ\nabla)_{a_i}^{n}(\mathcal{S}_j)||>0\nonumber
\end{align}
where $\mathcal{S}_{t}=(a_1,a_2,\ldots,a_m)$ such that $a_i\neq a_j$ for some $1\leq i,j\leq m$.
\end{theorem}

\begin{proof}
Let $\{\mathcal{S}_j\}_{j=1}^{\infty}$ be a collection of tuples of the polynomial ring $\mathbb{R}[x]$. Suppose that $\# \mathcal{Z}[(\gamma^{-1}\circ\beta\circ\gamma\circ\nabla)^{n}(\mathcal{S}_j)]=k$ for some $k>1$ and let $\mathcal{S}_t\in \mathcal{Z}[(\gamma^{-1}\circ\beta\circ\gamma\circ\nabla)^{n}(\mathcal{S}_j)]$ be a boundary point of the $n^{th}$ phase expansion where $n<\deg{\mathcal{S}_j}$. Suppose on the contrary that \begin{align}
||(\gamma^{-1} \circ\beta\circ\gamma\circ\nabla)_{a_i}^{n}(\mathcal{S}_j)||=0\nonumber
\end{align}
where $\mathcal{S}_{t}=(a_1,a_2,\ldots, a_m)$. We get $(\gamma^{-1}\circ\beta\circ\gamma\circ\nabla)_{a_i}^{n}(\mathcal{S}_j)=\mathcal{S}_0$ for $1\leq i \leq m$. Thus, the co-boundary point $a_i\mathcal{S}_{e}$ is also a boundary point. We deduce
\begin{align}
a_i\mathcal{S}_{e} \in \mathcal{Z}[(\gamma^{-1}\circ\beta\circ\gamma\circ\nabla)^{n}(\mathcal{S}_j)],\nonumber
\end{align}
for $1\leq i\leq m$. Since $\mathcal{S}_{t}=(a_1,a_2,\ldots, a_m)$ is such that $a_i\neq a_j$ for some $1\leq i,j\leq m$, it implies that $\mathcal{S}_t \neq a_i\mathcal{S}_{e}$ for some $1\leq i \leq m$. We obtain
\begin{align}
\#\mathcal{Z}[(\gamma^{-1}\circ\beta\circ\gamma\circ\nabla)^{n}(\mathcal{S}_j)]>k\nonumber
\end{align}
violating the size of the boundary.
\end{proof}
\bigskip

The above result places a barrier between the boundary and co-boundary points. In a sense, the boundary points (boundary) and the co-boundary points (co-boundary) generated by expansion should not overlap.

\begin{corollary}
Let $f_1,f_2,\ldots f_n\in\mathbb{R}[x]$. The system 
\begin{eqnarray}
f_2'+f_3'+\cdots +f_n'=0\nonumber \\f_1'+f_3'+\cdots +f_n'=0\nonumber \\ \vdots \vdots  \vdots \vdots \vdots \vdots \vdots \vdots \vdots \vdots \vdots \vdots \vdots \vdots \vdots \vdots \vdots \vdots \vdots \vdots \vdots \vdots \vdots \vdots \vdots \vdots \vdots \vdots \vdots \vdots \vdots \vdots \vdots \vdots \nonumber \\f'_1+f_2'+\cdots +f'_{n-1}=0\nonumber
\end{eqnarray}
has no non-trivial solution.
\end{corollary}
\end{definition}
\bigskip

Here, we introduce a classification scheme for all tuples of the polynomial ring $\mathbb{R}[x]$. This scheme is based on the boundary points of a given phase of expansion.

\begin{definition}
Let $\mathcal{S}\in \mathcal{Z}[(\gamma^{-1}\circ\beta\circ\gamma \circ\nabla)^n(\mathcal{S}_j)]$ for $n\geq 0$. By the \emph{phase identifier}, we mean the value $|\mathrm{Id}_i(\mathcal{S})|$. We say that the phase identifier is \emph{weak} if $\inf |\mathrm{Id}_i(\mathcal{S})|\leq 1$; otherwise, we say it is \emph{strong}.
\end{definition}
\bigskip

In the language of expansivity, we state the celebrated Sendov conjecture which states that any zero of a polynomial must lie in the same unit disk with some zero of the derivative. We restate the conjecture in this language. It is also important to compare this version with the original Sendov conjecture. The zeros of the polynomial $P_n(x)$ of degree $n$ correspond to the co-boundary points, and the zeros of $P_n'(x)$ correspond to the boundary points in the language of expansivity. Thus, we restate the Sendov conjecture in the language as follows:

\begin{conjecture}[Sendov]
Let $P(x)=a_nx^n+a_{n-1}x^{n-1}+\cdots +a_1x+a_0$ be a polynomial of degree $n\geq 2$. Let $\{b_i\}_{i=1}^{n}$ be the set of zeros of $P(x)$ such that $|b_i|\leq 1$ and let $\mathcal{S}=(a_nx^n,a_{n-1}x^{n-1},\ldots, a_1x,a_0)$ be a tuple representation of $P(x)$. For each $b_{i}$, there exist some $\mathcal{S}_{a}\in \mathcal{Z}[(\gamma^{-1}\circ\beta\circ\gamma \circ \nabla)^1(\mathcal{S})]$ such that 
\begin{align}
|\mathrm{Id}_{n+1}(b_i\mathcal{S}_{e}-\mathcal{S}_{a})|\leq 1.\nonumber
\end{align}
\end{conjecture}

\begin{remark}
The Sendov conjecture, in the language of expansivity, can be stated as saying that any co-boundary point of the trivial expansion with weak phase identifier must in some sense be close to some boundary point of the first phase expansion with a weak phase identifier. 
\end{remark}

\subsection{The speed of an expansion}

\begin{definition}\label{speed}
Let $\{\mathcal{S}_j\}_{j=1}^{\infty}$ be a collection of tuples of the polynomial ring $\mathbb{R}[x]$. Let $\mathcal{S}\in \{\mathcal{S}_j\}_{j=1}^{\infty}$. By the \emph{speed} of expansion, denoted by $\upsilon(\mathcal{S})$, we mean \begin{align}
\upsilon (\mathcal{S})=\frac{\mathcal{N}(\mathcal{S})}{\mathrm{deg}(\mathcal{S})}.\nonumber
\end{align}
\end{definition}
\bigskip

Here, we relate the concept of the speed $\upsilon(\mathcal{S})$ of expansion of a tuple of the ring $\mathbb{R}[x]$ to the concept of the measure of expansion. We show that the speed of expansion is unique up to measure. 

\begin{proposition}
Let $\mathcal{F}=\{\mathcal{S}_j\}_{j=1}^{\infty}$ be a collection of $n$ tuples of polynomials in the ring $\mathbb{R}[x]$, and let $\mathcal{S}_{a}$, $\mathcal{S}_b\in \mathcal{F}$. If $\mathcal{N}(\mathcal{S}_a)=\mathcal{N}(\mathcal{S}_b)$, then $\upsilon(\mathcal{S}_a)=\upsilon(\mathcal{S}_b)$.  
\end{proposition}

\begin{proof}
Let $\mathcal{F}=\{\mathcal{S}_j\}_{j=1}^{\infty}$ be a collection of tuples of the ring $\mathbb{R}[x]$, and suppose that $\mathcal{N}(\mathcal{S}_a)=\mathcal{N}(\mathcal{S}_b)$ for $\mathcal{S}_a,\mathcal{S}_b \in \mathcal{F}$. By Theorem \ref{permutation2} get $\mathcal{S}_a=\mathrm{Sgn}(\tau)\tau(\mathcal{S}_b)+\mathcal{S}_k$, where $\mathrm{deg}(\mathcal{S}_k)<\mathrm{deg}(\mathcal{S}_b)$ and $\tau:\{1,2,\ldots,n\}\longrightarrow \{1,2,\ldots,n\}$ is some permutation. By Theorem \ref{permutation2}, we deduce $\mathrm{deg}(\mathcal{S}_a)=\mathrm{deg}(\mathcal{S}_b)$. By definition \ref{speed}, the result follows immediately.
\end{proof}
\bigskip

Theorem \ref{measure2} suggests that the measure of expansion of elements in the collection $\mathcal{F}$ is a norm. The speed of expansion $\upsilon(\mathcal{S})$ inherits this property given the relationship with the measure. The following proposition verifies this claim.

\begin{proposition}
The speed of expansion $\upsilon(\mathcal{S})$ is a norm.
\end{proposition}

\begin{proof}
Let $\mathcal{F}$ be a collection of tuples of the polynomial in the ring $\mathbb{R}[x]$. Let $\mathcal{S}\in \mathcal{F}$, then $\upsilon(\mathcal{S})>0$ since $\mathcal{N}(\mathcal{S})>0$. In the case $\upsilon(\mathcal{S})=0$, we get by definition \ref{speed} that $\mathcal{N}(\mathcal{S})=0$. Using Theorem \ref{measure2}, we deduce $\mathcal{S}=\mathcal{S}_0$. Conversely, if $\mathcal{S}=\mathcal{S}_0$, then $\mathcal{N}(\mathcal{S})=\mathcal{N}(\mathcal{S}_0)=0$. This implies $\upsilon(\mathcal{S})=0$. Now, let $a\in \mathbb{R}$ for $a>0$. We deduce 
\begin{align}
\upsilon(a\mathcal{S})&=\frac{\mathcal{N}(a\mathcal{S})}{\mathrm{deg}(a\mathcal{S})}\nonumber \\&=\frac{a\mathcal{N}(\mathcal{S})}{\mathrm{deg}(\mathcal{S})}\nonumber \\&=a\upsilon(\mathcal{S})\nonumber
\end{align}
since, by Theorem \ref{permutation2}, the measure $\mathcal{N}(\mathcal{S})$ is a norm. Also, we get 
\begin{align}
\upsilon(\mathcal{S}_1+\mathcal{S}_2)&=\frac{\mathcal{N}(\mathcal{S}_1+\mathcal{S}_2)}{\mathrm{deg}(\mathcal{S}_1+\mathcal{S}_2)}\nonumber \\& \leq \frac{\mathcal{N}(\mathcal{S}_1)+\mathcal{N}(\mathcal{S}_2)}{\mathrm{deg}(\mathcal{S}_1+\mathcal{S}_2)}\nonumber \\& \leq \frac{\mathcal{N}(\mathcal{S}_1)}{\mathrm{deg}(\mathcal{S}_1)}+\frac{\mathcal{N}(\mathcal{S}_1)}{\mathrm{deg}(\mathcal{S}_2)}\nonumber \\&=\upsilon(\mathcal{S}_1)+\upsilon(\mathcal{S}_2).\nonumber
\end{align}
\end{proof}
\bigskip

Here, we show that if the boundary of two tuples of the polynomial ring $\mathbb{R}[x]$ coincides at some phase of expansion, then they must have the same speed of expansion. 

\begin{proposition}
Let $\mathcal{F}=\{\mathcal{S}_j\}_{j=1}^{\infty}$ be a collection of the tuples of the polynomial ring $\mathbb{R}[x]$. For any $\mathcal{S}_a,\mathcal{S}_b\in \mathcal{F}$ with $\mathrm{deg}(\mathcal{S}_a)=\mathrm{deg}(\mathcal{S}_b)$, if there exists some $n\geq 1$ such that 
\begin{align}
\mathcal{Z}[(\gamma^{-1}\circ\beta\circ\gamma\circ \nabla)^n(\mathcal{S}_a)]=\mathcal{Z}[(\gamma^{-1}\circ\beta\circ \gamma\circ\nabla)^n(\mathcal{S}_b)],\nonumber
\end{align}
then $\upsilon(\mathcal{S}_a)=\upsilon(\mathcal{S}_b)$.
\end{proposition}

\begin{proof}
Let $\mathcal{F}=\{\mathcal{S}_j\}_{j=1}^{\infty}$ and let $\mathcal{S}_a,\mathcal{S}_b\in \mathcal{F}$ with $\mathrm{deg}(\mathcal{S}_a)=\mathrm{deg}(\mathcal{S}_b)$. Suppose that $\mathcal{Z}[(\gamma^{-1}\circ\beta\circ\gamma\circ \nabla)^n(\mathcal{S}_a)]=\mathcal{Z}[(\gamma^{-1}\circ\beta\circ \gamma\circ\nabla)^n(\mathcal{S}_b)]$ for some $n\geq 1$, then it follows from Theorem \ref{boundary coincides} that $\mathcal{S}_a=\mathcal{S}_b+\mathcal{S}_{\mathbb{R}}$. By Theorem \ref{rank1}, it follows that $\mathcal{N}(\mathcal{S}_a)=\mathcal{N}(\mathcal{S}_b)$. The claim follows from definition \ref{speed}.
\end{proof}

\subsection{Momentum of phase expansions}

\begin{definition}\label{momentum}
Let $\{\mathcal{S}_j\}_{j=1}^{\infty}=\mathcal{F}$ be any collection of the tuples of $\mathbb{R}[x]$. By the \emph{momentum} of the $n$th phase expansion, denoted by $\mathcal{M}(\mathcal{S}^n_j)$, we mean 
\begin{align}
\mathcal{M}(\mathcal{S}^n_j):=\upsilon(\mathcal{S}_j^{n_j})\mathcal{H}(\mathcal{S}_j^n)\nonumber
\end{align}
where $\mathcal{B}^n$ denotes the set of boundary points of the $n^{th}$ phase expansion and 
\begin{align}
\mathcal{H}(\mathcal{S}_j^n)&=\sum \limits_{\mathcal{S}_k\in \mathcal{B}^n}||\mathcal{S}_k||\nonumber
\end{align}
is the \emph{mass} of the $n^{th}$ phase expansion with 
$$
||\mathcal{S}_k||:=\sqrt{\sum \limits_{i=1}^{n}|a_i|^2}
$$ 
for $\mathcal{S}_k=(a_1,a_2,\ldots,a_n)$.
\end{definition}

\section{Inverse problems}\label{sec:inverse}

In this section, we devote our attention to the study of ways to recover a tuple from expanded tuples of $\mathbb{R}[x]$ at any given phase of expansion. Using Proposition \ref{composite} and the concepts of integration, which can be viewed as an inverse of differentiation on tuples of $\mathbb{R}[x]$, we find that for any given expanded tuple, say $\mathcal{S}_1$ and given that the composite map 
$$
\gamma^{-1} \circ \beta \circ \gamma \circ \nabla:\{\mathcal{S}_i\}_{i=1}^{\infty}\longrightarrow \{\mathcal{S}_i\}_{i=1}^{\infty}
$$ 
is bijective, we have $\gamma^{-1} \circ \beta \circ \gamma \circ \nabla(\mathcal{S})=\mathcal{S}_1$ if and only if $\Delta \circ \gamma^{-1}\circ \beta^{-1} \circ \gamma(\mathcal{S}_1)=\mathcal{S}$. In most cases, we would like to keep track of the original tuple that satisfies certain initial conditions. Assume that at $a\in \mathbb{R}$, the desired tuple satisfies $\mathcal{S}_{a}=(a_1,a_2,\ldots,a_n)$, then the one copy map $\Delta \circ \gamma^{-1}\circ\beta^{-1}\circ \gamma(\mathcal{S}_1)$ will be evaluated at $a$, denoted by $\Delta_a \circ \gamma_a^{-1}\circ \beta_a^{-1} \circ \gamma_a$.

\subsection{Inverse problem for second phase expansions}

Let $\mathcal{S}:=(f_1,f_2,\ldots,f_n)$ be a tuple of $\mathbb{R}[x]$ expanded to the second phase. Let $\mathcal{S}^{1}$ and $\mathcal{S}^{2}$ be the first and the second phase expanded tuple, respectively. To recover the original tuple $\mathcal{S}$, we need to recover the first phase expanded tuple from the second, and then the original from the first. We can recover the first phase expanded tuple from the composite map
\begin{align}
\Delta_a \circ\gamma_a^{-1}\circ\beta_a^{-1}\circ \gamma_a(\mathcal{S}^{2})=\mathcal{S}^{1},\nonumber
\end{align}
and the original is obtained by 
\begin{align}
\Delta_b \circ \gamma_b^{-1}\circ \beta_b^{-1} \circ \gamma_b(\mathcal{S}^{1})=\mathcal{S}.\nonumber
\end{align}
Thus, we find $\Delta_b \circ \gamma_b^{-1}\circ \beta_b^{-1} \circ \gamma_b(\Delta_a \circ \gamma_a^{-1}\circ \beta_a^{-1} \circ \gamma_a(\mathcal{S}^{2}))=\mathcal{S}$ and 
\begin{align}
(\Delta_a \circ \gamma_a^{-1}\circ \beta_a^{-1} \circ \gamma_a)^2(\mathcal{S}^2)=\mathcal{S}.\nonumber
\end{align} 
if and only if $a=b$ for $a,b\in \mathbb{R}$.

\subsection{Inverse problem for higher phase expansions}

To recover a tuple from the $n^{th}$ phase expanded tuple, we only need $n$ copies of the recovery map of each phase. The recovery process can be performed by applying the $n$ copies of the map $\Delta\circ\gamma^{-1}\circ\beta^{-1}\circ\gamma$ to the $n^{th}$ phase expanded tuple, for values of $n$ reasonably small. In practice, it will suffice to apply the $n$ distinct copies of the map $(\Delta_{a_1}\circ\gamma_{a_1}^{-1}\circ\beta_{a_1}^{-1} \circ\gamma_{a_1})\circ(\Delta_{a_2}\circ\gamma_{a_2}^{-1}\circ \beta_{a_2}^{-1}\circ\gamma_{a_2})\circ\cdots\circ(\Delta_{a_n} \circ\gamma_{a_n}^{-1}\circ\beta_{a_n}^{-1}\circ\gamma_{a_n})$ to the $n^{th}$ phase expanded tuple. However, this process becomes less efficient and very brutal if the phase expansion number $n$ is sufficiently large. So we ask a fairly natural question as follows: 

\begin{question}
What is the most efficient way to recover a tuple from the $n^{th}$ expanded tuple for sufficiently large values of $n$?
\end{question}

\begin{theorem}
The set $\mathcal{T}:=\{\mathrm{Id}\}\quad \cup \quad \bigcup_{k=1}^{\infty}\left \{(\Delta \circ \gamma^{-1}\circ \beta^{-1} \circ \gamma)^{k}\right\}\quad \cup \quad \bigcup_{k=1}^{\infty}\\ \left \{(\gamma^{-1} \circ \beta \circ \gamma \circ \nabla)^{k}\right\}$ forms a group. 
\end{theorem}

\begin{proof}
Let $\mathcal{S}$ be any tuple of $\mathbb{R}[x]$ satisfying certain initial conditions at each expansion phase. We observe $\mathrm{Id}(\mathcal{S})=\mathcal{S}$, where 
$$
\mathrm{Id}(\mathcal{S})=\mathrm{Id}(f_1,f_2,\ldots,f_n):=(\mathrm{Id}(f_1),\mathrm{Id}(f_2),\ldots,\mathrm{Id}(f_n))=(f_1,f_2,\ldots,f_n)=\mathcal{S}
$$ 
That is, $\mathrm{Id}$ leaves each tuple of $\mathbb{R}[x]$ invariant. Again, arbitrarily choose a tuple $\mathcal{S}$ of $\mathbb{R}[x]$. We get $(\Delta\circ \gamma^{-1}\circ \beta^{-1}\circ\gamma)^{l}\circ(\Delta\circ \gamma^{-1}\circ\beta^{-1}\circ\gamma)^{m}(\mathcal{S})=(\Delta \circ\gamma^{-1}\circ\beta^{-1}\circ\gamma)^{m}\circ(\Delta\circ \gamma^{-1}\circ \beta^{-1} \circ \gamma)^{l}(\mathcal{S})=(\Delta \circ \gamma^{-1}\circ \beta^{-1} \circ \gamma)^{m+l}(\mathcal{S})$. In other words, applying $m$ and $l$ copies of a recovery map is the same as applying $m+l$ copies of the recovery map to  the tuple $\mathcal{S}$. Thus $(\Delta \circ \gamma^{-1}\circ \beta^{-1} \circ \gamma)^{l}\circ (\Delta \circ \gamma^{-1}\circ \beta^{-1} \circ \gamma)^{m}\in \mathcal{T}$. A similar characterization applies to expansions. For mixed maps with distinct copies, we find that 
$$
(\gamma^{-1}\circ\beta\circ\gamma\circ\nabla)^{l}\circ(\Delta \circ\gamma^{-1}\circ\beta^{-1}\circ\gamma)^{m}=(\gamma^{-1}\circ \beta\circ\gamma\circ\nabla)^{l-m}
$$ 
is an expansion provided that $l-m>0$ and is a recovery map in the case $l-m<0$. Each of the either situations is still contained in the set $\mathcal{T}$. Thus, the set is closed. Again, arbitrarily choose a tuple whose degree of expansion is $n$. We deduce 
$$
(\gamma^{-1}\circ\beta\circ\gamma\circ\nabla)^{l}\circ(\Delta \circ\gamma^{-1}\circ\beta^{-1}\circ\gamma)^{l}=(\Delta\circ \gamma^{-1}\circ\beta^{-1}\circ\gamma)^{l}\circ(\gamma^{-1}\circ \beta\circ\gamma\circ\nabla)^{l}=\mathrm{Id}\in \mathcal{T}
$$ 
for $l\leq n$. Thus, each copy of an expansion has a recovery and vice-versa in the set $\mathcal{T}$. The associative property is easy to verify. Hence, the collection is a group.
\end{proof}

\section{Embedding and extension of expansions}\label{sec:expansion embedding and extension}

In this section, we introduce the notion of an \emph{embedding} and an \emph{extension} of a phase of an expansion.

\begin{definition}\label{embedding}
Let $\mathcal{F}=\{\mathcal{S}_j\}_{j=1}^{\infty}$ be a collection of tuples of $\mathbb{R}[x]$. Let $\mathcal{S}_a,\mathcal{S}_b\in \mathcal{F}$. We say that expansion $(\gamma^{-1} \circ \beta \circ\gamma\circ\nabla)^{n_1}(\mathcal{S}_b)$ is an \emph{embedding} of expansion $(\gamma^{-1}\circ\beta\circ \gamma\circ\nabla)^{n_2}(\mathcal{S}_a)$ if 
\begin{align}
\mathcal{Z}[(\gamma^{-1}\circ\beta\circ\gamma\circ\nabla)^{n_1}(\mathcal{S}_b)]\subset\mathcal{Z}[(\gamma^{-1}\circ\beta\circ \gamma\circ\nabla)^{n_2}(\mathcal{S}_a)]\nonumber
\end{align}
for some $n_1, n_2\in \mathbb{N}$. On the other hand, we say that $(\gamma^{-1}\circ\beta\circ\gamma\circ\nabla)^{n_2}(\mathcal{S}_a)$ is an \emph{extension} of the expansion $(\gamma^{-1}\circ\beta\circ\gamma\circ\nabla)^{n_1}(\mathcal{S}_b)$.
\end{definition}
\bigskip

The notion of \emph{embedding} and an \emph{extension} of expansion could be adapted in practice. Intuitively, as the cell wall of a living organism becomes turgid, the membranes expand to make room for this behaviour. This notion reinforces that, in practice, we can pinch any two portions of the membrane and join these ends together, thereby obtaining a membrane similar to the previous membrane but now with fewer materials of the previous membrane. In relation to our work, it is natural to ask whether there exists a tuple whose boundary of expansion represents this boundary. If it does, then how does this tuple relate to the tuple of the actual expansion. The sequel will be devoted to investigate these things in far greater detail.

\begin{proposition}\label{mass embedding}
Let $\mathcal{F}=\{\mathcal{S}_j\}_{j=1}^{\infty}$ be a collection of $n$ tuples of $\mathbb{R}[x]$ and $\mathcal{S}_{a}, \mathcal{S}_{b}\in \mathcal{F}$. If $(\gamma^{-1}\circ\beta\circ \gamma\circ\nabla)^{n_2}(\mathcal{S}_a)$ is an embedding of the expansion $(\gamma^{-1} \circ \beta \circ \gamma \circ \nabla)^{n_1}(\mathcal{S}_b)$, then 
\begin{align}
\mathcal{H}(\mathcal{S}^{n_2}_{a})<\mathcal{H}(\mathcal{S}^{n_1}_b).\nonumber
\end{align}
\end{proposition}

\begin{proof}
Let $\mathcal{S}_{a},\mathcal{S}_{b}\in \mathcal{F}$ and suppose that $(\gamma^{-1}\circ\beta\circ\gamma\circ\nabla)^{n_2}(\mathcal{S}_a)$ is an embedding of expansion $(\gamma^{-1} \circ\beta\circ\gamma\circ \nabla)^{n_1}(\mathcal{S}_b)$. It follows by definition \ref{embedding}
\begin{align}
\mathcal{Z}[(\gamma^{-1}\circ\beta\circ\gamma\circ\nabla)^{n_2}(\mathcal{S}_a)]\subset \mathcal{Z}[(\gamma^{-1} \circ \beta \circ \gamma\circ \nabla)^{n_1}(\mathcal{S}_b)]\nonumber
\end{align}
for some $n_1,n_2\in \mathbb{N}$. The claim follows from this condition by definition \ref{momentum}.
\end{proof}

We now prove the statement we made in the previous remarks concerning the finite process of embedding of expansions. 

\begin{theorem}
Let $\mathcal{F}=\{\mathcal{S}_j\}_{j=1}^{\infty}$ be a collection of tuples of the ring $\mathbb{R}[x]$. Some $\mathcal{S}_a\in \mathcal{F}$ do not admit embedding.
\end{theorem} 

\begin{proof}
Suppose that the collection $\mathcal{F}=\{\mathcal{S}_j\}_{j=1}^{\infty}$ admits an embedding for all $\mathcal{S}_a\in \mathcal{F}$. Then for some $\mathcal{S}_{1}\in \mathcal{F}$, it follows by definition \ref{embedding} that there exist some $\mathcal{S}_{2}\in \mathcal{F}$ such that 
\begin{align}
\mathcal{Z}[(\gamma^{-1}\circ\beta\circ\gamma\circ\nabla)^{n_2}(\mathcal{S}_2)]\subset \mathcal{Z}[(\gamma^{-1}\circ\beta\circ \gamma\circ\nabla)^{n_1}(\mathcal{S}_1)].\nonumber
\end{align}
By Proposition \ref{mass embedding}, we have  $\mathcal{H}(\mathcal{S}_2^{n_2})<\mathcal{H}(\mathcal{S}_1^{n_1})$. Since $\mathcal{S}_2\in \mathcal{F}$, it implies that there exists some $\mathcal{S}_3\in \mathcal{F}$ such that 
\begin{align}
\mathcal{Z}[(\gamma^{-1}\circ\beta\circ\gamma\circ\nabla)^{n_3}(\mathcal{S}_3)]\subset\mathcal{Z}[(\gamma^{-1}\circ\beta\circ \gamma\circ\nabla)^{n_2}(\mathcal{S}_2)].\nonumber
\end{align}
By Proposition \ref{mass embedding}, we get $\mathcal{H}(\mathcal{S}_3^{n_3})<\mathcal{H}(\mathcal{S}_2^{n_2})$. Since the collection is infinite, we obtain by induction 
\begin{align}
\mathcal{H}(\mathcal{S}_1^{n_1})>\mathcal{H}(\mathcal{S}_2^{n_2})>\cdots >\mathcal{H}(\mathcal{S}_n^{n_n})>\cdots >\mathcal{H}(\mathcal{S}_{n+1}^{k_{n+1}})>\cdots >\mathcal{H}(\mathcal{S}_r^{k_r})= 0.\nonumber
\end{align}
We eventually obtain a sequence of masses equal to zero. This cannot happen since the mass $\mathcal{H}(\mathcal{S}_j^{n_j})$ for $j\geq 1$ of elements in the collection $\mathcal{F}$ satisfies 
\begin{align}
\mathcal{H}(\mathcal{S}_j^{n_j})>0.\nonumber 
\end{align}
\end{proof}
\bigskip

\begin{proposition}
Let $\mathcal{F}=\{\mathcal{S}\}_{j=1}^{\infty}$ be a collection of tuples of $\mathbb{R}[x]$. Let $\mathcal{S}_{a}, \mathcal{S}_b\in \mathcal{F}$ and suppose that $(\gamma^{-1} \circ \beta\circ\gamma\circ\nabla)^{n_2}(\mathcal{S}_a)$ is an embedding of expansion $(\gamma^{-1} \circ \beta \circ \gamma \circ \nabla)^{n_1}(\mathcal{S}_b)$. If $\mathcal{M}(\mathcal{S}^{n_1}_a)=\mathcal{M}(\mathcal{S}^{n_2}_b)$, then \begin{align}
\nu (\mathcal{S}_b^{n_2})<\nu (\mathcal{S}_a^{n_1}).\nonumber
\end{align}
\end{proposition}

\begin{proof}
Let $\mathcal{F}=\{\mathcal{S}_j\}_{j=1}^{\infty}$ and $\mathcal{S}_a,\mathcal{S}_b\in \mathcal{F}$. Suppose that $(\gamma^{-1}\circ\beta\circ\gamma\circ\nabla)^{n_2}(\mathcal{S}_a)$ is an embedding of the expansion $(\gamma^{-1} \circ\beta\circ\gamma\circ\nabla)^{n_1}(\mathcal{S}_b)$. By Proposition \ref{mass embedding}, we have $\mathcal{H}(\mathcal{S}^{n_2}_{a})<\mathcal{H}(\mathcal{S}^{n_1}_b)$. Using equation 
\begin{align}
\mathcal{M}(\mathcal{S}_j^n)&=\nu (\mathcal{S}_j^n) \mathcal{H}(\mathcal{S}_j^n)\nonumber
\end{align}
with the condition $\mathcal{M}(\mathcal{S}^{n_1}_a)=\mathcal{M}(\mathcal{S}^{n_2}_b)$, the claim follows immediately.
\end{proof}

\subsection{The index of expansion}

In this section, we introduce the concept of the \emph{index} $\mathcal{I}(\mathcal{S}_j)$ of expansion of the tuple $\mathcal{S}_j$. 

\begin{definition}
Let $\mathcal{P}:=\{\mathcal{S}_j\}_{j=1}^{n}$ be a finite collection of tuples of $\mathbb{R}[x]$. By the \emph{index} of the $m^{th}$ phase expansion of the tuple $\mathcal{S}_k$ for $1\leq k \leq n$, we mean the ratio 
\begin{align}
\mathcal{I}(\mathcal{S}_k^m)=\frac{\sum \limits_{j=1}^{n}\mathcal{M}(\mathcal{S}^m_{j})}{\mathcal{M}(\mathcal{S}^m_k)}.\nonumber
\end{align}
\end{definition}
\bigskip

Here, we establish an inequality that relates the index of expansion of a tuple to the largest size of the number of embeddings of expansion.

\begin{theorem}\label{index}
Let $\mathcal{P}:=\{\mathcal{S}_j\}_{j=1}^{n}$ and suppose that $(\gamma^{-1}\circ\beta\circ\gamma\circ\nabla)^{n_k}(\mathcal{S}_{k})$ $(1\leq k\leq n)$ admits an embedding $(\gamma^{-1}\circ\beta\circ\gamma\circ\nabla)^{n_j}(\mathcal{S}_{j})$ for all $1\leq j\leq n$. If $\nu(\mathcal{S}_k^{n_k})\geq \nu (\mathcal{S}_j^{n_j})$ for all $1\leq j\leq n$, then 
\begin{align}
\mathcal{I}(\mathcal{S}^{n_k}_{k})<n.\nonumber
\end{align}
\end{theorem}

\begin{proof}
Let $\mathcal{P}:=\{\mathcal{S}_j\}_{j=1}^{n}$ and suppose that $(\gamma^{-1}\circ\beta\circ\gamma\circ\nabla)^{n_k}(\mathcal{S}_{k})$ $(1\leq k\leq n)$ admits an embedding $(\gamma^{-1} \circ \beta \circ \gamma \circ \nabla)^{n_j}(\mathcal{S}_{j})$ for all $1\leq j \leq n$. We deduce
\begin{align}
\sum \limits_{j=1}^{n}\mathcal{M}(\mathcal{S}^{n_j}_j)&=\mathcal{M}(\mathcal{S}^{n_1}_1)+\mathcal{M}(\mathcal{S}^{n_2}_2)+\cdots+\mathcal{M}(\mathcal{S}^{s_k}_k)+\cdots+\mathcal{M}(\mathcal{S}^{n_n}_n)\nonumber \\&=\nu (\mathcal{S}_1^{n_1})\mathcal{H}(\mathcal{S}^{n_1}_1)+\nu (\mathcal{S}_2^{n_2})\mathcal{H}(\mathcal{S}^{n_2}_2)+\cdots+\nu (\mathcal{S}_k^{n_k})\mathcal{H}(\mathcal{S}^{n_k}_k)+\cdots+\nu (\mathcal{S}_n^{n_n})\mathcal{H}(\mathcal{S}^{n_n}_n)\nonumber \\&\leq \nu(\mathcal{S}_1^{n_1})\mathcal{H}(\mathcal{S}^{n_k}_k)+\nu(\mathcal{S}_2^{n_2})\mathcal{H}(\mathcal{S}^{n_k}_k)+\cdots+\nu(\mathcal{S}_k^{n_k})\mathcal{H}(\mathcal{S}^{n_k}_k)+\cdots+\nu(\mathcal{S}_n^{n_n})\mathcal{H}(\mathcal{S}^{n_k}_k)\nonumber \\&\leq n\nu (\mathcal{S}_k^{n_k})\mathcal{H}(\mathcal{S}^{n_k}_k)\nonumber \\&=n\mathcal{M}(\mathcal{S}^{n_k}_k)\nonumber
\end{align}
and the inequality is established.
\end{proof}

\begin{corollary}
Let $\mathcal{P}:=\{\mathcal{S}_j\}_{j=1}^{n}$ and suppose that $(\gamma^{-1}\circ\beta\circ\gamma\circ\nabla)^{n_k}(\mathcal{S}_{k})$ $(1\leq k\leq n)$ admits an embedding $(\gamma^{-1}\circ\beta\circ\gamma\circ\nabla)^{n_j}(\mathcal{S}_{j})$ for all $k<j\leq n$ for  $k\geq 1$.  If $\nu(\mathcal{S}_k^{n_k})\geq \nu (\mathcal{S}_j^{n_j})$ for all $k<j\leq n$~$(k\geq 1)$, then 
\begin{align}
\sum \limits_{r=1}^{n}\mathcal{I}(\mathcal{S}^{n_r}_{r})&<\frac{n(n+1)}{2}.\nonumber
\end{align}
\end{corollary}

\begin{proof}
This follows by applying Theorem \ref{index}.
\end{proof}

\subsection{Application of mass embedding to the Sendov conjecture}

In this section, we prove a weak variant of the Sendov conjecture under the assumption that the first phase of an expansion of any tuple $\mathcal{S}_a\in \{\mathcal{S}_j\}_{j=1}^{\infty}$ is an embedding of the trivial expansion.

\begin{theorem}\label{sendov 1}
Let $\mathcal{F}=\{\mathcal{S}_j\}_{j=1}^{\infty}$ be a collection of tuples of $\mathbb{R}[x]$ and let $\mathcal{S}_a,\mathcal{S}_b \in \mathcal{F}$. Suppose that $(\gamma^{-1} \circ \beta \circ \gamma \circ \nabla)^{1}(\mathcal{S}_{a})$ is an embedding of $(\gamma^{-1}\circ\beta\circ\gamma\circ\nabla)^{0}(\mathcal{S}_{b}):=\mathcal{S}_b$, where $\mathcal{S}_b=(P(x), P(x),\ldots,P(x))$ with $P(x):=a_nx^n+a_{n-1}x^{n-1}+\cdots +a_1x+a_0$ for $n\geq 3$ and $\mathcal{S}_a=(\gamma^{-1}\circ \beta\circ\gamma\circ\nabla)^{1}(\mathcal{S}_{c})$, where $\mathcal{S}_c=(a_nx^n,a_{n-1}x^{n-1},\ldots, a_1x,a_0)$, a tuple representation of $P(x)$. If the mass $\mathcal{H}(\mathcal{S}^0_{b})$ of the trivial expansion $(\gamma^{-1} \circ \beta \circ \gamma \circ \nabla)^{0}(\mathcal{S}_{b})$ satisfies
\begin{align}
\mathcal{H}(\mathcal{S}^0_{b})<\delta\nonumber
\end{align}
where $\delta$ ($1>\delta>0$) is sufficiently small, then for each boundary point of the trivial expansion $(\gamma^{-1} \circ \beta \circ \gamma \circ \nabla)^{0}(\mathcal{S}_{b})$ of the form $b_i\mathcal{S}_e$ there exists a boundary point $\mathcal{S}_0\in \mathcal{Z}[(\gamma^{-1}\circ\beta\circ\gamma\circ\nabla)^{1}(\mathcal{S}_a)]$ such that 
\begin{align}
||b_i\mathcal{S}_e-\mathcal{S}_0||<1.\nonumber
\end{align}
\end{theorem}
\bigskip

\begin{proof}
Let $\mathcal{F}=\{\mathcal{S}_j\}_{j=1}^{\infty}$ be a collection of tuples of $\mathbb{R}[x]$ and let $\mathcal{S}_a,\mathcal{S}_b \in \mathcal{F}$. Suppose that $(\gamma^{-1}\circ\beta\circ \gamma \circ\nabla)^{1}(\mathcal{S}_{a})$ is an embedding of $(\gamma^{-1}\circ\beta\circ\gamma\circ\nabla)^{0}(\mathcal{S}_{b})=\mathcal{S}_b$, where 
\begin{align}
\mathcal{S}_b=(P(x),P(x),\ldots,P(x)) \nonumber
\end{align}
with $P(x):=a_nx^n+a_{n-1}x^{n-1}+\cdots +a_1x+a_0$ and $\mathcal{S}_a=(\gamma^{-1}\circ\beta\circ\gamma\circ\nabla)^{1}(\mathcal{S}_{c})$, where $\mathcal{S}_c=(a_nx^n,a_{n-1}x^{n-1},\ldots,a_1x+a_0)$ is a tuple representation of $P(x)$. Furthermore, suppose that for any boundary point $b_i\mathcal{S}_e$ of the trivial expansion $(\gamma^{-1} \circ \beta \circ \gamma \circ \nabla)^{0}(\mathcal{S}_{b})=\mathcal{S}_b$ 
\begin{align}
||b_i\mathcal{S}_e-\mathcal{S}_0||\geq 1\nonumber
\end{align}
for all $\mathcal{S}_0\in \mathcal{Z}[(\gamma^{-1}\circ\beta\circ \gamma\circ\nabla)^{1}(\mathcal{S}_a)]$. Since $\mathcal{H}(\mathcal{S}^0_b)<\delta$ and $\delta$ ($1>\delta>0$) is sufficiently small, it implies that the mass $\mathcal{H}(\mathcal{S}^{1}_a)\geq 1$. Since $(\gamma^{-1}\circ\beta\circ\gamma\circ\nabla)^{1}(\mathcal{S}_{a})$ is an embedding of $(\gamma^{-1}\circ\beta \circ\gamma\circ\nabla)^{0}(\mathcal{S}_{b})=\mathcal{S}_b$, it follows from proposition \ref{mass embedding}
\begin{align}
1>\mathcal{H}(\mathcal{S}^{0}_b)>\mathcal{H}(\mathcal{S}^{1}_a)\geq 1.\nonumber
\end{align}
This inequality cannot hold. This completes the proof of the Theorem.
\end{proof}

\begin{remark}
Theorem \ref{sendov 1} is close to proving the Sendov conjecture. It suggests that for any polynomial $P(x)$ with sufficiently small zeros, we can find a zero of $P'(x)$ that is somewhat close under the assumption that the set of zeros of $P'(x)$ are subsets of the zeros of $P(x)$. This result is much weaker than the statement of the Sendov conjecture.
\end{remark}
\bigskip

Theorem \ref{sendov 1} can be improved by replacing the mass-embedding condition with the assumption that the mass of each phase of expansion diminishes. In principle, the Sendov conjecture would be proven in full if we could unconditionally prove the mass diminishes with higher phase expansions. At the moment, we assume this property to prove Sendov's conjecture. A sequel to this paper will be devoted to studying phase expansions in relation to their masses.
\bigskip

\begin{theorem}\label{sendov 2}
Let $\mathcal{F}=\{\mathcal{S}_j\}_{j=1}^{\infty}$ be a collection of tuples of $\mathbb{R}[x]$ and let $\mathcal{S}_a,\mathcal{S}_b \in \mathcal{F}$. Suppose that $\mathcal{H}(\mathcal{S}^0_b)>\mathcal{H}(\mathcal{S}^1_a)$, where $\mathcal{S}_b=(P(x),P(x),\ldots,P(x))$ with $P(x):=a_nx^n+a_{n-1}x^{n-1}+\cdots+a_1x+a_0$ with $n\geq 3$ and $\mathcal{S}_a=(\gamma^{-1}\circ\beta\circ\gamma\circ\nabla)^{1}(\mathcal{S}_{c})$, where $\mathcal{S}_c=(a_nx^n,a_{n-1}x^{n-1},\ldots,a_1x,a_0)$, a tuple representation of $P(x)$. If the mass $\mathcal{H}(\mathcal{S}^0_{b})$ of the trivial expansion $(\gamma^{-1} \circ \beta \circ \gamma \circ \nabla)^{0}(\mathcal{S}_{b})$ satisfies
\begin{align}
\mathcal{H}(\mathcal{S}^0_{b})<\delta\nonumber
\end{align}
where $1>\delta>0$ is sufficiently small, then for each  boundary point of the trivial expansion $(\gamma^{-1} \circ \beta \circ \gamma \circ \nabla)^{0}(\mathcal{S}_{b})$ of the form $b_i\mathcal{S}_e$ there exists a boundary point $\mathcal{S}_0\in \mathcal{Z}[(\gamma^{-1}\circ\beta\circ\gamma\circ\nabla)^{1}(\mathcal{S}_a)]$ such that 
\begin{align}
||b_i\mathcal{S}_e-\mathcal{S}_0||<1.\nonumber
\end{align}
\end{theorem}

\begin{proof}
Let $\mathcal{F}=\{\mathcal{S}_j\}_{j=1}^{\infty}$ be a collection of tuples of $\mathbb{R}[x]$ and let $\mathcal{S}_a, \mathcal{S}_b \in \mathcal{F}$. Suppose that $\mathcal{H}(\mathcal{S}^0_b)>\mathcal{H}(\mathcal{S}^1_a)$, where 
\begin{align}
\mathcal{S}_b=(P(x),P(x),\ldots,P(x))\nonumber
\end{align}
with $P(x):=a_nx^n+a_{n-1}x^{n-1}+\cdots+a_1x+a_0$ and $\mathcal{S}_a=(\gamma^{-1}\circ\beta\circ\gamma\circ\nabla)^{1}(\mathcal{S}_{c})$, where $\mathcal{S}_c=(a_nx^n,a_{n-1}x^{n-1},\ldots,a_1x+a_0)$ is a tuple representation of $P(x)$. Furthermore, suppose that for any boundary point of the form  $b_i\mathcal{S}_e$ of the trivial expansion $(\gamma^{-1}\circ \beta\circ\gamma\circ\nabla)^{0}(\mathcal{S}_{b})=\mathcal{S}_b$ 
\begin{align}
||b_i\mathcal{S}_e-\mathcal{S}_0||\geq 1\nonumber
\end{align}
for all $\mathcal{S}_0\in \mathcal{Z}[(\gamma^{-1}\circ\beta\circ \gamma\circ\nabla)^{1}(\mathcal{S}_a)]$. Since $\mathcal{H}(\mathcal{S}^0_b)<\delta$ and $1>\delta>0$ is sufficiently small and $n\geq 3$, it implies that the mass $\mathcal{H}(\mathcal{S}^{1}_a)\geq 1$. We obtain the inequality 
\begin{align}
1>\mathcal{H}(\mathcal{S}^{0}_b)>\mathcal{H}(\mathcal{S}^{1}_a)\geq 1.\nonumber
\end{align}
This inequality cannot hold.
\end{proof}
\bigskip

\begin{remark}
It is important to distinguish between Theorem \ref{sendov 1} and Theorem \ref{sendov 2}. The mass-embedding condition assumed in Theorem \ref{sendov 1} evokes a diminishing state of the expansion mass. However, the converse, as stated in Theorem \ref{sendov 2} does not necessarily evoke a mass-embedding regime. Indeed, the mass of an expansion at a given phase could be smaller than the mass of another expansion not because it is an embedding.
\end{remark}

\begin{theorem}\label{parody}
Let $\mathcal{F}=\{\mathcal{S}_j\}_{j=1}^{\infty}$ be a collection of tuples of $\mathbb{R}[x]$. For each $\mathcal{S}\in \mathcal{F}$
\begin{align}
\sum \limits_{k=0}^{\mathrm{deg}(\mathcal{S})-1}\nu (\mathcal{S}^k)&=\nu (\mathcal{S})\mathrm{deg}(\mathcal{S})\log (\mathrm{deg}(\mathcal{S}))+\mathrm{deg}(\mathcal{S})\nu(\mathcal{S})\gamma+O(\nu (\mathcal{S}))\nonumber
\end{align}
where $\gamma=0.5772\cdots$ is the Euler-Macheroni constant.
\end{theorem}

\begin{proof}
Clearly 
\begin{align}
\sum \limits_{k=0}^{\mathrm{deg}(\mathcal{S})-1}\nu(\mathcal{S}^k)&=\nu (\mathcal{S})+\nu (\mathcal{S}^{1})+\cdots+\nu (\mathcal{S}^{\mathrm{deg}(\mathcal{S})-1})\nonumber \\&=\frac{\mathcal{N}(\mathcal{S})}{\mathrm{deg}(\mathcal{S})}+\frac{\mathcal{N}(\mathcal{S}^1)}{\mathrm{deg}(\mathcal{S}^1)}+\cdots+\frac{\mathcal{N}(\mathcal{S}^{\mathrm{deg}(\mathcal{S})-1})}{\mathrm{deg}(\mathcal{S}^{\mathrm{deg}(\mathcal{S})-1})}\nonumber \\&=\mathcal{N}(\mathcal{S})\bigg(\frac{1}{\mathrm{deg}(\mathcal{S})}+\frac{1}{\mathrm{deg}(\mathcal{S}^1)}+\cdots+\frac{1}{\mathrm{deg}(\mathcal{S}^{\mathrm{deg}(\mathcal{S})-1})}\bigg)\nonumber \\&=\mathcal{N}(\mathcal{S})\bigg(\frac{1}{\mathrm{deg}(\mathcal{S})}+\frac{1}{\mathrm{deg}(\mathcal{S})-1}+\cdots+\frac{1}{2}+1\bigg)\nonumber \\&=\mathcal{N}(\mathcal{S})\sum \limits_{m=1}^{\mathrm{deg}(\mathcal{S})}\frac{1}{m}\nonumber \\&=\nu (\mathcal{S})\mathrm{deg}(\mathcal{S})\sum \limits_{m=1}^{\mathrm{deg}(\mathcal{S})}\frac{1}{m}\nonumber
\end{align}
 establishing the formula.
\end{proof}
\bigskip

As we shall see in the sequel, this formula will be useful for studying the diminishing state of the behaviour of the mass of expansions. For the time being, we use this formula to prove that the mass diminishes at some successive phase of expansion.

\begin{theorem}\label{bound}
Let $\mathcal{F}=\{\mathcal{S}_j\}_{j=1}^{\infty}$ be a collection of tuples of $\mathbb{R}[x]$. Suppose that $\mathcal{S}\in \mathcal{F}$, then 
\begin{align}
\mathcal{H}(\mathcal{S}^{n})>\mathcal{H}(\mathcal{S}^{n+1})\nonumber
\end{align}
for some $0\leq n\leq \mathrm{deg}(\mathcal{S})-2$.
\end{theorem}

\begin{proof}
Let $\mathcal{F}=\{\mathcal{S}_j\}_{j=1}^{\infty}$ be a collection of tuples of $\mathbb{R}[x]$ and specify $\mathcal{S}\in \mathcal{F}$. Consider the finite collection $\mathcal{P}=\{\mathcal{S}^k\}_{k=0}^{\mathrm{deg}(\mathcal{S})-1}$. Suppose that \begin{align}
\mathcal{H}(\mathcal{S}^n)\leq \mathcal{H}(\mathcal{S}^{n+1})\nonumber
\end{align}
for all $0\leq n \leq \mathrm{deg}(\mathcal{S})-2$. By Theorem \ref{parody}, we get 
\begin{align}
\sum \limits_{k=0}^{\mathrm{deg}(\mathcal{S})-1}\mathcal{M}(\mathcal{S}^k)&=\sum \limits_{k=0}^{\mathrm{deg}(\mathcal{S})-1}\mathcal{H}(\mathcal{S}^k)\nu (\mathcal{S}^k)\nonumber \\&\leq \mathcal{H}(\mathcal{S}^{\mathrm{deg}(\mathcal{S})-1})\sum \limits_{k=0}^{\mathrm{deg}(\mathcal{S})-1}\nu (\mathcal{S}^k)\nonumber \\& \ll \mathcal{H}(\mathcal{S}^{\mathrm{deg}(\mathcal{S})-1})\nu (\mathcal{S})\mathrm{deg}(\mathcal{S})\log (\mathrm{deg}(\mathcal{S}))\nonumber \\& \ll \mathcal{H}(\mathcal{S}^{\mathrm{deg}(\mathcal{S})-1})\nu (\mathcal{S}^{\mathrm{deg}(\mathcal{S})-1})\mathrm{deg}(\mathcal{S})\log (\mathrm{deg}(\mathcal{S}))\nonumber \\& \ll \mathcal{M}(\mathcal{S}^{\mathrm{deg}(\mathcal{S})-1})\mathrm{deg}(\mathcal{S})\log (\mathrm{deg}(\mathcal{S})).\nonumber
\end{align}
It implies that the index of expansion 
$$
\mathcal{I}((\mathcal{S}^{\mathrm{deg}(\mathcal{S})-2})^1)\ll \mathrm{deg}(\mathcal{S})\log (\mathrm{deg}(\mathcal{S}))
$$ 
contradicting the upper bound in Theorem \ref{index}.
\end{proof}
\bigskip

Here, we classify all phase of expansions with decreasing mass.

\begin{definition}
Let $\mathcal{F}=\{\mathcal{S}_j\}_{j=1}^{\infty}$ be a collection of tuples of $\mathbb{R}[x]$. For any $\mathcal{S}_k\in \mathcal{F}$, we say that the expansion $(\gamma^{-1}\circ\beta \circ\gamma\circ\nabla)^{n}(\mathcal{S}_k)$ is \emph{regular} if $\mathcal{H}(\mathcal{S}_k^n)>\mathcal{H}(\mathcal{S}_k^{n+1})$ for some $0\leq n\leq \mathrm{deg}(\mathcal{S}_k)-2$.
\end{definition}

\begin{theorem}\label{regular}
Let $\mathcal{F}=\{\mathcal{S}_j\}_{j=1}^{\infty}$ be a collection of tuples of $\mathbb{R}[x]$. Suppose that the expansion $(\gamma^{-1}\circ\beta\circ\gamma\circ\nabla)^{n}(\mathcal{S}_k)$ with ($n\leq \mathrm{deg}(\mathcal{S}_k)-3$) is \emph{regular} for $\mathcal{S}_k\in \mathcal{F}$. If  
\begin{align}
\mathcal{H}(\mathcal{S}^n_{k})<\delta \nonumber
\end{align}
for $\delta$ ($0<\delta <1$) sufficiently small, then for each  $\mathcal{S}_1\in Z[(\gamma^{-1}\circ\beta\circ\gamma\circ \nabla)^{n}(\mathcal{S}_{k})]$ there exist some $\mathcal{S}_0\in \mathcal{Z}[(\gamma^{-1}\circ\beta\circ\gamma\circ\nabla)^{n+1}(\mathcal{S}_k)]$ such that 
\begin{align}
||\mathcal{S}_1-\mathcal{S}_0||<1.\nonumber
\end{align}
\end{theorem}
\bigskip

\begin{proof}
Let $\mathcal{F}=\{\mathcal{S}_j\}_{j=1}^{\infty}$ be a collection of tuples of $\mathbb{R}[x]$. Choose arbitrarily $\mathcal{S}_k\in \mathcal{F}$ and suppose that the expansion $(\gamma^{-1} \circ \beta \circ \gamma \circ \nabla)^{n}(\mathcal{S}_k)$ is regular. Suppose that for each $\mathcal{S}_1\in Z[(\gamma^{-1}\circ\beta \circ\gamma\circ\nabla)^{n}(\mathcal{S}_{k})]$, then 
\begin{align}
||\mathcal{S}_1-\mathcal{S}_0||\geq 1\nonumber
\end{align}
for all $\mathcal{S}_0\in \mathcal{Z}[(\gamma^{-1}\circ\beta\circ \gamma\circ\nabla)^{n+1}(\mathcal{S}_k)]$. Since $\mathcal{H}(\mathcal{S}^n_{k})<\delta$ with $\delta$ ($0<\delta<1$) sufficiently small and $n\leq \mathrm{deg}(\mathcal{S}_k)-3$, we get $\mathcal{H}(\mathcal{S}_k^{n+1})\geq 1$. Under the regularity condition, we have 
\begin{align}
1>\delta >\mathcal{H}(\mathcal{S}_k^{n})>\mathcal{H}(\mathcal{S}_k^{n+1})\geq 1.\nonumber
\end{align}
This inequality cannot hold. This completes the proof of the theorem.
\end{proof}
\bigskip

The combination of Theorem \ref{bound} and Theorem \ref{regular} affirms that the Sendov conjecture is true at some phase of expansion.

\begin{conjecture}[The mass law]
Let $\mathcal{S}\in \mathcal{F}$ with $\mathrm{deg}(\mathcal{S})>1$, where $\mathcal{F}=\{\mathcal{S}_j\}_{j=1}^{\infty}$ is a collection of tuples of $\mathbb{R}[x]$. We have 
\begin{align}
\mathcal{H}(\mathcal{S}^{1})\gg \frac{||\mathcal{S}(\mathrm{deg}(\mathcal{S}))||\mathrm{deg}(\mathcal{S})}{\mathcal{N}(\mathcal{S})\log (\mathrm{deg}(\mathcal{S}))}.\nonumber
\end{align}
\end{conjecture}
\bigskip

We could demonstrate the validity of this conjecture under certain assumptions, namely that the mass decreases uniformly with each successive phase of expansion. Additionally, under the assumption of Conjecture \ref{measure inequality}, we have 
\begin{align}
\sum \limits_{k=0}^{\mathrm{deg}(\mathcal{S})-1}\mathcal{M}(\mathcal{S}^k)&\geq \sum \limits_{k=0}^{\mathrm{deg}(\mathcal{S})-1}\mathcal{H}(\mathcal{S}^k)\nu (\mathcal{S}^{\mathrm{deg}(\mathcal{S})-1})\nonumber \\&\gg ||\mathcal{S}(\mathrm{deg}(\mathcal{S}))||\sum \limits_{k=0}^{\mathrm{deg}(\mathcal{S})-1}\mathcal{H}(\mathcal{S}^k)\nonumber \\&\gg ||\mathcal{S}(\mathrm{deg}(\mathcal{S}))||\mathrm{deg}(\mathcal{S})\mathcal{H}(\mathcal{S}^{\mathrm{deg}(\mathcal{S})-1})\nonumber \\&\gg ||\mathcal{S}(\mathrm{deg}(\mathcal{S}))||\mathrm{deg}(\mathcal{S}).\nonumber
\end{align}
Using Theorem \ref{parody}, we have the upper bound 
\begin{align}
\sum \limits_{k=0}^{\mathrm{deg}(\mathcal{S})-1}\mathcal{M}(\mathcal{S}^{k})& \ll \mathcal{H}(\mathcal{S}^{1})\sum \limits_{k=0}^{\mathrm{deg}(\mathcal{S})-1}\nu (\mathcal{S}^{k})\nonumber \\&=(1+o(1))\mathcal{H}(\mathcal{S}^{1})\nu (\mathcal{S})\mathrm{deg}(\mathcal{S})\log (\mathrm{deg}(\mathcal{S}))\nonumber \\&=(1+o(1))\mathcal{H}(\mathcal{S}^{1})\mathcal{N}(\mathcal{S})\log (\mathrm{deg}(\mathcal{S})).\nonumber
\end{align}
Combining the upper bound and the lower bound establishes the lower bound for the mass of the first phase expansion. This derivation is not rigorous. We can make it rigorous without relying on unproven conjectures. At the moment, this quest seems out of reach, because asserting the diminishing state of the mass of expansion in a sufficiently uniform way and establishing the measure inequality seems to be a difficult problem.

\section{Isomorphic boundaries and expansions}\label{sec:isomorphic boundary and expansion}

In this section, we introduce the concept of \emph{isomorphic} boundaries and expansions. 

\begin{definition}\label{isomorphism}
Let $\mathcal{F}=\{\mathcal{S}_j\}_{j=1}^{\infty}$ be a collection of tuples of $\mathbb{R}[x]$. We say that expansions $(\gamma^{-1}\circ\beta\circ\gamma\circ \nabla)^{n_k}(\mathcal{S}_{k})$ and $(\gamma^{-1}\circ\beta\circ\gamma\circ \nabla)^{n_l}(\mathcal{S}_{l})$ ~(resp. boundaries) are \emph{isomorphic} if $\mathcal{H}(\mathcal{S}_{l}^{n_l})=\mathcal{H}(\mathcal{S}_{k}^{n_k})$. We write 
\begin{align}
\mathcal{Z}[(\gamma^{-1} \circ \beta \circ \gamma \circ \nabla)^{n_k}(\mathcal{S}_{k})] \simeq \mathcal{Z}[(\gamma^{-1} \circ \beta \circ \gamma \circ \nabla)^{n_l}(\mathcal{S}_{l})]\nonumber
\end{align}
to denote the isomorphism.
\end{definition}
\bigskip

The notion of isomorphism of expansions induces an equivalence relation and partitions expansions to equivalent classes. We have \begin{align}
\mathcal{Z}[(\gamma^{-1}\circ\beta\circ\gamma\circ\nabla)^{n_k}(\mathcal{S}_{k})]\simeq \mathcal{Z}[(\gamma^{-1}\circ\beta\circ \gamma\circ\nabla)^{n_k}(\mathcal{S}_{k})]\nonumber
\end{align}
since $\mathcal{H}(\mathcal{S}_{k}^{n_k})=\mathcal{H}(\mathcal{S}_{k}^{n_k})$. The symmetric property also holds. For the transitivity property, we have 
\begin{align}
\mathcal{Z}[(\gamma^{-1} \circ \beta \circ \gamma \circ \nabla)^{n_k}(\mathcal{S}_{k})] \simeq \mathcal{Z}[(\gamma^{-1} \circ \beta \circ \gamma \circ \nabla)^{n_l}(\mathcal{S}_{l})]\nonumber
\end{align}
which implies $\mathcal{H}(\mathcal{S}_{k}^{n_k})=\mathcal{H}(\mathcal{S}_{l}^{n_l})$ and 
\begin{align}
\mathcal{Z}[(\gamma^{-1} \circ \beta \circ \gamma \circ \nabla)^{n_l}(\mathcal{S}_{l})] \simeq \mathcal{Z}[(\gamma^{-1} \circ \beta \circ \gamma \circ \nabla)^{n_h}(\mathcal{S}_{h})]\nonumber
\end{align}
which implies $\mathcal{H}(\mathcal{S}_{l}^{n_l})=\mathcal{H}(\mathcal{S}_{h}^{n_h})$. We deduce
\begin{align}
\mathcal{Z}[(\gamma^{-1}\circ\beta\circ\gamma\circ\nabla)^{n_k}(\mathcal{S}_{k})]\simeq \mathcal{Z}[(\gamma^{-1}\circ\beta\circ \gamma\circ\nabla)^{n_h}(\mathcal{S}_{h})].\nonumber
\end{align}
\bigskip

\begin{proposition}
Let $\mathcal{F}=\{\mathcal{S}_j\}_{j=1}^{\infty}$ be a collection of tuples of $\mathbb{R}[x]$. Let $\mathcal{S}_{l},\mathcal{S}_k \in \mathcal{F}$. Suppose that $\mathcal{Z}[(\gamma^{-1}\circ \beta\circ\gamma\circ\nabla)^{n_k}(\mathcal{S}_{k})]\simeq \mathcal{Z}[(\gamma^{-1}\circ\beta\circ\gamma\circ\nabla)^{n_l}(\mathcal{S}_{l})]$ and that $(\gamma^{-1}\circ\beta\circ\gamma\circ \nabla)^{n_k}(\mathcal{S}_{k})$ is regular. If 
\begin{align}
\mathcal{Z}[(\gamma^{-1} \circ \beta \circ \gamma \circ \nabla)^{n_k+1}(\mathcal{S}_{k})] \simeq \mathcal{Z}[(\gamma^{-1} \circ \beta \circ \gamma \circ \nabla)^{n_l+1}(\mathcal{S}_{l})]\nonumber
\end{align}
then $(\gamma^{-1} \circ \beta \circ \gamma \circ \nabla)^{n_l}(\mathcal{S}_{l})$ is also regular.
\end{proposition}
\bigskip

\begin{proof}
Let $\mathcal{F}=\{\mathcal{S}_j\}_{j=1}^{\infty}$ be a collection of tuples of $\mathbb{R}[x]$. Let $\mathcal{S}_{l},\mathcal{S}_k \in \mathcal{F}$. Suppose that $\mathcal{Z}[(\gamma^{-1}\circ \beta\circ\gamma\circ\nabla)^{n_k}(\mathcal{S}_{k})]\simeq \mathcal{Z}[(\gamma^{-1}\circ\beta\circ\gamma\circ\nabla)^{n_l}(\mathcal{S}_{l})]$. By definition \ref{isomorphism}, we have 
\begin{align}
\mathcal{H}(\mathcal{S}_{k}^{n_k})=\mathcal{H}(\mathcal{S}_{l}^{n_l}).\nonumber
\end{align} 
Since $(\gamma^{-1}\circ\beta\circ\gamma\circ\nabla)^{n_k}(\mathcal{S}_{k})$ is regular, it implies that 
$$
\mathcal{H}(\mathcal{S}_k^{n_k})=\mathcal{H}(\mathcal{S}_{l}^{n_l})>\mathcal{H}(\mathcal{S}_{k}^{n_k+1}).
$$ 
Since 
\begin{align}
\mathcal{Z}[(\gamma^{-1} \circ\beta\circ\gamma\circ\nabla)^{n_k+1}(\mathcal{S}_{k})]\simeq \mathcal{Z}[(\gamma^{-1}\circ\beta\circ \gamma\circ\nabla)^{n_l+1}(\mathcal{S}_{l})]\nonumber
\end{align}
it follows by definition \ref{isomorphism} that 
$$
\mathcal{H}(\mathcal{S}_{l}^{n_l})=\mathcal{H}(\mathcal{S}_k^{n_k})>\mathcal{H}(\mathcal{S}_{k}^{n_k+1})=\mathcal{H}(\mathcal{S}_{l}^{n_l+1})
$$ 
and the claim follows immediately.
\end{proof}

\subsection{Boundary deformation of expansions}

In this section, we introduce and study the notion of \emph{deformation} of the boundary of expansions.

\begin{definition}\label{deform}
Let $\mathcal{F}=\{\mathcal{S}_j\}_{j=1}^{\infty}$ be a collection of tuples of $\mathbb{R}[x]$. We say that the boundary of the expansion $\mathcal{Z}[(\gamma^{-1} \circ \beta \circ \gamma \circ \nabla)^{n_l}(\mathcal{S}_{l})]$ is a \emph{deformation} of the boundary of expansion $\mathcal{Z}[(\gamma^{-1} \circ \beta \circ \gamma \circ \nabla)^{n_h}(\mathcal{S}_{h})]$ if there exists a map 
\begin{align}
\pi:\mathcal{Z}[(\gamma^{-1}\circ\beta\circ\gamma\circ \nabla)^{n_h}(\mathcal{S}_{h})]\longrightarrow \mathcal{Z}[(\gamma^{-1}\circ\beta\circ\gamma\circ\nabla)^{n_l}(\mathcal{S}_{l})]\nonumber 
\end{align}
such that 
$$
\#\mathcal{Z}[(\gamma^{-1} \circ \beta \circ \gamma \circ \nabla)^{n_h}(\mathcal{S}_{h})]>\# \mathcal{Z}[(\gamma^{-1} \circ \beta \circ \gamma \circ \nabla)^{n_l}(\mathcal{S}_{l})]
$$ 
with $\mathcal{H}(\mathcal{S}_{h}^{n_h})=\mathcal{H}(\mathcal{S}_{l}^{n_l})$.
\end{definition}
\bigskip

Here, we relate the properties of the deformation of boundaries of expansions with isomorphism of expansions.

\begin{theorem}
Let $\mathcal{F}=\{\mathcal{S}_j\}_{j=1}^{\infty}$ be a collection of tuples of $\mathbb{R}[x]$ and suppose that $\mathcal{Z}[(\gamma^{-1}\circ\beta\circ\gamma\circ\nabla)^{n_l}(\mathcal{S}_{l})]$ is a deformation of $\mathcal{Z}[(\gamma^{-1} \circ\beta\circ\gamma\circ\nabla)^{n_h}(\mathcal{S}_{h})]$. If $\mathcal{Z}[(\gamma^{-1}\circ\beta\circ\gamma\circ\nabla)^{n_r}(\mathcal{S}_{r})]$ is also a deformation of $\mathcal{Z}[(\gamma^{-1}\circ\beta\circ\gamma\circ\nabla)^{n_h}(\mathcal{S}_{h})]$,
then 
\begin{align}
\mathcal{Z}[(\gamma^{-1}\circ\beta\circ\gamma\circ\nabla)^{n_r}(\mathcal{S}_{r})]\simeq \mathcal{Z}[(\gamma^{-1}\circ\beta\circ \gamma\circ\nabla)^{n_l}(\mathcal{S}_{l})].\nonumber
\end{align}
\end{theorem}
\bigskip

\begin{theorem}
Let $\mathcal{F}=\{\mathcal{S}_j\}_{j=1}^{\infty}$ be a collection of tuples of $\mathbb{R}[x]$ and suppose that $\mathcal{Z}[(\gamma^{-1}\circ\beta\circ\gamma\circ\nabla)^{n_l}(\mathcal{S}_{l})]$ is a deformation of $\mathcal{Z}[(\gamma^{-1} \circ\beta\circ\gamma\circ\nabla)^{n_h}(\mathcal{S}_{h})]$. If $\mathcal{Z}[(\gamma^{-1}\circ\beta\circ\gamma\circ\nabla)^{n_r}(\mathcal{S}_{r})]$ is a deformation of $\mathcal{Z}[(\gamma^{-1} \circ\beta\circ\gamma\circ\nabla)^{n_l}(\mathcal{S}_{l})]$, then $\mathcal{Z}[(\gamma^{-1} \circ \beta \circ \gamma \circ \nabla)^{n_r}(\mathcal{S}_{r})]$ is a deformation of $\mathcal{Z}[(\gamma^{-1} \circ \beta \circ \gamma \circ \nabla)^{n_h}(\mathcal{S}_{h})]$
\end{theorem}

\begin{proof}
Let $\mathcal{F}=\{\mathcal{S}_j\}_{j=1}^{\infty}$ be a collection of tuples of $\mathbb{R}[x]$ and suppose that $\mathcal{Z}[(\gamma^{-1}\circ\beta\circ\gamma\circ\nabla)^{n_l}(\mathcal{S}_{l})]$ is a deformation of $\mathcal{Z}[(\gamma^{-1} \circ\beta\circ\gamma\circ\nabla)^{n_h}(\mathcal{S}_{h})]$. By  definition \ref{deform}, there exist some map 
\begin{align}
\pi_1:\mathcal{Z}[(\gamma^{-1}\circ\beta\circ\gamma\circ \nabla)^{n_h}(\mathcal{S}_{h})]\longrightarrow \mathcal{Z}[(\gamma^{-1}\circ\beta\circ\gamma\circ\nabla)^{n_l}(\mathcal{S}_{l})]\nonumber
\end{align}
such that $\#\mathcal{Z}[(\gamma^{-1}\circ\beta\circ\gamma\circ \nabla)^{n_h}(\mathcal{S}_{h})]>\# \mathcal{Z}[(\gamma^{-1}\circ \beta\circ\gamma\circ \nabla)^{n_l}(\mathcal{S}_{l})]$ with $\mathcal{H}(\mathcal{S}_{l}^{n_l})=\mathcal{H}(\mathcal{S}_{h}^{n_h})$. Since $\mathcal{Z}[(\gamma^{-1}\circ \beta\circ\gamma\circ\nabla)^{n_r}(\mathcal{S}_{r})]$ is a deformation of $\mathcal{Z}[(\gamma^{-1}\circ\beta\circ\gamma \circ\nabla)^{n_l}(\mathcal{S}_{l})]$, there exists a mapping \begin{align}
\pi_2:\mathcal{Z}[(\gamma^{-1}\circ\beta\circ\gamma\circ \nabla)^{n_l}(\mathcal{S}_{l})]\longrightarrow \mathcal{Z}[(\gamma^{-1}\circ\beta\circ\gamma\circ\nabla)^{n_r}(\mathcal{S}_{r})]\nonumber
\end{align}
such that $\#\mathcal{Z}[(\gamma^{-1}\circ\beta\circ\gamma\circ \nabla)^{n_l}(\mathcal{S}_{l})]>\#\mathcal{Z}[(\gamma^{-1}\circ \beta\circ\gamma\circ\nabla)^{n_r}(\mathcal{S}_{r})]$ with $\mathcal{H}(\mathcal{S}_{l}^{n_l})=\mathcal{H}(\mathcal{S}_{r}^{n_r})$. Choosing $\beta=\pi_2 \circ \pi_1$, we deduce
\begin{align}
\beta:\mathcal{Z}[(\gamma^{-1}\circ\beta\circ\gamma\circ \nabla)^{n_h}(\mathcal{S}_{h})]\longrightarrow \mathcal{Z}[(\gamma^{-1}\circ\beta\circ\gamma\circ\nabla)^{n_r}(\mathcal{S}_{r})].\nonumber
\end{align}
Since $\mathcal{H}(\mathcal{S}_{h}^{n_h})=\mathcal{H}(\mathcal{S}_{l}^{n_l})=\mathcal{H}(\mathcal{S}_{r}^{n_r})$ and \begin{align}
\#\mathcal{Z}[(\gamma^{-1}\circ\beta\circ\gamma\circ\nabla)^{n_h}(\mathcal{S}_{h})]&>\#\mathcal{Z}[(\gamma^{-1}\circ\beta\circ \gamma\circ\nabla)^{n_l}(\mathcal{S}_{l})]\nonumber \\&>\# \mathcal{Z}[(\gamma^{-1}\circ\beta\circ\gamma\circ\nabla)^{n_r}(\mathcal{S}_{r})]\nonumber
\end{align}
the claim follows.
\end{proof}

\subsection{Overlapping and non-overlapping expansions}

In this section, we study the concept of \emph{overlapping} of expansions. 

\begin{definition}
Let $\mathcal{F}=\{\mathcal{S}_j\}_{j=1}^{\infty}$ be a collection  of tuples of $\mathbb{R}[x]$. The expansions $(\gamma^{-1}\circ \beta\circ\gamma\circ\nabla)^{n_l}(\mathcal{S}_{l})$ and $(\gamma^{-1}\circ\beta\circ\gamma\circ\nabla)^{n_k}(\mathcal{S}_{k})$ are said to be \emph{overlapping} if \begin{align}
\mathcal{Z}[(\gamma^{-1}\circ\beta\circ\gamma\circ\nabla)^{n_l}(\mathcal{S}_{l})]\bigcap \mathcal{Z}[(\gamma^{-1}\circ\beta\circ \gamma\circ\nabla)^{n_k}(\mathcal{S}_{k})]\neq \emptyset. \nonumber
\end{align}
We denote this overlapping region by $\mathcal{O}(\mathcal{S}_{l}^{n_l}, \mathcal{S}_{k}^{n_k})$. We call \begin{align}
\mathcal{D}^l[\mathcal{O}(\mathcal{S}_{l}^{n_l}, \mathcal{S}_{k}^{n_k})]=\frac{\#\mathcal{O}(\mathcal{S}_{l}^{n_l},\mathcal{S}_{k}^{n_k})}{\# \mathcal{Z}[(\gamma^{-1}\circ\beta\circ\gamma\circ\nabla)^{n_l}(\mathcal{S}_{l})]}\nonumber
\end{align}
the \emph{density} of the overlapping region relative to the expansion $(\gamma^{-1}\circ\beta\circ\gamma\circ\nabla)^{n_l}(\mathcal{S}_{l})$ and 
\begin{align}
\mathcal{D}^k[\mathcal{O}(\mathcal{S}_{l}^{n_l}, \mathcal{S}_{k}^{n_k})]=\frac{\#\mathcal{O}(\mathcal{S}_{l}^{n_l}, \mathcal{S}_{k}^{n_k})}{\#\mathcal{Z}[(\gamma^{-1}\circ\beta\circ \gamma\circ\nabla)^{n_k}(\mathcal{S}_{k})]}\nonumber
\end{align}
the \emph{density} of the overlapping region relative to the expansion $(\gamma^{-1}\circ\beta\circ\gamma\circ\nabla)^{n_k}(\mathcal{S}_{k})$.
\end{definition}
\bigskip

\section{Associate expansions}\label{sec:Associate expansions}

In this section, we introduce the notion of \emph{associate} of expansion. We study how this property interacts with the notion of isomorphism.

\begin{definition}\label{associate}
Let $\mathcal{F}=\{\mathcal{S}_j\}_{j=1}^{\infty}$ be a collection of tuples of $\mathbb{R}[x]$. Let $\mathcal{S}_k,\mathcal{S}_l\in \mathcal{F}$. We say that the expansion $(\gamma^{-1}\circ\beta \circ\gamma\circ\nabla)^{n_k}(\mathcal{S}_{k})$ is an \emph{associate} of the expansion $(\gamma^{-1}\circ\beta\circ \gamma\circ\nabla)^{n_l}(\mathcal{S}_{l})$ if for each \begin{align}
\mathcal{S}_0\in \mathcal{Z}[(\gamma^{-1}\circ\beta\circ\gamma \circ\nabla)^{n_l}(\mathcal{S}_{l})]\nonumber
\end{align}
there exists an $\mathcal{S}_1\in \mathcal{Z}[(\gamma^{-1}\circ \beta\circ\gamma\circ\nabla)^{n_k}(\mathcal{S}_{k})]$ such that $\mathcal{S}_0=m\mathcal{S}_1$ for some $m\in \mathbb{N}$.
\end{definition}

\begin{proposition}
Let $\mathcal{S}_k, \mathcal{S}_l\in\mathcal{F}=\{\mathcal{S}_j\}_{j=1}^{\infty}$ and suppose that $ \mathcal{Z}[(\gamma^{-1}\circ\beta\circ\gamma\circ\nabla)^{n_k}(\mathcal{S}_{k})]\simeq \mathcal{Z}[(\gamma^{-1}\circ\beta\circ \gamma\circ\nabla)^{n_l}(\mathcal{S}_{l})]$. Suppose that $\mathcal{S}_a\neq r\mathcal{S}_b$ ($r\in \mathbb{N}$) for $\mathcal{S}_a,\mathcal{S}_b\in \mathcal{Z}[(\gamma^{-1}\circ \beta\circ\gamma\circ\nabla)^{n_k}(\mathcal{S}_{k})]$. If  $(\gamma^{-1}\circ\beta\circ\gamma\circ\nabla)^{n_k}(\mathcal{S}_{k})$ is an associate of $(\gamma^{-1}\circ\beta \circ\gamma\circ\nabla)^{n_l}(\mathcal{S}_{l})$, then 
\begin{align}
\mathcal{Z}[(\gamma^{-1}\circ\beta\circ\gamma\circ\nabla)^{n_l}(\mathcal{S}_{l})]=\mathcal{Z}[(\gamma^{-1}\circ\beta\circ\gamma \circ\nabla)^{n_k}(\mathcal{S}_{k})].\nonumber
\end{align}
\end{proposition}

\begin{proof}
Let $\mathcal{S}_k, \mathcal{S}_l\in\mathcal{F}=\{\mathcal{S}_j\}_{j=1}^{\infty}$ and suppose that $\mathcal{S}_a\neq r\mathcal{S}_b$ ($r\in \mathbb{N}$) for $\mathcal{S}_a,\mathcal{S}_b\in \mathcal{Z}[(\gamma^{-1}\circ \beta\circ\gamma\circ\nabla)^{n_k}(\mathcal{S}_{k})]$. Suppose that $(\gamma^{-1}\circ\beta\circ\gamma\circ\nabla)^{n_k}(\mathcal{S}_{k})$ is an associate of $(\gamma^{-1}\circ\beta \circ\gamma\circ\nabla)^{n_l}(\mathcal{S}_{l})$, then by definition \ref{associate}, for each $\mathcal{S}_0\in \mathcal{Z}[(\gamma^{-1}\circ\beta\circ\gamma\circ\nabla)^{n_k}(\mathcal{S}_{k})]$ there exists a unique $\mathcal{S}_1\in \mathcal{Z}[(\gamma^{-1}\circ\beta\circ\gamma\circ\nabla)^{n_l}(\mathcal{S}_{l})]$ such that $\mathcal{S}_0=m\mathcal{S}_1$ for some $m\in \mathbb{N}$. Since 
\begin{align}
\mathcal{Z}[(\gamma^{-1}\circ\beta\circ\gamma\circ\nabla)^{n_k}(\mathcal{S}_{k})]\simeq \mathcal{Z}[(\gamma^{-1}\circ\beta\circ \gamma\circ\nabla)^{n_l}(\mathcal{S}_{l})]\nonumber
\end{align}
it follows from the definition \ref{isomorphism}
\begin{align}
\mathcal{H}(\mathcal{S}_{l}^{n_l})=\mathcal{H}(\mathcal{S}_{k}^{n_k}).\nonumber
\end{align}
We deduce
\begin{align}
\mathcal{H}(\mathcal{S}_{l}^{n_l})&=\sum \limits_{\mathcal{S}_a\in \mathcal{B}^{n_l}}||\mathcal{S}_a||\nonumber \\&=\sum \limits_{\substack{\mathcal{S}_b\in \mathcal{B}^{n_k}\\\mathcal{S}_a=m_b\mathcal{S}_b}}||m_b\mathcal{S}_b||\nonumber \\&=\sum \limits_{\substack{\mathcal{S}_b\in \mathcal{B}^{n_k}\\\mathcal{S}_a=m_b\mathcal{S}_b}}m_b||\mathcal{S}_b||\nonumber \\&=\sum \limits_{\mathcal{S}_b\in \mathcal{B}^{n_l}}||\mathcal{S}_b||.\nonumber
\end{align}
We take $m_b=1$ and deduce the claim.
\end{proof}

\subsection{Sub-expansions}

In this section, we introduce the concept of \emph{sub-expansions} of an expansion.

\begin{definition}\label{sub}
Let $(\gamma^{-1}\circ\beta\circ\gamma\circ \nabla)^m(\mathcal{S}_a)$ and $(\gamma^{-1}\circ\beta\circ\gamma \circ\nabla)^n(\mathcal{S}_b)$ be any two expansions with $m<n$. We say that $(\gamma^{-1}\circ\beta\circ\gamma\circ \nabla)^m(\mathcal{S}_a)$ is a \emph{sub-expansion} of the expansion $(\gamma^{-1}\circ\beta\circ\gamma\circ \nabla)^n(\mathcal{S}_b)$, denoted by 
\begin{align}
(\gamma^{-1}\circ\beta\circ\gamma\circ\nabla)^m(\mathcal{S}_a) \leq (\gamma^{-1}\circ\beta\circ\gamma\circ \nabla)^n(\mathcal{S}_b)\nonumber
\end{align}
if there exist some $j\geq 1$ such that $(\gamma^{-1}\circ\beta \circ\gamma\circ\nabla)^m(\mathcal{S}_a)=(\gamma^{-1}\circ\beta \circ\gamma\circ\nabla)^{m+j}(\mathcal{S}_b)$. We say that the sub-expansion is \emph{proper} if $m+j=n$. We denote this proper sub-expansion by 
\begin{align}
(\gamma^{-1}\circ\beta\circ\gamma\circ\nabla)^m(\mathcal{S}_a)< (\gamma^{-1}\circ\beta\circ\gamma\circ \nabla)^n(\mathcal{S}_b)\nonumber
\end{align}
\end{definition}
\bigskip

Here, we show that the regularity condition on an expansion can be localized as well as extended through expansions. 

\begin{proposition}\label{sub-expansion}
Let $(\gamma^{-1}\circ\beta\circ\gamma\circ \nabla)^m(\mathcal{S}_a)<(\gamma^{-1}\circ\beta\circ\gamma\circ \nabla)^n(\mathcal{S}_b)$, which means $(\gamma^{-1}\circ\beta\circ\gamma\circ \nabla)^m(\mathcal{S}_a)$ is a proper sub-expansion of the expansion $(\gamma^{-1}\circ\beta\circ\gamma\circ \nabla)^n(\mathcal{S}_b)$. The expansion  $(\gamma^{-1}\circ\beta\circ\gamma\circ \nabla)^m(\mathcal{S}_a)$ is regular if and only if $(\gamma^{-1} \circ\beta\circ\gamma\circ\nabla)^n(\mathcal{S}_b)$ is regular.
\end{proposition}

\begin{proof}
Let $(\gamma^{-1}\circ\beta\circ\gamma\circ \nabla)^m(\mathcal{S}_a)<(\gamma^{-1}\circ\beta\circ\gamma\circ \nabla)^n(\mathcal{S}_b)$ and suppose that the expansion $(\gamma^{-1}\circ \beta\circ\gamma\circ\nabla)^m(\mathcal{S}_a)$ is regular. We deduce $\mathcal{H}(\mathcal{S}_a^{m})>\mathcal{H}(\mathcal{S}_a^{m+1})$ for some $1\leq m \leq \mathrm{deg}(\mathcal{S}_a)-2$. By definition \ref{sub}, it follows that there exist some $j\geq 1$ such that we can write  $(\gamma^{-1}\circ \beta\circ\gamma\circ\nabla)^m(\mathcal{S}_a)=(\gamma^{-1}\circ \beta\circ\gamma\circ\nabla)^{m+j}(\mathcal{S}_b)$. Since the expansion is proper, it implies $m+j=n$ and we have 
\begin{align}
(\gamma^{-1}\circ\beta\circ\gamma\circ\nabla)^m(\mathcal{S}_a)=(\gamma^{-1}\circ\beta\circ\gamma\circ\nabla)^{n}(\mathcal{S}_b).\nonumber
\end{align}
It follows that $\mathcal{H}(\mathcal{S}_a^m)=\mathcal{H}(\mathcal{S}_b^n)$. Since 
\begin{align}
(\gamma^{-1}\circ\beta\circ\gamma\circ\nabla)^{m+1}(\mathcal{S}_a)=(\gamma^{-1}\circ\beta\circ\gamma\circ\nabla)^{n+1}(\mathcal{S}_b)\nonumber
\end{align}
the regularity of the expansion $(\gamma^{-1}\circ\beta\circ \gamma\circ\nabla)^n(\mathcal{S}_b)$ also follows. The converse can also be shown using a similar argument.
\end{proof}

\section{Distribution of the boundary points of expansion}\label{sec:Boundary point distribution}

In this section, we study the distribution of the boundary points of any phase of expansion. We first  introduce the notion of an integration of polynomials along the boundaries of various phases of expansion, which we then use as our main tool.

\begin{definition}\label{special integral}
Let $f(x)=a_nx^n+a_{n-1}x^{n-1}+\cdots +a_1x+a_0$ be a polynomial of degree $n$. We call the tuple 
\begin{align}
\mathcal{S}_f&:=(a_nx^n,a_{n-1}x^{n-1},\ldots, a_1x+a_0)\nonumber \\&=(g_1(x), g_2(x),\ldots,g_n(x))\nonumber
\end{align}
the tuple representation of $f$. By the integral of $f(x)$ along the boundary of the $m^{th}$ phase expansion of $\mathcal{S}_f$, we mean the formal integral
\begin{align}
\int \limits_{\substack{\mathcal{B}^m\\m<n}}f(t)dt:=\sum \limits_{i=1}^{\# \mathcal{B}^m-1}\sum \limits_{\substack{\mathcal{S}_i, \mathcal{S}_{i+1}\in \mathcal{B}^m\\||\mathcal{S}_i||<||\mathcal{S}_{i+1}||}}\overrightarrow{O\Delta_{\mathcal{S}_i, \mathcal{S}_{i+1}}(\mathcal{S}_f)}\cdot \overrightarrow{O\mathcal{S}_e}\nonumber
\end{align}
where 
\begin{align}
\Delta_{\mathcal{S}_i,\mathcal{S}_{i+1}}(\mathcal{S}_f)=\bigg(\int \limits_{a_1}^{b_1}g_1(x)dx,\int \limits_{a_2}^{b_2}g_2(x)dx, \ldots,\int \limits_{a_n}^{b_n}g_n(x)dx\bigg)\nonumber
\end{align}
and where $\overrightarrow{O\Delta_{\mathcal{S}_i,\mathcal{S}_{i+1}}(\mathcal{S}_f)}$ and $\overrightarrow{O\mathcal{S}_e}$ are the position vectors of $\Delta_{\mathcal{S}_i,\mathcal{S}_{i+1}}(\mathcal{S}_f)$ and $\mathcal{S}_e$ respectively with $\mathcal{S}_i=(a_1,a_2,\ldots, a_n)$ and $\mathcal{S}_{i+1}=(b_1,b_2,\ldots, b_n)$.
\end{definition}
\bigskip

In practice, it may be difficult to assert the local distribution of the boundary points of an expansion. However, we can show that if we shrink the space bounded by the boundary of an expansion, then the boundary points must be closely packed together in some sense. We use the integral proposed by \ref{special integral} as a black box.

\begin{theorem}\label{area}
Let $f(x)=a_nx^n+a_{n-1}x^{n-1}+\cdots+a_1x+a_0$ be a polynomial of degree $n$. Then 
\begin{align}
\bigg |\bigg |\int \limits_{\substack{\mathcal{B}^m\\m<n}}f(t)dt\bigg |\bigg |>\epsilon\nonumber
\end{align}
for some $\epsilon>0$ if and only if $||\mathcal{S}_i-\mathcal{S}_{i+1}||>\delta$ for $\delta:=\delta(n)>0$ for some \begin{align}
\mathcal{S}_i\in \mathcal{Z}[(\gamma^{-1}\circ\beta\circ\gamma \circ\nabla)^{m}(\mathcal{S}_{f})]\nonumber
\end{align} 
with $1\leq i\leq \#\mathcal{B}^m-1$ and $||\mathcal{S}_i-\mathcal{S}_{i+1}||<||\mathcal{S}_i-\mathcal{S}_j||$ for $i+1\neq j$.
\end{theorem}
\bigskip

\begin{proof}
Let $f(x)=a_nx^n+a_{n-1}x^{n-1}+\cdots+a_1x+a_0\in \mathbb{R}[x]$ be a polynomial of degree $n$ and suppose that 
\begin{align}
\bigg |\bigg |\int \limits_{\substack{\mathcal{B}^m\\m<n}}f(t)dt\bigg |\bigg |>\epsilon\nonumber
\end{align}
for some $\epsilon>0$. Using repeated application of the triangle inequality, we find that 
\begin{align}
\bigg |\bigg |\int \limits_{\substack{\mathcal{B}^m\\m<n}}f(t)dt\bigg |\bigg |& \leq \sum \limits_{i=1}^{\# \mathcal{B}^m-1}\sum \limits_{\substack{\mathcal{S}_i, \mathcal{S}_{i+1}\in \mathcal{B}^m\\||\mathcal{S}_i||<||\mathcal{S}_{i+1}||}}||\overrightarrow{O\Delta_{\mathcal{S}_i, \mathcal{S}_{i+1}}(\mathcal{S}_f)}||||\overrightarrow{O\mathcal{S}_e}||\nonumber \\&=\sqrt{n}\sum \limits_{i=1}^{\# \mathcal{B}^m-1}\sum \limits_{\substack{\mathcal{S}_i, \mathcal{S}_{i+1}\in \mathcal{B}^m\\||\mathcal{S}_i||<||\mathcal{S}_{i+1}||}}||\overrightarrow{O\Delta_{\mathcal{S}_i, \mathcal{S}_{i+1}}(\mathcal{S}_f)}||\nonumber \\& \leq (\# \mathcal{B}^m-1)\sqrt{n}\mathrm{max}\left \{||\overrightarrow{O\Delta_{\mathcal{S}_i, \mathcal{S}_{i+1}}(\mathcal{S}_f)}||\right \}_{\substack{i=1\\||\mathcal{S}_i||<||\mathcal{S}_{i+1}||}}^{\# (\mathcal{B}^m-1)}.\nonumber 
\end{align}
Since the inequality \begin{align}||\overrightarrow{O\Delta_{\mathcal{S}_i, \mathcal{S}_{i+1}}(\mathcal{S}_f)}||&=\sqrt{|\int \limits_{a_1}^{b_1}g_1dx|^2+\cdots +|\int \limits_{a_n}^{b_n}g_ndx|^2}\nonumber \\&\leq M\sqrt{|a_1-b_1|^2+\cdots +|a_n-b_n|^2}\nonumber
\end{align}
holds, it follows that there exist some $\mathcal{S}_i, \mathcal{S}_{i+1}\in \mathcal{Z}[(\gamma^{-1}\circ\beta\circ \gamma\circ\nabla)^{m}(\mathcal{S}_{f})]$ with $||\mathcal{S}_i-\mathcal{S}_{i+1}||<||\mathcal{S}_i-\mathcal{S}_j||$ for all $i+1\neq j$. It follows that for some closest pair of boundary point, the inequality 
\begin{align}
\frac{\epsilon}{(\# \mathcal{B}^m-1)M\sqrt{n}}<\sqrt{|a_1-b_1|^2+\cdots +|a_n-b_n|^2}\nonumber
\end{align}
holds. Thus, we get  $||\mathcal{S}_i-\mathcal{S}_{i+1}||>\delta$ by choosing $\delta =\frac{\epsilon}{(\# \mathcal{B}^m-1)M\sqrt{n}}$. Conversely, suppose that there exist some closest boundary points $\mathcal{S}_i,\mathcal{S}_{i+1}\in \mathcal{Z}[(\gamma^{-1}\circ\beta\circ\gamma\circ\nabla)^{m}(\mathcal{S}_{f})]$ such that 
\begin{align}
||\mathcal{S}_i-\mathcal{S}_{i+1}||>\delta \nonumber
\end{align}
for some $\delta:=\delta(n)>0$. It follows that $\sqrt{|a_1-b_1|^2+\cdots +|a_n-b_n|^2}>\delta $. Choosing $R=\mathrm{min}\left \{|g_i(x)|:x\in [a_i,b_i] \right \}_{i=1}^{n}$, we find that 
\begin{align}
||\overrightarrow{O\Delta_{\mathcal{S}_i, \mathcal{S}_{i+1}}(\mathcal{S}_f)}||&=\sqrt{|\int \limits_{a_1}^{b_1}g_1dx|^2+\cdots +|\int \limits_{a_n}^{b_n}g_ndx|^2}\nonumber \\&\geq  R\sqrt{|a_1-b_1|^2+\cdots +|a_n-b_n|^2}\nonumber \\&=\delta R.\nonumber
\end{align}
We deduce
\begin{align}
\sum \limits_{i=1}^{\# \mathcal{B}^m-1}\sum \limits_{\substack{\mathcal{S}_i, \mathcal{S}_{i+1}\in \mathcal{B}^m\\||\mathcal{S}_i||<||\mathcal{S}_{i+1}||}}\overrightarrow{O\Delta_{\mathcal{S}_i, \mathcal{S}_{i+1}}(\mathcal{S}_f)}\cdot \overrightarrow{O\mathcal{S}_e}>&\sum \limits_{i=1}^{\# \mathcal{B}^m-1}\sum \limits_{\substack{\mathcal{S}_i, \mathcal{S}_{i+1}\in \mathcal{B}^m\\||\mathcal{S}_i||<||\mathcal{S}_{i+1}||}}\delta R||\overrightarrow{O\mathcal{S}_e}||\cos \alpha\nonumber \\&=\delta (\# \mathcal{B}^m-1)R\sqrt{n}\cos \alpha \nonumber
\end{align}
where $\alpha$ is the angle between the vectors $\overrightarrow{O\Delta_{\mathcal{S}_i,\mathcal{S}_{i+1}}(\mathcal{S}_f)}$ and $\overrightarrow{O\mathcal{S}_e}$. It follows that 
\begin{align}
\bigg |\bigg |\int \limits_{\substack{\mathcal{B}^m\\m<n}}f(t)dt\bigg |\bigg |>C(n)\delta (\# \mathcal{B}^m-1)R\sqrt{n}|\cos \alpha| \nonumber
\end{align}
for some $C(n)>0$. The result follows by taking 
\begin{align}
\delta:=\frac{\epsilon}{C(n)(\# \mathcal{B}^m-1)R\sqrt{n}|\cos \alpha|}.\nonumber
\end{align}
\end{proof}
\bigskip

Theorem \ref{area} partially solves Conjecture \ref{distribution}. Indeed, the space occupied by the boundaries increase with expansion. Thus, Theorem \ref{area} in the affirmative suggests that we can use the area as a yardstick to determine the distribution of the points of any given phase of expansion. We apply this new tool to study the mass of the corresponding phases of expansions as follows.

\begin{corollary}
Let $f(x)\in \mathbb{R}[x]$ be a polynomial of degree $n$. If \begin{align}
\bigg |\bigg |\int \limits_{\substack{\mathcal{B}^m\\m<n}}f(t)dt\bigg |\bigg |<1\nonumber
\end{align}
with $||\mathcal{S}_i||<1$ for some $\mathcal{S}_{i}\in \mathcal{Z}[(\gamma^{-1}\circ\beta\circ\gamma\circ\nabla)^{m}(\mathcal{S}_{f})]$, then $\mathcal{H}(\mathcal{S}_f^{m})<\epsilon$ for some $\epsilon:=\epsilon(n)>0$.
\end{corollary}

\begin{proof}
Let $f(x):=a_nx^n+a_{n-1}x^{n-1}+\cdots+a_1x+a_0\in \mathbb{R}[x]$ be a polynomial of degree $n$. Suppose that 
\begin{align}
\bigg |\bigg |\int \limits_{\substack{\mathcal{B}^m\\m<n}}f(t)dt\bigg |\bigg |<1.\nonumber
\end{align}
By applying Theorem \ref{area}, we get
\begin{align}
||\mathcal{S}_i-\mathcal{S}_{i+1}||<1\nonumber
\end{align}
for all $1\leq i \leq \#\mathcal{B}^m-1$. Since $||\mathcal{S}_i||<1$ for some $1\leq i\leq \#\mathcal{B}^m$, it follows that $||\mathcal{S}_j||<1$ for all $1\leq j \leq \#\mathcal{B}^m$ and the claim follows immediately.
\end{proof}

\begin{remark}
With this new tool available, we can now establish a uniform version of the diminishing state of the mass of phases of an expansion whose phase boundaries are produced from the expansion of some part of the boundary. We state this result at the compromise of taking sufficiently small boundaries.
\end{remark}

\begin{theorem}\label{tool}
Let $f(x):=a_nx^n+a_{n-1}x^{n-1}+\cdots+a_1x+a_0\in \mathbb{R}[x]$ be a polynomial of degree $n$. Suppose that 
\begin{align}
\sum \limits_{m=1}^{n-1}\bigg |\bigg |\int \limits_{\substack{\mathcal{B}^m\\m<n}}f(t)dt\bigg |\bigg |<1.\nonumber
\end{align}
If 
\begin{align}
\bigcap \limits_{m=1}^{n-1}\mathcal{Z}[(\gamma^{-1}\circ\beta \circ\gamma\circ\nabla)^{m}(\mathcal{S}_{f})]\neq \emptyset\nonumber
\end{align}
then $\mathcal{H}(\mathcal{S}_f^m)>\mathcal{H}(\mathcal{S}_f^{m+1})$ for all $1\leq m \leq n-1=deg(\mathcal{S}_f)-1$.
\end{theorem}

\begin{proof}
Let $f(x):=a_nx^n+a_{n-1}x^{n-1}+\cdots+a_1x+a_0\in \mathbb{R}[x]$ be a polynomial of degree $n$ and suppose that 
\begin{align}
\sum \limits_{m=1}^{n-1}\bigg |\bigg |\int \limits_{\substack{\mathcal{B}^m\\m<n}}f(t)dt\bigg |\bigg |<1.\nonumber
\end{align}
By Theorem \ref{area}, we get
\begin{align}
||\mathcal{S}_i-\mathcal{S}_{i+1}||<1\nonumber
\end{align}
for $\mathcal{S}_i\in \mathcal{Z}[(\gamma^{-1}\circ\beta\circ \gamma\circ\nabla)^{m}(\mathcal{S}_{f})]$ for all $1\leq i\leq \# \mathcal{B}^m-1$ for each $1\leq m\leq n-1=\mathrm{deg}(\mathcal{S}_f)-1$. It follows that $||\mathcal{S}_i||\approx ||\mathcal{S}_{i+1}||$. Since 
\begin{align}
\bigcap \limits_{m=1}^{n-1}\mathcal{Z}[(\gamma^{-1}\circ\beta \circ\gamma\circ\nabla)^{m}(\mathcal{S}_{f})]\neq \emptyset\nonumber
\end{align}
it implies that the boundary points of each phase of expansions are of size comparable to the size of the boundary points of other phases of expansions. This fact completes the proof, since the boundary points decrease with successive phases of expansions.
\end{proof}
\bigskip

Here, we show that this special integral can also be used as criterion for determining the sub-expansions of an expansion, provided it is small enough.

\begin{theorem}\label{areasubgroup}
Let $f(x):=a_nx^n+a_{n-1}x^{n-1}+\cdots+a_1x+a_0\in \mathbb{R}[x]$ be a polynomial of degree $n$. If 
\begin{align}
\bigg|\bigg |\int \limits_{\substack{\mathcal{B}^{m_1}\\m_1<n}}f(t)dt\bigg |\bigg |<\bigg |\bigg |\int \limits_{\substack{\mathcal{B}^{m_2}\\m_2<n}}f(t)dt\bigg |\bigg |<1\nonumber
\end{align}
then $(\gamma^{-1}\circ\beta\circ\gamma\circ\nabla)^{m_1}(\mathcal{S}_{f})<(\gamma^{-1}\circ\beta\circ\gamma\circ \nabla)^{m_2}(\mathcal{S}_{f})$.
\end{theorem}

\begin{proof}
Let $f(x):=a_nx^n+a_{n-1}x^{n-1}+\cdots+a_1x+a_0\in \mathbb{R}[x]$ be a polynomial of degree $n$, and let 
\begin{align}
\bigg|\bigg |\int \limits_{\substack{\mathcal{B}^{m_1}\\m_1<n}}f(t)dt\bigg |\bigg |<\bigg |\bigg |\int \limits_{\substack{\mathcal{B}^{m_2}\\m_2<n}}f(t)dt\bigg |\bigg |<1.\nonumber
\end{align}
Suppose on the contrary that 
\begin{align}
(\gamma^{-1}\circ\beta\circ\gamma\circ\nabla)^{m_2}(\mathcal{S}_{f})<(\gamma^{-1}\circ\beta\circ\gamma\circ \nabla)^{m_1}(\mathcal{S}_{f}).\nonumber
\end{align}
By Theorem \ref{area}, we have $||\mathcal{S}_{i}-\mathcal{S}_{i+1}||<||\mathcal{S}_j-\mathcal{S}_{j+1}||<1$ with $\mathcal{S}_i,\mathcal{S}_{i+1}\in \mathcal{Z}[(\gamma^{-1}\circ \beta\circ\gamma\circ\nabla)^{m_2}(\mathcal{S}_{f})]$ and $\mathcal{S}_{j},\mathcal{S}_{j+1}\in \mathcal{Z}[(\gamma^{-1} \circ\beta\circ\gamma\circ\nabla)^{m_1}(\mathcal{S}_{f})]$ such that 
\begin{align}
||\mathcal{S}_i-\mathcal{S}_{i+1}||=\mathrm{inf}\left \{||\mathcal{S}_i-\mathcal{S}_{k}||:\mathcal{S}_k\in \mathcal{Z}[(\gamma^{-1}\circ\beta\circ\gamma\circ\nabla)^{m_2}(\mathcal{S}_{f})]\right \}\nonumber
\end{align} 
and 
\begin{align}
||\mathcal{S}_j-\mathcal{S}_{j+1}||=\mathrm{inf}\left \{||\mathcal{S}_j-\mathcal{S}_{l}||:\mathcal{S}_l\in \mathcal{Z}[(\gamma^{-1}\circ\beta\circ\gamma\circ\nabla)^{m_1}(\mathcal{S}_{f})]\right\}.\nonumber
\end{align}
It follows that the boundary points of the two distinct boundaries of expansions are of size comparable to each other, up to a very small error. Since points on the boundary become sparse for higher phase of expansions, we get
\begin{align}
\bigg |\bigg |\int \limits_{\substack{\mathcal{B}^{m_2}\\m_2<n}}f(t)dt\bigg |\bigg |\leq \bigg|\bigg |\int \limits_{\substack{\mathcal{B}^{m_1}\\m_1<n}}f(t)dt\bigg |\bigg |\nonumber
\end{align}
which violates the hypothesis.
\end{proof}

\begin{remark}
It is important to state that this pass is somewhat easy; the pass from the area bounded by the boundary of expansion to information about the sub-expansions of an expansion when we allow for only sufficiently small areas. The reverse, on the other hand, may not necessarily be true. 
\end{remark}

\subsection{Interior and exterior points of expansion}

We devote this section to study the \emph{interior} and the \emph{exterior} points of the boundary of expansions. We also introduce the concept of the \emph{neighbourhood} of the boundary of an expansion.

\begin{definition}\label{interiorexterior}
Let $(\gamma^{-1}\circ\beta\circ\gamma\circ\nabla)^{m}(\mathcal{S})$ be an expansion for $1\leq m \leq \mathrm{deg}(\mathcal{S})-1$ with the boundary $\mathcal{Z}[(\gamma^{-1}\circ \beta\circ\gamma\circ\nabla)^{m}(\mathcal{S})]$. By the \emph{interior} of the expansion, denoted by $\mathrm{Int}[(\gamma^{-1} \circ \beta \circ \gamma \circ \nabla)^{m}(\mathcal{S})]$, we mean the set of points
\begin{align}
\mathrm{Int}[(\gamma^{-1}\circ\beta\circ\gamma\circ\nabla)^{m}(\mathcal{S})]&:=\left \{\mathcal{S}_a\in \mathbb{R}^{n}:||\mathcal{S}_a||<||\mathcal{S}_j||~\mathrm{or}~||\mathcal{S}_a||>||\mathcal{S}_j||,~\mathrm{for}~\mathrm{most}~\mathcal{S}_j\in \mathcal{B}^m\right \}.\nonumber
\end{align}
The points in the interior of expansion are called the \emph{interior points} of expansion. The interior is said to be an \emph{upper interior}, denoted by 
$$
\mathrm{Int}_{u}[(\gamma^{-1} \circ\beta\circ\gamma\circ\nabla)^{m}(\mathcal{S})]
$$ 
if each interior point belongs to the set
\begin{align}
\mathrm{Int}_{u}[(\gamma^{-1}\circ\beta\circ\gamma\circ\nabla)^{m}(\mathcal{S})]&:=\left\{\mathcal{S}_a\in \mathbb{R}^{n}:||\mathcal{S}_a||>||\mathcal{S}_j||,~\mathrm{for}~\mathrm{most}~\mathcal{S}_j\in \mathcal{B}^m\right \}.\nonumber
\end{align}
Otherwise, it is a lower interior, denoted by 
\begin{align}
\mathrm{Int}_{l}[(\gamma^{-1}\circ\beta\circ\gamma\circ\nabla)^{m}(\mathcal{S})]&:=\left \{\mathcal{S}_a\in \mathbb{R}^{n}:||\mathcal{S}_a||<||\mathcal{S}_j||,~\mathrm{for}~\mathrm{most}~\mathcal{S}_j\in \mathcal{B}^m\right\}.\nonumber
\end{align}
Similarly the exterior of an expansion, denoted by
$$
\mathrm{Ext}[(\gamma^{-1} \circ \beta \circ \gamma \circ \nabla)^{m}(\mathcal{S})]
$$ 
is the set of points
\begin{align}
\mathrm{Ext}[(\gamma^{-1} \circ \beta \circ \gamma \circ \nabla)^{m}(\mathcal{S})]&=\left \{\mathcal{S}_a\in \mathbb{R}^{n}:||\mathcal{S}_a||>||\mathcal{S}_j||~\mathrm{or}~||\mathcal{S}_a||<||\mathcal{S}_j||,~\mathrm{for}~\mathrm{all}~\mathcal{S}_j\in \mathcal{B}^m\right \}.\nonumber
\end{align}
A similar characterization also holds for exterior of an expansion as does the interior of an expansion.
\end{definition}
\bigskip

We show that the interior of an expansion can be used to determine the mass of an expansion. We use the proposed integral as our main tool. 

\begin{proposition}\label{interiorexterior mass}
Let $f(x)=a_nx^n+a_{n-1}x^{n-1}+\cdots+a_1x+a_0$ be a polynomial of degree $n$, and let $\mathcal{S}_f$ be the tuple representation of $f$. Suppose that 
\begin{align}
\bigg |\bigg |\int \limits_{\substack{\mathcal{B}^{m}\\m<n}}f(t)dt\bigg |\bigg |<1\nonumber
\end{align}
and 
\begin{align}
\mathrm{Int}_{l}[(\gamma^{-1} \circ \beta \circ \gamma \circ \nabla)^{m}(\mathcal{S}_f)]&:=\left \{\mathcal{S}_a\in \mathbb{R}^{n}:||\mathcal{S}_a||<||\mathcal{S}_j||,~\mathrm{for}~\mathrm{most}~\mathcal{S}_j\in \mathcal{B}^m\right \}\nonumber \\&=\mathrm{Int}[(\gamma^{-1} \circ\beta\circ\gamma\circ\nabla)^{m}(\mathcal{S}_f)].\nonumber
\end{align}
If 
\begin{align}
\sum \limits_{\substack{\mathcal{S}_a\in \mathcal{R}\\\mathcal{R}\subset\mathrm{Int}[(\gamma^{-1}\circ \beta\circ\gamma\circ\nabla)^{m}(\mathcal{S}_f)]\\\# \mathcal{R}=\#\mathcal{B}^m}}||\mathcal{S}_a||>\epsilon\nonumber
\end{align}
for some $\epsilon>0$ and $||\mathcal{S}_a-\mathcal{S}_j||\geq 1$, then $\mathcal{H}(\mathcal{S}_f^{m})>\epsilon$.
\end{proposition}

\begin{proof}
Let $f(x)=a_nx^n+a_{n-1}x^{n-1}+\cdots+a_1x+a_0$ be a polynomial of degree $n$, and let $\mathcal{S}_f$ be the tuple representation of $f$. Suppose that
\begin{align}
\bigg |\bigg |\int \limits_{\substack{\mathcal{B}^{m}\\m<n}}f(t)dt\bigg |\bigg |<1.\nonumber
\end{align}
By Theorem \ref{area}, we deduce $||\mathcal{S}_i||\approx ||\mathcal{S}_j||$ for any pair of points $\mathcal{S}_i,\mathcal{S}_j\in \mathcal{Z}[(\gamma^{-1}\circ \beta\circ\gamma\circ\nabla)^{m}(\mathcal{S}_f)]$. Since \begin{align}
\mathrm{Int}_{l}[(\gamma^{-1}\circ\beta\circ\gamma\circ\nabla)^{m}(\mathcal{S}_f)]&:=\left \{\mathcal{S}_a\in \mathbb{R}^{n}:||\mathcal{S}_a||<||\mathcal{S}_j||,~\mathrm{for}~\mathrm{most}~\mathcal{S}_j\in \mathcal{B}^m\right \}\nonumber \\&=\mathrm{Int}[(\gamma^{-1}\circ\beta\circ\gamma\circ \nabla)^{m}(\mathcal{S}_f)]\nonumber
\end{align}
and the exceptional set of the interior is negligible, we have 
\begin{align}
\sum \limits_{\substack{\mathcal{S}_a\in \mathcal{R}\\\mathcal{R}\subset\mathrm{Int}[(\gamma^{-1}\circ\beta\circ\gamma\circ\nabla)^{m}(\mathcal{S}_f)]\\\#\mathcal{R}=\# \mathcal{B}^m}}||\mathcal{S}_a||&<\sum \limits_{\mathcal{S}_b \in \mathcal{Z}[(\gamma^{-1}\circ\beta\circ\gamma\circ\nabla)^{m}(\mathcal{S}_f)]}||\mathcal{S}_b||\nonumber \\&=\mathcal{H}(\mathcal{S}_f^m).\nonumber
\end{align}
This completes the proof of the assertion.
\end{proof}
\bigskip

We show that all points that are not on the boundary of an expansion occupying a small enough space must necessarily be exterior points.

\begin{proposition}\label{bunchedup}
Let $f(x)=a_nx^n+a_{n-1}x^{n-1}+\cdots+a_1x+a_0$ be a polynomial of degree $n$, and let $\mathcal{S}_f$ be the tuple representation of $f$. If 
\begin{align}
\bigg |\bigg |\int \limits_{\substack{\mathcal{B}^{m}\\m<n}}f(t)dt\bigg |\bigg |<\delta,\nonumber
\end{align}
where $0<\delta <1$, then $\mathrm{Ext}[(\gamma^{-1} \circ \beta \circ \gamma \circ \nabla)^{m}(\mathcal{S}_f)]\neq \emptyset$.
\end{proposition}

\begin{proof}
Let $f(x)=a_nx^n+a_{n-1}x^{n-1}+\cdots+a_1x+a_0$ be a polynomial of degree $n$, and let $\mathcal{S}_f$ be the tuple representation of $f$. Suppose that 
\begin{align}
\bigg |\bigg |\int \limits_{\substack{\mathcal{B}^{m}\\m<n}}f(t)dt\bigg |\bigg |<\delta\nonumber
\end{align}
where $0<\delta <1$. By Theorem \ref{area}, we have 
\begin{align}
||\mathcal{S}_i-\mathcal{S}_{i+1}||<\epsilon\nonumber
\end{align}
for $\epsilon>0$ sufficiently small for all $1\leq i\leq \# \mathcal{B}^m-1$. It follows that $||\mathcal{S}_{k}||\approx ||\mathcal{S}_l||$ for all $\mathcal{S}_k,\mathcal{S}_l\in \mathcal{B}^m$. Now, we choose $\mathcal{S}_a$ such that $||\mathcal{S}_a-\mathcal{S}_i||\geq 1$ for all $1\leq i\leq \# \mathcal{B}^m$. We deduce $\mathcal{S}_a \not\in \mathcal{B}^m$. Without loss of generality, we may assume that $||\mathcal{S}_a||<||\mathcal{S}_k||$. It implies $||\mathcal{S}_a||<||\mathcal{S}_l||$. The result follows by inducting this argument on other boundary points. 
\end{proof}

\begin{theorem}\label{spaced up}
Let $f(x)=a_nx^n+a_{n-1}x^{n-1}+\cdots+a_1x+a_0$ be a polynomial of degree $n$, and let $\mathcal{S}_f$ be the tuple representation of $f$. If
\begin{align}
\bigg |\bigg |\int \limits_{\substack{\mathcal{B}^{m}\\m<n}}f(t)dt\bigg |\bigg |>\delta,\nonumber
\end{align}
for $\delta>0$, then $\mathrm{Int}[(\gamma^{-1}\circ\beta\circ\gamma\circ\nabla)^{m}(\mathcal{S})]\neq \emptyset$.
\end{theorem}

\begin{proof}
Let $f(x)=a_nx^n+a_{n-1}x^{n-1}+\cdots+a_1x+a_0$ be a polynomial of degree $n$, and let $\mathcal{S}_f$ be the tuple representation of $f$. Suppose that 
\begin{align}
\bigg |\bigg |\int \limits_{\substack{\mathcal{B}^{m}\\m<n}}f(t)dt\bigg |\bigg |>\delta\nonumber
\end{align}
for $\delta>0$ and assume to the contrary that  $\mathrm{Int}[(\gamma^{-1}\circ\beta\circ\gamma\circ\nabla)^{m}(\mathcal{S}_f)]=\emptyset$. By Theorem \ref{area}, it follows that 
\begin{align}
||\mathcal{S}_{i}-\mathcal{S}_{i+1}||>\epsilon \nonumber
\end{align}
for $\epsilon>0$ and $1\leq i \leq \# \mathcal{B}^m-1$ with \begin{align}||\mathcal{S}_{i}-\mathcal{S}_{i+1}||=\mathrm{Inf}\left \{||\mathcal{S}_i-\mathcal{S}_j||:\mathcal{S}_j\in \mathcal{B}^m\right \}.\nonumber
\end{align}
That is, the points on the boundary of expansion are mostly spaced out. Under the assumption that $\mathrm{Int}[(\gamma^{-1}\circ\beta\circ\gamma\circ\nabla)^{m}(\mathcal{S}_f)]=\emptyset$, we deduce that for any $\mathcal{S}_l\not\in \mathcal{B}^m$
\begin{align}
\mathcal{S}_l\in \mathrm{Ext}[(\gamma^{-1}\circ\beta\circ\gamma\circ\nabla)^{m}(\mathcal{S}_f)].\nonumber
\end{align}
Now, we choose $\mathcal{S}_l$ such that
\begin{align}
||\mathcal{S}_l||=\frac{1}{\# \mathcal{B}^m}\sum \limits_{\mathcal{S}_i\in \mathcal{B}^m}||\mathcal{S}_i||.\nonumber
\end{align}
We deduce that $\mathcal{S}_l\not \in \mathcal{B}^m$. Suppose that $\mathcal{S}_l\in \mathcal{B}^m$, then 
\begin{align}
\frac{1}{\#\mathcal{B}^m}\sum \limits_{\mathcal{S}_i\in \mathcal{B}^m}||\mathcal{S}_i||=||\mathcal{S}_j||\nonumber
\end{align}
for some $\mathcal{S}_j\in \mathcal{B}^m$. It follows that \begin{align}
\sum \limits_{\mathcal{S}_i\in \mathcal{B}^m}||\mathcal{S}_i||=\# \mathcal{B}^m||\mathcal{S}_j||.\nonumber
\end{align}
This violates the assumption that $||\mathcal{S}_{i}-\mathcal{S}_{i+1}||>\epsilon$ for $\epsilon>0$. Since $\mathcal{S}_l\in \mathrm{Ext}[(\gamma^{-1}\circ\beta\circ\gamma\circ\nabla)^{m}(\mathcal{S}_f)]$, we can assume (without loss of generality) that \begin{align}
\frac{1}{\#\mathcal{B}^m}\sum \limits_{\mathcal{S}_i\in \mathcal{B}^m}||\mathcal{S}_i||<||\mathcal{S}_k||\nonumber
\end{align}
for all $\mathcal{S}_k\in \mathcal{B}^m$. Choosing 
$$
||\mathcal{S}_k||=\mathrm{min}\left\{||\mathcal{S}_j||:\mathcal{S}_j\in\mathcal{B}^m\right\}
$$ 
we deduce
\begin{align}
\sum \limits_{\mathcal{S}_i\in \mathcal{B}^m}||\mathcal{S}_i||<\# \mathcal{B}^m||\mathcal{S}_k||.\nonumber
\end{align}
This inequality cannot hold.
\end{proof}

\subsection{The neighbourhood of expansion}

In this section, we introduce the concept of the \emph{neighbourhood} of an expansion. 

\begin{definition}\label{neighbourhood}
Let $\mathcal{B}^m=\mathcal{Z}[(\gamma^{-1}\circ\beta\circ\gamma\circ\nabla)^{m}(\mathcal{S}_f)]$ be the boundary of an expansion. By the \emph{neighbourhood} of $\mathcal{S}_j\in \mathcal{B}^m$ with radius $\epsilon$, denoted by $\mathcal{E}_{\epsilon}(\mathcal{S}_j)$, we mean the set 
\begin{align}
\mathcal{E}_{\epsilon}(\mathcal{S}_j):=\left \{\mathcal{S}_a:||\mathcal{S}_a-\mathcal{S}_j||<\epsilon~\mathrm{for}~\mathcal{S}_j\in \mathcal{B}^m\right \}.\nonumber
\end{align}
\end{definition}
\bigskip

We investigate the relationship between the region bounded by the boundary of an expansion and the distribution of points near the boundary.

\begin{proposition}
Let $f(x):=a_nx^n+a_{n-1}x^{n-1}+\cdots+a_1x+a_0\in \mathbb{R}[x]$ be a polynomial of degree $n$ and suppose that 
\begin{align}
\bigg |\bigg |\int \limits_{\substack{\mathcal{B}^m\\m<n}}f(t)dt\bigg |\bigg |<\delta\nonumber
\end{align}
for $\delta>0$ sufficiently small. We have 
\begin{align}
\mathcal{E}_{1}(\mathcal{S}_j)\bigcap\mathcal{E}_{\frac{1}{2}}(\mathcal{S}_{j+1})\neq\emptyset \nonumber
\end{align}
for $\mathcal{S}_j,\mathcal{S}_{j+1}\in \mathcal{B}^m$.
\end{proposition}

\begin{proof}
Let $f(x):=a_nx^n+a_{n-1}x^{n-1}+\cdots+a_1x+a_0\in \mathbb{R}[x]$ be a polynomial of degree $n$ and suppose that 
\begin{align}
\bigg |\bigg |\int \limits_{\substack{\mathcal{B}^m\\m<n}}f(t)dt\bigg |\bigg |<\delta\nonumber
\end{align}
for $\delta>0$ sufficiently small. By Theorem \ref{area}, we have
\begin{align}
||\mathcal{S}_i-\mathcal{S}_{i+1}||<\epsilon \nonumber
\end{align}
for $\epsilon>0$ sufficiently small for $\mathcal{S}_i, \mathcal{S}_{i+1}\in \mathcal{B}^m$ with $1\leq i\leq \# \mathcal{B}^m-1$. It follows that any two boundary points are sufficiently close to each other. The result follows from this fact. Suppose that $\mathcal{E}_{1}(\mathcal{S}_j)\bigcap\mathcal{E}_{\frac{1}{2}}(\mathcal{S}_{j+1})=\emptyset $, then for all $\mathcal{S}_a\in \mathcal{E}_{\frac{1}{2}}(\mathcal{S}_{j+1})$ we deduce $\mathcal{S}_a\not \in \mathcal{E}_{1}(\mathcal{S}_{j})$. This implies that $||\mathcal{S}_j-\mathcal{S}_a||\geq 1$. We deduce \begin{align}
\epsilon +||\mathcal{S}_{j+1}-\mathcal{S}_a||\geq 1\nonumber
\end{align}
and it follows that $\frac{1}{2}>||\mathcal{S}_{j+1}-\mathcal{S}_a||\geq 1-\epsilon$. This inequality cannot hold since $\epsilon>0$ is sufficiently small.
\end{proof}
\bigskip

Here, we prove that sub-expansions with an identical mass of an expansion occupying a small enough space by their boundaries should, in principle, have a substantially bigger mass on average compared to their mother expansion.

\begin{theorem}
Let $f(x)=a_nx^n+a_{n-1}x^{n-1}+\cdots+a_1x+a_0\in \mathbb{R}[x]$ be a polynomial of degree $n$. Let $(\gamma^{-1}\circ\beta\circ\nabla)^{m_2}(\mathcal{S}_f)<(\gamma^{-1}\circ\beta\circ\nabla)^{m_1}(\mathcal{S}_f)$ with $\mathcal{H}(\mathcal{S}_f^{m_1})\approx \mathcal{H}(\mathcal{S}_f^{m_2})$. If 
\begin{align}
\bigg|\bigg|\int \limits_{\substack{\mathcal{B}^{m_1}\\m_1<n}}f(t)dt\bigg|\bigg|<\bigg|\bigg|\int \limits_{\substack{\mathcal{B}^{m_2}\\m_2<n}}f(t)dt\bigg|\bigg|\nonumber
\end{align}
then 
\begin{align}
\frac{1}{\#\mathcal{B}^{m_2}}\sum \limits_{\mathcal{S}_a\in \mathcal{B}^{m_2}}||\mathcal{S}_a||>||\mathcal{S}_b||\nonumber
\end{align}
for some $\mathcal{S}_b\in \mathcal{B}^{m_1}$.
\end{theorem}

\begin{proof}
Let $f(x)=a_nx^n+a_{n-1}x^{n-1}+\cdots+a_1x+a_0\in \mathbb{R}[x]$ be a polynomial of degree $n$ and suppose that $(\gamma^{-1}\circ\beta\circ\nabla)^{m_2}(\mathcal{S}_f)<(\gamma^{-1}\circ\beta\circ\nabla)^{m_1}(\mathcal{S}_f)$. It follows that $m_2=m_1+j$ for $j\geq 1$. That is, $m_2>m_1$. Since \begin{align}
\bigg|\bigg|\int \limits_{\substack{\mathcal{B}^{m_1}\\m_1<n}}f(t)dt\bigg|\bigg|<\bigg|\bigg|\int \limits_{\substack{\mathcal{B}^{m_2}\\m_2<n}}f(t)dt\bigg|\bigg|\nonumber
\end{align}
we get $||\mathcal{S}_i||\approx ||\mathcal{S}_{i+1}||$ for $1\leq i\leq \# \mathcal{B}^{m_1}-1$ and $||\mathcal{S}_j||\approx ||\mathcal{S}_{j+1}||$ for all $1\leq j \leq \# \mathcal{B}^{m_2}-1$. Suppose, on the contrary, that 
\begin{align}
\frac{1}{\#\mathcal{B}^{m_2}}\sum \limits_{\mathcal{S}_a\in \mathcal{B}^{m_2}}||\mathcal{S}_a||\leq ||\mathcal{S}_b||\nonumber
\end{align}
for all $\mathcal{S}_b\in \mathcal{B}^{m_1}$, then it implies \begin{align}
\frac{1}{\# \mathcal{B}^{m_2}}\sum \limits_{\mathcal{S}_a\in \mathcal{B}^{m_2}}||\mathcal{S}_a||\leq \frac{1}{\# \mathcal{B}^{m_1}}\sum \limits_{\mathcal{S}_b\in \mathcal{B}^{m_1}}||\mathcal{S}_b||\nonumber
\end{align}
if and only 
\begin{align}
\frac{\mathcal{H}(\mathcal{S}_f^{m_2})}{\# \mathcal{B}^{m_2}}\leq \frac{\mathcal{H}(\mathcal{S}_f^{m_1})}{\# \mathcal{B}^{m_1}}.\nonumber
\end{align}
This inequality cannot hold since $\mathcal{H}(\mathcal{S}_f^{m_2})\approx \mathcal{H}(\mathcal{S}_f^{m_1})$ and boundary points experience a sizable drop with expansion. This completes the proof.
\end{proof}
\bigskip

We can use these tools to study some statistics of polynomials.

\begin{theorem}
Let $f(x)=a_nx^n+a_{n-1}x^{n-1}+\cdots+a_1x+a_0$ and $g(x)=b_kx^k+b_{k-1}x^{k-1}+\cdots+b_1x+b_0$ be polynomials of degree $n$ and $k$, respectively. Let $\mathcal{H}(\mathcal{S}_f^{m})-\mathcal{H}(\mathcal{S}_g^{m})>1+\epsilon$ for any $\epsilon>0$ and $\mathcal{B}_g^{m}\cap \mathcal{B}_f^{m}\neq \emptyset$. If 
\begin{align}
\bigg|\bigg|\int \limits_{\substack{\mathcal{B}_f^{m}\\m<n}}f(t)dt\bigg|\bigg|<\delta <1 \quad \mathrm{and} \quad \bigg|\bigg|\int \limits_{\substack{\mathcal{B}_g^{m}\\m<k}}f(t)dt\bigg|\bigg|<\delta <1 \nonumber
\end{align}
then $\mathrm{deg}(f)>\mathrm{deg}(g)$.
\end{theorem}

\begin{proof}
Let $f(x)=a_nx^n+a_{n-1}x^{n-1}+\cdots+a_1x+a_0$ and $g(x)=b_kx^k+b_{k-1}x^{k-1}+\cdots+b_1x+b_0$ be polynomials of degree $n$ and $k$, respectively. Suppose that 
\begin{align}
\bigg|\bigg|\int \limits_{\substack{\mathcal{B}_f^{m}\\m<n}}f(t)dt\bigg|\bigg|<\delta <1 \quad \mathrm{and} \quad \bigg|\bigg|\int \limits_{\substack{\mathcal{B}_g^{m}\\m<k}}f(t)dt\bigg|\bigg|<\delta <1 \nonumber
\end{align}
then $||\mathcal{S}_i||\approx ||\mathcal{S}_{i+1}||$ for all $1\leq i\leq \#\mathcal{B}_f^m-1$ and $||\mathcal{S}_j||\approx ||\mathcal{S}_{j+1}||$ for all $1\leq j\leq \#\mathcal{B}_g^m-1$. Since $\mathcal{B}_g^{m}\cap \mathcal{B}_f^{m}\neq \emptyset$, it follows that $||\mathcal{S}_i||\approx ||\mathcal{S}_{i+1}||\approx ||\mathcal{S}_j||\approx ||\mathcal{S}_{j+1}||$ for all $1\leq i\leq \#\mathcal{B}_f^m-1$ and $1\leq j \leq \#\mathcal{B}_g^m-1$. Since $\mathcal{H}(\mathcal{S}_f^{m})-\mathcal{H}(\mathcal{S}_g^{m})>1+\epsilon$ for any $\epsilon>0$, it follows that $\#\mathcal{B}_f^m>\#\mathcal{B}_g^m$. We deduce $n>k$. This proves the claim.
\end{proof}

\subsection{Rotation of the boundary of expansion}

In this section, we introduce the concept of \emph{rotation} of the boundary of an expansion. 

\begin{definition}\label{Rotate}
Let $(\gamma^{-1}\circ\beta\circ\gamma\circ\nabla)^{m}(\mathcal{S})$ be an expansion with boundary $\mathcal{B}^m$. We say that the map $\Lambda$ is a \emph{rotation} of the boundary $\mathcal{B}^m$ if 
\begin{align}
\Lambda :\mathcal{B}^m\longrightarrow \mathcal{B}^m.\nonumber
\end{align} 
We say that an expansion admits a \emph{rotation} if there exists such a map. In other words, we say that the map $\Lambda$ induces a rotation on the expansion. We say that the boundary is \emph{stable} under rotation if $||\Lambda(\mathcal{S}_a)||\approx ||\mathcal{S}_a||$ for all $\mathcal{S}_a\in \mathcal{B}^m$. Otherwise, we say that it is \emph{unstable} under the rotation.
\end{definition}

\begin{proposition}\label{rotate2}
Let $f(x):=a_nx^n+\cdots +a_1x+a_0\in \mathbb{R}[x]$ be a polynomial of degree $n\geq 3$. Let $(\gamma^{-1}\circ\beta\circ\gamma\circ\nabla)^{m}(\mathcal{S}_f)$ be an expansion with boundary $\mathcal{B}^m$ that admits the rotation $\Lambda$. If 
\begin{align}
\bigg |\bigg |\int \limits_{\substack{\mathcal{B}^m\\m<n}}f(t)dt\bigg |\bigg |<1\nonumber
\end{align}
then the boundary $\mathcal{B}^m$ is stable.
\end{proposition}

\begin{proof}
Let $f(x):=a_nx^n+\cdots +a_1x+a_0\in \mathbb{R}[x]$ be a polynomial of degree $n\geq 3$ and suppose that $(\gamma^{-1}\circ\beta\circ\gamma\circ\nabla)^{m}(\mathcal{S}_f)$ is an expansion with boundary $\mathcal{B}^m$ admitting the rotation $\Lambda$. Suppose also that 
\begin{align}
\bigg |\bigg |\int \limits_{\substack{\mathcal{B}^m\\m<n}}f(t)dt\bigg |\bigg |<1.\nonumber
\end{align}
By Theorem \ref{area}, we get  $||\mathcal{S}_j||\approx ||\mathcal{S}_{j+1}||$ for $1\leq j \leq \# \mathcal{B}^m-1$ with $\mathcal{S}_j, \mathcal{S}_{j+1}\in \mathcal{B}^m$. It follows that for the rotation $\Lambda:\mathcal{B}^m\longrightarrow \mathcal{B}^m$, we have that for any $\mathcal{S}_j\in \mathcal{B}^m$, then 
\begin{align}
\Lambda(\mathcal{S}_j)=\mathcal{S}_k\nonumber
\end{align}
for some $\mathcal{S}_k\in \mathcal{B}^m$. We deduce $||\Lambda(\mathcal{S}_j)||=||\mathcal{S}_k||\approx ||\mathcal{S}_j||$. This proves the stability.
\end{proof}

\section{Simple and compact expansions}\label{sec:simple and compact expansion}

In this section, we study a particular type of expansion. The main tool in the classification of these types of expansion is the concept of rotation of the boundary of an expansion. 

\begin{definition}
 Let $(\gamma^{-1}\circ\beta\circ\gamma\circ\nabla)^{m+1}(\mathcal{S})$ be an expansion with boundary $\mathcal{B}^{m+1}$. We say that the expansion is \emph{simple} if any rotation of $\mathcal{B}^{m+1}$ given by   $\Lambda:\mathcal{B}^{m+1}\longrightarrow \mathcal{B}^{m+1}$ is not a rotation of $\mathcal{B}^m$.
\end{definition}

\begin{definition}\label{compact}
Let $(\gamma^{-1}\circ\beta\circ\gamma\circ\nabla)^{m}(\mathcal{S})$ be an expansion and $\epsilon>0$ be small. We say that the expansion is \emph{compact} if there exist some $\mathcal{S}_l\in \mathcal{B}^{m+1}$ such that $\mathcal{S}_l\in \mathcal{E}_{\epsilon}(\mathcal{S}_j)$ for each $\mathcal{S}_j\in \mathcal{B}^{m}$ for all $1\leq m \leq \mathrm{deg}(\mathcal{S}_j)-1$. 
\end{definition}
\bigskip

Here, we prove that the mass of an expansion diminishes uniformly for compact expansions. In other words, the Sendov conjecture is true for compact expansions.

\begin{theorem}
If $(\gamma^{-1}\circ\beta\circ\gamma\circ\nabla)^{m}(\mathcal{S})$ is a compact expansion, then $\mathcal{H}(\mathcal{S}^{m+1})<\mathcal{H}(\mathcal{S}^{m})$ for all $1\leq m \leq \mathrm{deg}(\mathcal{S})-1$.
\end{theorem}

\begin{proof}
Suppose that $(\gamma^{-1}\circ\beta\circ\gamma\circ\nabla)^{m}(\mathcal{S})$ is a compact expansion. By definition \ref{compact}, we deduce that for some small $\epsilon>0$ there exist some $\mathcal{S}_l\in \mathcal{B}^{m+1}$ such that $\mathcal{S}_l\in \mathcal{E}_{\epsilon}(\mathcal{S}_j)$. It implies that 
\begin{align}
||\mathcal{S}_l-\mathcal{S}_j||<\epsilon\nonumber
\end{align}
for some small $\epsilon>0$ for all $\mathcal{S}_j\in \mathcal{B}^m$. Since the boundary points decrease with expansions, the claim follows immediately.
\end{proof}

\section{Further remarks and future works}

In this paper, we put a premium on inverse problems; in particular, inverse problems for higher--extremely higher--phase expansions, even though understanding higher phase inverse problems requires understanding the higher phase expansions. Simply put, we would require a nice formula that represents the $n$ copies of the recovery map $\Delta\circ\gamma^{-1}\circ\beta^{-1}\circ\gamma$. In other words, we write 
\begin{align}
(\Delta\circ\gamma^{-1}\circ\beta^{-1}\circ\gamma)^{n}=F\circ (\Delta\circ\gamma^{-1}\circ\beta^{-1}\circ \gamma)\nonumber
\end{align}
where $F$ is some smooth map depending on $n$? The theory developed is still open for further development. One area that remains underexplored is the study of the theory in the case the polynomial has several indeterminate, which one may consider as several variables \emph{expansivity theory}.

\bibliographystyle{amsplain}

\end{document}